\documentclass[11pt,reqno,twoside]{amsart}

 \pdfoutput=1

\usepackage{amsfonts,amsmath,amssymb,amsthm}
\usepackage{mathrsfs,mathtools}
\usepackage{enumerate,enumitem}
\usepackage{xspace} 
\usepackage{url}

\usepackage{graphicx}
\usepackage{xcolor}
\usepackage{epstopdf,epsfig,subfigure}

\usepackage{verbatim}

\usepackage{hyperref}

\theoremstyle{plain}
\newtheorem{theorem}{Theorem}[section]

\newtheorem{lemma}[theorem]{Lemma}
\newtheorem*{mainresult*}{Main Result}
\newtheorem{assumption}[theorem]{Assumption}

\theoremstyle{definition}

\newtheorem{remark}[theorem]{Remark}

\numberwithin{equation}{section}

\theoremstyle{plain}

\newcommand{\linspan}{\mathop{\rm span}\nolimits}

\newcommand{\rest}{\left.\kern-2\nulldelimiterspace\right|_}
\newcommand{\norm}[2]{\left|#1\right|_{#2}}

\newcommand{\Id}{{\mathbf1}}
\newcommand{\indf}{1}

\newcommand{\ex}{\mathrm{e}}
\newcommand{\p}{\partial}
\newcommand{\ed}{\mathrm d}

\newcommand*{\Bigcdot}{\raisebox{-.25ex}{\scalebox{1.25}{$\cdot$}}}

\newcommand{\clB}{{\mathcal B}}
\newcommand{\clC}{{\mathcal C}}

\newcommand{\clE}{{\mathcal E}}

\newcommand{\clG}{{\mathcal G}}

\newcommand{\clK}{{\mathcal K}}
\newcommand{\clL}{{\mathcal L}}

\newcommand{\clN}{{\mathcal N}}

\newcommand{\clS}{{\mathcal S}}
\newcommand{\clT}{{\mathcal T}}
\newcommand{\clU}{{\mathcal U}}

\newcommand{\clW}{{\mathcal W}}

\newcommand{\clZ}{{\mathcal Z}}

\newcommand{\bbE}{{\mathbb E}}

\newcommand{\bbH}{{\mathbb H}}

\newcommand{\bbM}{{\mathbb M}}
\newcommand{\bbN}{{\mathbb N}}

\newcommand{\bbR}{{\mathbb R}}
\newcommand{\bbS}{{\mathbb S}}

\newcommand{\bfC}{{\mathbf C}}

\newcommand{\fkE}{{\mathfrak E}}

\newcommand{\fkL}{{\mathfrak L}}

\newcommand{\rmD}{{\mathrm D}}

\newcommand{\bfn}{{\mathbf n}}

\newcommand{\bfu}{{\mathbf u}}

\newcommand{\rmd}{{\mathrm d}}
\newcommand{\rme}{{\mathrm e}}

\newcommand{\fkm}{{\mathfrak m}}
\newcommand{\fkn}{{\mathfrak n}}

\newcommand{\fks}{{\mathfrak s}}

\newcommand{\ovlineC}[1]{\overline C_{\left[#1\right]}}

\definecolor{DarkBlue}{rgb}{0,0.08,0.45}
\definecolor{DarkRed}{rgb}{.65,0,0}
\definecolor{applegreen}{rgb}{0.55, 0.71, 0.0}

\newcounter{mymac@matlab}
\newcommand{\matlab}{MATLAB%
   \ifnum\value{mymac@matlab}<1%
   \textregistered%
   \setcounter{mymac@matlab}{1}%
   \fi%
  }

\newcommand{\black}{ \color{black} } 
\newcommand{\blue}{ \color{blue} }

\begin{document}
\title{Existence, uniqueness, and stabilization results for parabolic variational inequalities}
\author{Axel Kr\"{o}ner$^1$}
\author{Carlos N. Rautenberg$^2$}
\author{S\'ergio S.~Rodrigues$^3$}
\thanks{\newline\noindent$^1$ Weierstrass Institute
for Applied Analysis and Stochastics
Berlin, Germany,  ({\small\tt axel.kroener@wias-berlin.de}).
\newline $^2$ George Mason University, Fairfax, VA, USA,
({\small\tt crautenb@gmu.edu}).\qquad  {C. N. R.} was supported by NSF grant DMS-2012391, and  acknowledges the support of Germany's Excellence Strategy - The Berlin Mathematics Research Center MATH+ (EXC-2046/1, project ID: 390685689) within project AA4-3.
\newline $^3$ Karl-Franzens University of Graz, Austria,
 ({\small\tt sergio.rodrigues@ricam.oeaw.ac.at}).\qquad
{S. S. R.} was supported by the ERC advanced grant 668998 (OCLOC) under the EU's H2020 research program, and acknowledges partial support from Austrian Science Fund (FWF): P 33432-NBL.}

\begin{abstract}  In this paper we consider feedback stabilization for parabolic variational
inequalities of obstacle type with time and space depending reaction and 
convection coefficients and show exponential stabilization to nonstationary trajectories.
Based on a Moreau--Yosida approximation, a feedback operator is established
using a finite (and uniform in the approximation index) number of actuators leading to exponential decay of given
rate of the state variable.
Several numerical examples are presented addressing smooth and nonsmooth obstacle functions. 
\end{abstract}



\maketitle

\pagestyle{myheadings} \thispagestyle{plain} \markboth{\sc A. Kr\"{o}ner, C. N. Rautenberg, and S. S. Rodrigues}
{\sc Stabilization for nonautonomous parabolic variational inequalities}

\section{Introduction}
Our goal is the stabilization to trajectories for parabolic variational inequalities, in particular towards the solution~$ y$ to the obstacle problem
\begin{subequations}\label{sys-haty}
 \begin{align}
&\langle \tfrac{\p}{\p t}{ y} +(-\Delta+\Id) y+a y+b\cdot\nabla y -f,v- y\rangle\geq 0,
\quad \forall v\leq \psi,\;\; t>0,\\
& y\leq \psi,\quad \clG y\rest{\Gamma}=\chi,\quad	t>0,\qquad y(\Bigcdot,0)= y_\circ,
\end{align}
\end{subequations}
 in a bounded domain~$\Omega\subset\bbR^d$ with a regular enough boundary $\Gamma:=\partial\Omega$,
 where~$d$ is a positive integer. The obstacle~$\psi=\psi(x,t)$ and the functions~$a=a(x,t)\in\bbR$,
 $b=b(x,t)\in\bbR^d$, $f=f(x,t)\in\bbR$, $\chi=\chi(\overline x,t)\in\bbR$, $v=v(x,t)\in\bbR$,
and~$y_\circ=y(x)$, are assumed to be sufficiently regular, for~$(x,\overline x,t)\in \Omega\times\Gamma\times(0,+\infty)$;
regularity details are specified later. The linear operator~$\clG$ is determined by either Dirichlet or Neumann boundary conditions.

 For some pairs $(a,b)$,  the solution~$ w$ issued from a different initial condition~$ w_\circ\ne y_\circ$
\begin{subequations}\label{sys-tildey}
 \begin{align}
&\langle \tfrac{\p}{\p t}{ w}+ (-\Delta+\Id)w+a w+b\cdot\nabla w -f,v- w\rangle\geq 0,
\quad \forall v\leq \psi,\;\; t>0,\\
& w\leq \psi,\quad \clG w\rest{\Gamma}=\chi,\quad t>0,\qquad	 w(\Bigcdot,0)= w_\circ,
\end{align}
\end{subequations}
may not converge to~$y$ as time increases.   Our goal is  to show that, by
means of an feedback control input~$\bfu=\clK(w-y)$, we can track~$y$ exponentially
fast with an arbitrary exponential rate~$-\mu<0$. That is, we want to construct an input feedback operator~$\clK$ such that the solution of
\begin{subequations}\label{sys-tildey-u}
 \begin{align}
&\langle \tfrac{\p}{\p t}{ w}+ (-\Delta+\Id)w+a w+b\cdot\nabla w-f-\clK(w-y),v- w\rangle\geq 0,
\quad \forall v\leq \psi,\quad t>0,\\
&w\leq \psi,\quad \clG w\rest{\Gamma}=\chi,\quad  t>0,\qquad w(\Bigcdot,0)= w_\circ,
\end{align}
\end{subequations}
satisfies, for a suitable constant~$C\ge1$,
\begin{equation}\label{goal-intro}
 \left| w(t)- y(t)\right|_{L^2(\Omega)}\le C{\rm e}^{-\mu t}\left| w_\circ- y_\circ\right|_{L^2(\Omega)},
 \quad\mbox{for all}\quad(w_\circ,y_\circ)\in L^2(\Omega)\times L^2(\Omega),\quad t\ge0.
\end{equation}

We are interested
in the case~ $\clK\colon L^2(\Omega) \to \clU_M$,
where~$\clU_M\subset L^2(\Omega)$ is a finite-dimensional subspace,
given by the linear span of a finite set of
actuators~$U_M=\{\Psi_i\mid 1\le i\le \fkm(M)\}\subset L^2(\Omega)$,
where~$\fkm(M)$ is a positive integer which will be appropriately chosen later on.
It follows that the control input 
will be of the form
\begin{align}\notag
\bfu(t)=\clK(w(t)-y(t))=\sum\limits_{i=1}^{M_\fkm}u_i(t)\Psi_i\in\clU_M.
\end{align}
Further, motivated by real applications, we consider the case in which the actuators
are determined by indicator functions~$\indf_{\omega_i}$ of small subdomains~$\omega_i\subset\Omega$,
\[
\Psi_i(x)=\indf_{\omega_i}(x)=\begin{cases}1,&\mbox{ if }x\in\omega_i,\\
                          0,&\mbox{ if }x\in\Omega\setminus\omega_i,
                         \end{cases} \qquad 1\le i\le M_\fkm.
\]

\begin{remark}\label{R:nu}
 Note that for simplicity we have taken the diffusion operator as~$-\Delta+\Id$. One reason is to facilitate the
inclusion of Neumann boundary conditions
in our investigation  where, in particular, we ask the operator to be injective.
This is not a significant restriction, since we can always transform a given dynamics
$\frac{\p}{\p t}y-\nu\Delta y+\widetilde ay+h=0$ into~$\frac{\p}{\p \tau}z+(-\Delta +\Id)z+(\nu^{-1}\widetilde a-1)z+\nu^{-1}h=0$
simply by rescaling time, $\tau=\nu t$, $z(\tau)=y(\nu^{-1}\tau)$.
\end{remark}

\subsection{Main stabilizability result}

Recall that for  Dirichlet and Neumann boundary conditions, the operator~$\clG$
reads, respectively,
\[
 \clG=\Id\qquad\mbox{and}\qquad\clG=\tfrac{\p}{\p\bfn}=\bfn\cdot\nabla,
\]
where~$\bfn=\bfn(\overline x)$ is the unit outward normal vector to~$\Gamma$ at~$\overline x\in\Gamma$.
In either case we set~$L^2(\Omega)$ as a pivot space, that is, we identify~$L^2(\Omega)$ with its own dual,~$L^2(\Omega)'=L^2(\Omega)$.

Depending on the choice of $\clG$, we define the spaces
\[
V\coloneqq\begin{cases}
       H^1_0(\Omega),&\;\mbox{if}\quad \clG=\Id,\\
       H^1(\Omega),&\;\mbox{if}\quad \clG=\tfrac{\p}{\p\bfn},
 \end{cases}
\]
and the symmetric isomorphism 
\begin{equation}\label{defA}
A\colon V\to V',\qquad \langle Ay,z\rangle_{V',V}\coloneqq (\nabla y,\nabla z)_{L^2(\Omega)^d}+(y,z)_{L^2(\Omega)}.
\end{equation}
Throughout the paper, we assume that the subset~$\Omega$ is bounded, open, and connected, located on one side of its
boundary~$\Gamma=\p\Omega$. Furthermore, either~$\Gamma$ is a compact~$C^2$-manifold or~$\Omega$ is a convex polygonal domain.
The domain of~$A$ is defined as~$\rmD(A)\coloneqq \{z\in L^2(\Omega)\mid A z\in L^2(\Omega)\}$,
and since~$\Omega$ is regular enough,\black
we have the following characterizations
\begin{align}\label{char-dom}
  \rmD(A)=\{z\in H^2(\Omega)\mid \clG z\rest\Gamma=0\}.
\end{align}
It also follows that~$A$ has a compact inverse, and that~$L^2(\Omega)=\rmD(A^0)$ and~$V=\rmD(A^\frac12)$. Note that
$
A\coloneqq(-\Delta+\Id)\rest{\rmD(A)}\colon \rmD(A)\to L^2(\Omega),
$
is the restriction of~$-\Delta+\Id$ to~$\rmD(A)$.

We shall assume  that~$V$\black and~$\rmD(A)$ are endowed, respectively, with the scalar products
\[
(y,z)_V\coloneqq \langle Ay,z\rangle_{V',V}\quad\mbox{and}\quad (y,z)_{\rmD(A)}\coloneqq (Ay,Az)_{L^2(\Omega)}
\]
and associated norms. Note that~$(y,z)_V=(y,z)_{H^1(\Omega)}$ coincides with the usual scalar product of~$H^1(\Omega)$. 
Finally, we denote the increasing sequence of eigenvalues of~$A$ by~$(\alpha_i)_{i\in\bbN}$, and a
complete basis of eigenfunctions by~$(e_i)_{i\in\bbN}$,
\[
Ae_i=\alpha_ie_i,\qquad e_i\in\rmD(A),\qquad 0<\alpha_i\le\alpha_{i+1}\to+\infty.
\]

Throughout this manuscript, for simplicity, we shall denote the Hilbert Sobolev spaces
\[
H^s\coloneqq H^s(\Omega)=W^{s,2}(\Omega)\quad\mbox{for}\quad s>0,\quad\mbox{and}\quad {L^2}\coloneqq
L^2(\Omega). 
\]

We consider sequences of sets of actuators and eigenfunctions  $E_M$ of the diffusion operator under homogeneous boundary conditions as follows, for some nondecreasing
function $\mathfrak{m}:\mathbb{N}\to\mathbb{N}$
\begin{subequations}\label{seqActEig}
 \begin{align}
 &(U_M)_{M\in \bbN},\quad U_M=\{\Psi_i\mid 1\le i\le \fkm(M)\}\subset {L^2(\Omega)},\\
 &(E_M)_{M\in \bbN},\quad E_M=\{e_i\mid i\in\bbE_M\}\subset \rmD(A)\subset {L^2(\Omega)},
 \quad\bbE_M=\{j_k^M\mid 1\le k\le \fkm(M\}\subset\bbN,
 \end{align}
where~$\bbN$ stands for the set of positive integers and the $j_k^M$s are specified later. Further, we denote 
 \begin{equation}
 \clU_M=\linspan U_M,\qquad \clE_M=\linspan E_M,
 \end{equation}
 and assume that
 \begin{equation}
 \dim\clU_M=M_\fkm=\dim\clE_M,\quad {L^2(\Omega)}=\clU_M+\clE_M^\perp,\quad\mbox{and}\quad\clU_M\:{ \textstyle\bigcap}\:\clE_M^\perp=\{0\}.\label{seqActEig.DS}
 \end{equation} 
 \end{subequations}
Due to~\eqref{seqActEig.DS}, the
oblique projection~$P_{\clU_M}^{\clE_M^\perp}$, in~${L^2(\Omega)}$ onto~$\clU_M$ along~$\clE_M^\perp$, is well defined as follows: we can write an arbitrary $h\in L^2$ in a unique way as $h=h_{\clU_M}+h_{\clE_M^\perp}$ with~$(h_{\clU_M},h_{\clE_M^\perp})\in\clU_M\times\clE_M^\perp$, then we set~$P_{\clU_M}^{\clE_M^\perp}h\coloneqq h_{\clU_M}$.

Our results will follow under general conditions on the dynamics tuple~$(a,b,f,\chi,\psi)$  and under a particular condition on
the sequence~$(\clU_M,\clE_M)_{M\in \bbN}$. Such conditions will be presented and  specified later on.
Without entering into more details at this point our main result  is the following, whose precise statement shall be given in Theorem \ref{T:main}.\black
\begin{mainresult*}
{ Let $r=r(t)\coloneqq\min(t,1)$ for  $t\ge 0$. Under sufficient regularity of the data
and some assumptions which will be specified in Section \ref{sec:assum} we have the following:
\par
{\rm (i)}  For every~$T>0$, there exists a unique solution $y\in W((0,T);H^1,V')$ of~\eqref{sys-haty} with~$ry\in W((0,T);H^2,{L^2})$. 
\par
{\rm (ii)}  For every~$\mu>0$, there are $M$ and~$\lambda$ large enough such that,  with~$ \clK_M^\lambda\coloneqq \lambda P_{\clU_{M}}^{\clE_{M}^\perp}A P_{\clE_{M}}^{\clU_{M}^\perp}$,
the solution of the system
\begin{subequations}\label{sys-tildey-K-intro}
 \begin{align}
&\left( \tfrac{\p}{\p t}{ w} +(-\nu\Delta+\Id) w+a w+b\cdot\nabla w - f
+ \clK_M^\lambda(w-y),v- w\right)_{L^2}\geq 0,
\quad \forall v\leq \psi,\quad t>0,\\
&w\leq \psi,\quad
w(0)= w_\circ,\quad \clG w\rest{\Gamma}=\chi.
\end{align}
\end{subequations}
satisfies the inequality~\eqref{goal-intro} with $C=1$.
Furthermore, 
\begin{subequations}\label{prop.K-intro}
\begin{align}
\norm{\clK_M^\lambda}{\clL(L^2)}
&\le\lambda\widehat\alpha_M\norm{P_{\clU_{M}}^{\clE_{M}^\perp}}{\clL(L^2)}^2\quad\mbox{and}\\
 \norm{\clK_M^\lambda({w} -{y})}{L^2(\bbR_+,L^2)}
& \le \lambda\widehat\alpha_M\mu^{-1} \norm{P_{\clU_{M}}^{\clE_{M}^\perp}}{\clL(L^2)}^2
\norm{{w}_\circ-{y}_\circ }{L^2},
 \end{align}
\end{subequations}
 where~$\widehat\alpha_M=\sup\{\alpha_i\mid e_i\in E_M\mbox{ and }Ae_i=\alpha_ie_i\}$.
}
\end{mainresult*}

\subsection{Previous literature}
The use of oblique projections has been introduced in Kunisch and Rodrigues~\cite{KunRod19-cocv},
in the  construction of explicit feedback operators for stabilization of linear
parabolic-like systems under homogeneous conditions $(f,\chi)=0$. Precisely, the feedback in~\cite{KunRod19-cocv} is given by
\begin{align}\label{FeedKunRod}  
 \clK_{M}(t)(y)=P_{\clU_{M}}^{\clE_{M}^\perp}
 \Bigl(A+A_{\rm rc}(t)-\lambda \Id\Bigr)y,
 \end{align}
where~$\clU_M$ is the finite-dimensional actuators space and the auxiliary space~$\clE_{M}$
is spanned by a suitable set of eigenfunctions of the diffusion-like operator~$A$.
Further~$A_{\rm rc}$ is a reaction-convection-like operator.
Appropriate variations of such feedback are used in Kunisch and Rodrigues~\cite{KunRod19-dcds}
to stabilize coupled parabolic-{\sc ode} systems, and
in Azmi and Rodrigues~\cite{AzmiRod20}  to stabilize damped wave equations.
In Rodrigues~\cite{Rod20-eect}, the analogous feedback
\begin{align}\label{FeedRod-eect}  
 \clK_{M}(t)(y)=P_{\clU_{M}}^{\clE_{M}^\perp}
 \Bigl(Ay+A_{\rm rc}(t)y+\clN(t,y)-\lambda y\Bigr),
 \end{align}
is used to semiglobally stabilize  parabolic equations, where the dynamics
includes a given nonlinear term~$\clN(t,\Bigcdot)$ and the number of actuators
is large enough, depending on the norm~$\norm{y_0}{V}$ of the initial state in
a suitable Hilbert space $V\subseteq L^2$.

In this paper we investigate the stabilizability of nonautonomous parabolic variational
inequalities through a limiting argument based on Moreau--Yosida approximations.
The latter are semilinear parabolic equations and by this reason we could try to
use the feedback~\eqref{FeedRod-eect}. However,  the number of actuators required by that feedback
increases (or may increase) with the norm of the nonlinear term, that is, the number of actuators is expected to
increase with the Moreau--Yosida parameter. Roughly speaking, the number of needed actuators could diverge
to~$+\infty$ as the Moreau--Yosida parameter does. This would mean that, even in the case we can find a
limit feedback operator, that operator could have an infinite-dimensional range, that is, we would need
an infinite number of actuators to be able to implement the controller. This is of course unfeasible for real world applications.
Therefore, we will use a different feedback operator in~\eqref{sys-tildey-K-intro}, namely, 
\begin{align}\label{FeedK}  
 \clK_{M}^\lambda=-\lambda P_{\clU_{M}}^{\clE_{M}^\perp}AP_{\clE_{M}}^{\clU_{M}^\perp}.
 \end{align}

We shall make use of the monotonicity of the nonlinear term associated with the Moreau--Yosida approximation. Without such monotonicity we do not know whether the feedback in~\eqref{FeedK} is able to stabilize parabolic systems for a general class of nonlinearities as in~\cite{Rod20-eect}. Moreover, it is also such monotonicity which will allow us to take the pair~$(\lambda,M)$ in~\eqref{FeedK} independently of the Moreau--Yosida parameter,
and this is why we will be able to take such feedback in the limit variational inequality.

This manuscript introduces the use of oblique projections in the construction of explicit feedback operators which are able to stabilize parabolic variational inequalities. Moreover, to the best knowledge of the authors, there are no results on  stabilization of parabolic variational inequalities available in the literature.
In spite of this fact we would like to refer the reader to previous works on controlled parabolic variational inequalities defined on a \textit{bounded} time interval. 

Feedback laws for optimal control of parabolic variational inequalities have been addressed in Popa \cite{MR1816854} and robust feedback 
laws in Maksimov \cite{MR2155305}. In the first reference the author shows that for a certain class of parabolic variational inequalities the optimal control is
given by a feedback law given by the optimal value function. In the latter reference the author considers a robust control problem for a parabolic variational inequality in the case of distributed control actions and disturbances, and establishes a feedback law using piecewise (in time) constant control functions being irrespective of the unknown effective perturbation.

For stabilization we are often interested in closed-loop (feedback) controls. However, we would like to refer the reader to several contributions concerning open-loop optimal control of parabolic variational inequalities (still, in a bounded time interval).  Wang~\cite{MR1866305} considers optimal control problems for systems governed by a parabolic variational inequality coupled with a semilinear parabolic differential equation, Ito and Kunisch~\cite{MR2609034} consider strong and weak solution concepts for parabolic variational inequalities and study existence. Furthermore the first order optimality system in a Lagrangian framework is derived.
 Sensitivity analysis is considered in Christof~\cite{MR3899156}. For optimal control of elliptic-parabolic variational
              inequalities with time-dependent constraints see Hofmann, Kubo, and Yamakaki~\cite{MR2228961}. Wachsmuth~\cite{MR3438403} studies optimal control of quasistatic plasticity with linear
              kinematic hardening and derives optimality conditions. Chen, Chu, and Tan~\cite{MR2346390} analyze bilateral obstacle control problem of parabolic variational inequalities.
For time optimal control of parabolic variational inequalities see Barbu~\cite{MR726973}, where a variant of the maximum principle for time-optimal trajectories of control systems governed by certain variational inequalities of parabolic type is derived. 
Optimal control problems of parabolic variational inequalities of second  kind have been addressed by Boukrouche and Tarzia~\cite{MR2801013}.

\vspace*{11pt}
The rest of the paper is organized as follows. In Section~\ref{sec:ExistUniq-var}
we analyze the Moreau--Yosida approximations. The stabilization of the Moreau--Yosida
approximations is addressed in Section~\ref{sec:stabilization_sequ}.
Section~\ref{sec:stabil-var} is dedicated to the proof of the main stabilization
result for the variational inequality. Finally, in Section~\ref{sec:num}
several numerical examples are presented for the case of a regular obstacle fulfilling the theoretical assumptions,
and in Section~\ref{sec:num_nonsmooth} a less regular obstacle~$\psi$ is considered for the sake of comparison.

\textbf{Notation:}
For an open interval $I\subseteq\bbR$ and two Banach spaces~$X,\,Y$, we write
	$W(I;\;X,\,Y)\coloneqq\{y\in L^2(I;\,X)\mid \dot y\in L^2(I;\,Y)\}$,
where~$\dot y\coloneqq\frac{\ed}{\ed t}y$ is taken in the sense of
distributions. This space is a Banach space when endowed with the natural norm
$	|y|_{W(I;\,X,\,Y)}\coloneqq\bigl(|y|_{L^2(I;\,X)}^2+|\dot y|_{L^2(I;\,Y)}^2\bigr)^{1/2}.$
If the inclusions $X\subseteq Z$ and~$Y\subseteq Z$ are continuous, where~$Z$ is a Hausdorff topological space,
then we can define the Banach spaces $X\cap Y$, $X\times Y$, and $X+Y$,
endowed with the norms defined as,
\begin{align}
|(a,\,b)|_{X\times Y} &:=\bigl(|a|_{X}^2+|b|_{Y}^2\bigr)^{\frac{1}{2}}, \quad |a|_{X\cap Y} :=|(a,\,a)|_{X\times Y},\notag\\
	|a|_{X+Y}&:=\inf\limits_{(a_1,\,a_2)\in X\times Y}\bigl\{|(a_1,\,a_2)|_{X\times Y}\mid a=a_1+a_2\bigr\},\notag
	\end{align}
respectively.
In case we know that $X\cap Y=\{0\}$, we say that $X+Y$ is a direct sum and we write $X\oplus Y$ instead.
If the inclusion
$X\subseteq Y$ is continuous, we write $X\xhookrightarrow{} Y$.

The space of continuous linear mappings from~$X$ into~$Y$ is denoted by~$\clL(X,Y)$. In case~$X=Y$ we 
write~$\clL(X)\coloneqq\clL(X,X)$.
The continuous dual of~$X$ is denoted~$X'\coloneqq\clL(X,\bbR)$.
The space of continuous functions from~$X$ into~$Y$ is denoted by~$\clC(X,Y)$.
Given a subset~$S\subset H$ of a Hilbert space~$H$, with scalar product~$(\Bigcdot,\Bigcdot)_H$, the orthogonal complement of~$S$ is
denoted~$S^\perp\coloneqq\{h\in H\mid (h,s)_H=0\mbox{ for all }s\in S\}$.
Given two closed subspaces~$F\subseteq H$ and~$G\subseteq H$ of the Hilbert space~$H=F\oplus G$, we denote by~$P_F^G\in\clL(H,F)$
the oblique projection in~$H$ onto~$F$ along~$G$. That is, writing $h\in H$ as $h=h_F+h_G$ with~$(h_F,h_G)\in F\times G$, we have~$P_F^Gh\coloneqq h_F$.
The orthogonal projection in~$H$ onto~$F$ is denoted by~$P_F\in\clL(H,F)$. Notice that~$P_F= P_F^{F^\perp}$.
By
$\overline C_{\left[a_1,\dots,a_n\right]}$ we denote a nonnegative function that
increases in each of its nonnegative arguments.
Finally, $C,\,C_i$, $i=0,\,1,\,\dots$, stand for unessential positive constants.

\section{ Existence, uniqueness, and approximation of the solution}\label{sec:ExistUniq-var}
We consider here a more general version of system~\eqref{sys-haty}, which will allow us to work with
the controlled system~\eqref{sys-tildey-K-intro} as well.
Namely
\begin{subequations}\label{sys-y-gen}
 \begin{align}
&\left( \tfrac{\p}{\p t}{ y} +(-\Delta+\Id) y+ Q y -f,v- y\right)_{L^2}\geq 0,
\quad \forall v\leq \psi,\;\; t>0,\\
& y\leq \psi,\quad \clG y\rest{\Gamma}=\chi,\quad	t>0,\qquad y(\Bigcdot,0)= y_\circ,
\end{align}
\end{subequations}
with~$Q=Q(x,t)\coloneqq \clB(x,t)+b(x,t)\cdot\nabla$ where $\clB(\Bigcdot,t)\in\clL(L^2)$
is a general linear bounded mapping, from~$L^2(\Omega)$ into itself.\
 
We show that there exists a solution of~\eqref{sys-y-gen},  which can be approximated by
the sequence~$({ y}_k)_{k\in\mathbb N}$, where~$y_k$ is the solution of the system
\begin{align}\label{sys-haty-k}
&\tfrac{\p}{\p t}{ y}_k +(-\Delta+\Id) y_k+Q y_k+k({ y}_k-\psi)^+=f,\qquad
{y}_k(0)={ y}_\circ ,\qquad\clG  y_k\rest{\Gamma}=\chi,
\end{align}
with
 \[
  v^+(x)\coloneqq
  \begin{cases}v(x),&\mbox{ if } v(x)>0,\\
  0,&\mbox{ if } v(x)\le 0,\end{cases}
  \quad\mbox{for}\quad v\in L^2.
 \]

\subsection{{ Assumptions on the data}}\label{sec:assum}
We assume the following regularity assumptions for the data. Hereafter, we will denote~$\bbR_+\coloneqq(0,+\infty)$.

\begin{assumption}\label{A:char-dom}
The subset~$\Omega$ is bounded, open, and connected, located on one side of its
boundary~$\Gamma=\p\Omega$. Furthermore, either~$\Gamma$ is a compact~$C^2$-manifold or~$\Omega$ is a convex polygonal domain.
 \end{assumption}
Under Assumption~\ref{A:char-dom} we have the characterizations~\eqref{char-dom},
this follows from~\cite[Thms. 2.2.2.3,  2.2.2.5, 3.2.1.3 and 3.2.1.3]{Grisv85}.

\begin{assumption}\label{A:Arc}
 The operator~$Q$ in~\eqref{sys-y-gen} is a sum~$Q=\clB+b\cdot\nabla$ with
\[
 \clB\in L^\infty(\bbR_+;\clL(L^2))\quad\mbox{and}\quad b\in L^\infty(\Omega\times\bbR_+)^d.
\]
 \end{assumption}
Assumption~\ref{A:Arc} is satisfied if, for example, $\clB =a\Id$ with $a\in
L^\infty(\Omega\times\bbR_+)$.

\begin{assumption}\label{A:fchi}
The external forces~$f$ and~$\chi$, and initial condition~$y_\circ$ in~\eqref{sys-haty}, satisfy
\begin{align}\notag
&f\in L^2_{\rm loc}(\bbR_+;L^2),\quad \chi \in  \clT, \quad y_\circ\in L^2,\quad\mbox{and}\quad 
y_\circ\le\psi(\Bigcdot,0).
\end{align}
\end{assumption}

 See Section \ref{sec:trace} for the definition of $ \clT$ as the trace space of $W_{\rm loc}(\bbR_+;H^2,L^2)$. The condition $\chi\in \clT$ specifically means that there exists a function in $W_{\rm loc}(\bbR_+;H^2,L^2)$ such that $\clG h$ is equal to $\chi$ in the trace sense.

\begin{assumption}\label{A:obst-lb}
The obstacle satisfies~$\psi\in W_{\rm loc}(\bbR_+;H^2,L^2)$
 and~$\clG\psi\rest\Gamma\ge\chi-\eta$ for a suitable real function~$\eta(\overline x,t)=\eta(t)$
independent of~$\overline x\in\Gamma$ where:{ 
\begin{itemize}
 \item[(i)] for Dirichlet boundary conditions, $\eta=0$, 
 \item[(ii)] for Neumann boundary conditions, $\eta\ge0$ and $\eta\in W^{1,2}_{\rm loc}(\bbR_+)$.
\end{itemize}}
\end{assumption}

\begin{remark}\label{R:dirpsichi}
Notice that for Dirichlet boundary conditions, since we will be looking for a
solution  satisfying~$y\rest\Gamma=\chi$ and~$y\le\psi$,
then the requirement~$\psi\rest\Gamma\ge\chi$ is necessary.  Instead, for Neumann boundary conditions,
we do not claim the necessity of the requirements in Assumption~\ref{A:obst-lb}. However, the relaxation of
those requirements will, probably, involve extra technical difficulties.
\end{remark}

\subsection{Trace and lifting operators}\label{sec:trace}
For simplicity, we denote
\[
\clW\coloneqq W_{\rm loc}(\bbR_+;H^2,L^2)\quad\mbox{and}\quad
\clW_0\coloneqq W_{\rm loc}(\bbR_+;\rmD(A),L^2)\subset\clW.
\]
Let us define the trace spaces on the boundary
\begin{align}\notag
 &\clT\coloneqq \left\{\clG h\rest\Gamma\mid h\in\clW\right\},
\;&& \clT_0\coloneqq
\left\{\clG h\rest\Gamma\mid h\in \clW_0\right\}.
\end{align}
Recall that we have
(cf.~\cite[Ch.~1, Thms.~3.2 and~9.6]{LioMag72-I})
for the trace spaces at initial time,
\begin{align}\notag
&\clW^{[t=0]}\coloneqq\{y(0)\mid y\in\clW\}=H^1,
\;&&\clW_0^{[t=0]}\coloneqq\{y(0)\mid y\in\clW_0\}=V.
\end{align}

Now for any finite time interval~$(t_1,t_2)$, with $t_2>t_1$, we { define} the Hilbert
spaces 
\begin{equation}\label{clWt1t2}
\clW_{(t_1,t_2)}:=W((t_1,t_2),H^2,L^2)
\end{equation}
and the corresponding traces are denoted by~$\clT_{(t_1,t_2)}=\clW_{(t_1,t_2)}\rest\Gamma$.\black

Next for each positive integer~$j\in\bbN$ we define the time interval~$I_j\coloneqq(j-1,j)$.
Observe that for any~$\chi\in\clT$ we have that $\chi\rest{I_j}\in \clT_{I_j}$.
We consider the extension (lifting) function defined, for~$\widetilde\chi\in \clT_{I_j}$ by 
\[
\fkE^j\widetilde\chi\in\clW_{I_j},\quad (\clG\fkE^j\widetilde\chi)\rest\Gamma=\widetilde\chi,
\quad\mbox{and}\quad \fkE^j\widetilde\chi\in\clW_{I_j,0}^\perp,\quad\mbox{with}\quad\clW_{I_j,0}\coloneqq \clW_{I_j}\bigcap \clW_0\rest{I_j},
\]
where the orthogonal space~$\clW_{I_j,0}^\perp$  to~$\clW_{I_j,0}$
is taken { with respect to } the  scalar product of~$\clW_{I_j}$. This defines the extension operator, 
$\fkE^j\in\clL(\clT_{I_j},\clW_{I_j})$, which is a right inverse for the trace operator
$(\clG(\Bigcdot))\rest\Gamma\in\clL(\clW_{I_j},\clT_{I_j})$.
We endow~$\clT_{I_j}$ with the scalar product induced by the trace mapping
\[
 (\chi_1,\chi_2)_{\clT_{I_j}}\coloneqq (\fkE^j\chi_1,\fkE^j\chi_2)_{\clW_{I_j}}.
\]
{ This allows to introduce } the extension~$\fkE\colon\clT\to\clW$ defined by concatenation
\[
\fkE\chi(t)\coloneqq (\fkE^{\lceil t\rceil}\chi\rest{I_{\lceil t\rceil}}) (t),
\]
where~$\lceil t\rceil$ is the positive integer satisfying~$\lceil t\rceil-1<t\le\lceil t\rceil$.

\begin{remark}\label{R:compcond_y0}
Note that for any~$h\in\clW$ satisfying~$\clG h\rest\Gamma=\chi$ we have that~$\fkE\chi-h\in \clW_{0}$.
In particular we have that
$\fkE\chi(t)-h(t)\in V$, for all $t\ge0$.
\end{remark}

\begin{remark}
Several existence results for parabolic variational inequalities can be found in the literature.
However, though we borrow some ideas and arguments from classic
references (e.g, ~\cite{bensoussan2011applications,Bensoussan1984,MR635927,Brezis71}) we could not find in the literature,
the existence results for obstacles as general as in Assumption~\ref{A:obst-lb}.
For example in~\cite[Ch.~3, Sect.~2.2, Thm.~2.2]{bensoussan2011applications}, for Dirichlet boundary conditions
it is assumed that the
boundary trace of the obstacle is static\black (independent of time). In~\cite[Sect.~II]{Brezis71} the
triple~$(a,b,\psi)$ is time-independent.
\end{remark}

\subsection{On the Moreau--Yosida approximation}
We present the main result concerning Moreau--Yosida approximations  for parabolic variational inequalities.
We start by denoting, for a given function~$\varphi\in L^2_{\rm loc}(\bbR_+,L^2)$, the 
convex sets 
\begin{subequations}\label{convex}
\begin{align}
\bfC^\varphi_T&\coloneqq\{v\in L^2((0,T);{H^1})\mid v\le\varphi\},\quad\mbox{for}\quad T>0,
\intertext{and}
\bfC^\varphi_\infty&\coloneqq\{v\in L^2_{\rm loc}(\mathbb{R}_+;{H^1})\mid v\le\varphi\}.
\end{align}
\end{subequations}

We set
\begin{align}
	\clZ_r&\coloneqq\{z\in W((0,T);H^1,V')\mid rz\in W((0,T);H^2,L^2)\},\notag
	\intertext{where} 
	r(t)& \coloneqq\min\{t,1\},\quad t\ge0.\notag
\end{align} 
 
\begin{theorem}\label{T:MY-approx}
Let Assumptions~\ref{A:char-dom}--\ref{A:obst-lb} hold true, $T>0$, and suppose $(f_k)\subset L^2((0,T); L^2)$ 
converges weakly to some $f$ in $L^2((0,T); L^2)$. Then, for a given~$k\in\bbN$. there exists one, and only one, weak
solution~$y_k\in\clZ_r$ for 
 \begin{equation}
\tfrac{\p}{\p t} y_k +(-\Delta+\Id)y_k+Qy_k+k(y_k-\psi)^+=f_k,
\qquad \clG y_k\rest\Gamma=\chi,\qquad y_k(0)=y_\circ.\label{sys:MYyk}
\end{equation}
Moreover, the sequence~$(y_k)$ of solutions satisfy
\begin{align}\label{seq-gk-weak}
y_{k}-\fkE\chi\xrightharpoonup[L^2((0,T); V)]{} y-\fkE\chi,\qquad
\tfrac{\p}{\p t} (y_{k}-\fkE\chi)\xrightharpoonup[L^2((0,T); V')]{}
\tfrac{\p}{\p t} (y-\fkE\chi),
 \end{align}
for some~$y\in\clZ_r$ with
 \begin{align}\label{eq:limMYyk}
 y\in\bfC^\psi_T,\qquad y(0)=y_\circ,\qquad\clG y\rest\Gamma=\chi, 
\end{align}
and,
for an arbitrary
$v\in\clZ_r{\textstyle\bigcap}\bfC^\psi_T$,  with~$v-y\in \clC((0,T]; V)$,  we have
\begin{align}\label{MY:VI-weakfull}
\langle\tfrac{\p}{\p t} y +(-\Delta+\Id)y+Qy-f,v-y\rangle_{V',V}\ge 0,\quad \text{almost everywhere  in } (0,T).
 \end{align}
Furthermore, we have
\begin{align}\label{seq-gk}
r(y_{k}-\fkE\chi)\xrightharpoonup[L^2((0,T); \rmD(A))]{} r(y-\fkE\chi),\quad
\tfrac{\p}{\p t}( r(y_{k}-\fkE\chi))\xrightharpoonup[L^2((0,T); L^2)]{}
\tfrac{\p}{\p t}( r(y-\fkE\chi)),
 \end{align}
and, for arbitrary $v\in L^2((0,T); L^2)$,
\begin{align}\label{MY:VI-strongfull}
\left(\tfrac{\p}{\p t} y +(-\Delta+\Id)y+Qy-f,v-y\right)_{L^2}\ge 0,\quad \text{almost everywhere in } (0,T).
\end{align}
Finally, $y$ is unique the only element in~$\clZ_r$ satisfying~\eqref{eq:limMYyk} and~\eqref{MY:VI-weakfull}, and we have
\begin{align}\label{MY:stlimitsVI-full}
y_{k}\xrightarrow[L^2((0,T); L^2)]{}  y\qquad\mbox{and}\qquad 
r(y_{k}-\fkE\chi)\xrightarrow[\clC({[0,T]}; L^2)]{}  r(y-\fkE\chi).
 \end{align} 
\end{theorem}

The proof of Theorem \ref{T:MY-approx} is given in several steps, which we include in several lemmas.

\begin{lemma}\label{L:exist-parab1k}
Let Assumptions~\ref{A:char-dom}--\ref{A:obst-lb} hold true.
Let us fix~$k\in \mathbb{N}$. There exists one, and only one, 
solution~$y_k\in W((0,T);H^1,V')$ for \eqref{sys:MYyk}, furthermore~$ry_k\in W((0,T);H^2,L^2)$. 
\end{lemma}
\begin{proof}  We sketch the proof which follows from standard arguments.
 By a lifting argument (cf.~\cite[Def.~3.1]{Rod14-na})
 we can reduce the problem to the case of homogeneous boundary conditions, where
 we can prove the existence of weak solutions, in~$W((0,T),V,V')$, as a weak limit of suitable Galerkin
 approximations. Weak solutions are understood in the classical sense~\cite{Temam01,Lions69}.
 Strong solutions in~$W((0,T),H^2,L^2)$ can be proven for more regular initial conditions~$y_\circ\in V$,
 see~\cite[Sect.4.3]{Rod20-eect}. For our initial conditions in~$y_\circ\in L^2\setminus V$, 
we can use the smoothing property of
 parabolic-like equations to conclude that~$ry_k\in W((0,T),H^2,L^2)$, see~\cite[Ch.~3, Thm.~3.10]{Temam01}
 and~\cite[Lem.~2.6]{PhanRod18-mcss}.
 Note that~$r(0)y_k(0)=0\in V$ at initial time.
\end{proof}

Note that by direct computations
 \begin{align}
 &(h,h^+)_{L^2}=\norm{h^+}{L^2}^2,\quad\mbox{for all}
 \quad h\in L^2.\label{monot+}
\end{align}

Let us denote
\begin{equation}\label{Crc}
C_{Q}\coloneqq \norm{Q}{L^\infty(\bbR^+,\clL(H^1,L^2))}
.
\end{equation}
\begin{lemma}\label{L:weak-parab1k}
 Let Assumptions~\ref{A:char-dom}--\ref{A:obst-lb} hold true. Then, the solution~$y_k$
 for \eqref{sys:MYyk} satisfies
\begin{align}
&2k\norm{(y_k-\psi)^+}{L^2((0,T),L^2)}^2+\norm{y_k}{L^\infty((0,T),L^2)}^2+\norm{y_k}{L^2((0,T), H^1)}^2\notag\\
&\hspace{3em}\le\ovlineC{C_Q ,T}\left(\norm{y_\circ}{L^2}^2
+\norm{\fkE\chi}{\clW_{(0,T)}}^2
+\norm{f_k}{L^2((0,T),L^2)}^2+\norm{\psi}{W((0,T), H^1,V')}^2\right),\notag
\end{align}
with~$\ovlineC{C_Q ,T}$  independent of~$k$.
\end{lemma}
\begin{proof}
Recall that $\psi\in W((0,T); H^2, L^2)$ by Assumption~\ref{A:obst-lb}. Now we set
\begin{equation}\label{vchipsi}
  v\coloneqq\fkE\chi -(\fkE\chi-\psi)^+,
 \end{equation}
which implies ~$v\in W((0,T);H^1,L^2)$. Also, $\psi-v\ge0$, because
\begin{equation}\notag
\begin{aligned}
&\psi-v=0 ,&\mbox{if}\quad \fkE\chi\ge\psi,\\
&\psi-v=\psi-\fkE\chi,&\mbox{if}\quad \fkE\chi\le\psi.
\end{aligned}
\end{equation}
Furthermore under
Dirichlet boundary conditions we also have that~$v\rest\Gamma=\chi$, because
$(\fkE\chi-\psi)^+\rest\Gamma=0$, due to~$\chi\le\psi\rest\Gamma$ in Assumption~\ref{A:obst-lb}.
Hence, we have
\begin{equation}\label{wvchi}
 p_k\coloneqq y_k-v\in W((0,T);V,L^2),\qquad v\le\psi, 
\end{equation}
and
\begin{align}
&\dot p_k +Ap_k+Qp_k+k(y_k-\psi)^+=h_k,\notag
\intertext{with}
&h_k\coloneqq f_k-\tfrac{\rmd}{\rmd t} v -(-\Delta+\Id)v\label{hkfkv}
-Qv.
\end{align}
After  testing the dynamics  with~$2p_k$ to obtain
\begin{align}\notag
\tfrac{\rmd}{\rmd t} \norm{p_k}{L^2}^2 +2\norm{p_k}{V}^2
+2k((y_k-\psi)^+,p_k)_{L^2}= 2\langle-Qp_k+h_k,p_k\rangle_{V',V}.
\end{align}
Observe that,  due to~\eqref{wvchi} we have $p_k\ge y_k-\psi$ and
\begin{equation}\notag
((y_k-\psi)^+,p_k)_{L^2}\ge\norm{(y_k-\psi)^+}{L^2}^2,
\end{equation}
and by using Assumption~\ref{A:Arc} and the Young inequality, and recalling~\eqref{Crc}, it follows that
\begin{align}
&\tfrac{\rmd}{\rmd t} \norm{p_k}{L^2}^2 +\norm{p_k}{V}^2+2k\norm{(y_k-\psi)^+}{L^2}^2
\le 2C_Q ^2\norm{p_k}{L^2}^2+2\norm{h_k}{V'}^2\notag\\
&\hspace{3em}\le \ovlineC{C_Q }\left(\norm{p_k}{L^2}^2
+\norm{h_k}{V'}^2\right).
\label{dtnkappa-w2}
\end{align}
\begin{subequations}\label{weakLinfHL2V}
By the Gronwall Lemma it follows that
\begin{align}
\norm{p_k}{L^\infty((0,T),L^2)}^2\le \ovlineC{C_Q ,T}\left(\norm{p_k(0)}{L^2}^2
+\norm{h_k}{L^2((0,T),V')}^2\right),\label{weakLinfH}
\end{align}
and by integration of~\eqref{dtnkappa-w2}, and using~\eqref{weakLinfH}, we find
\begin{align}
&\norm{p_k}{L^2((0,T),V)}^2+2k\norm{(y_k-\psi)^+}{L^2((0,T),L^2)}^2
\le \ovlineC{C_Q ,T}\left(\norm{p_k(0)}{L^2}^2
 +\norm{h_k}{L^2((0,T),V')}^2\right).\label{weakL2V}
\end{align}
\end{subequations}
\begin{subequations}\label{weak-fobs}
Now, note that from~\eqref{hkfkv}, \eqref{wvchi}, \eqref{vchipsi},  \eqref{hkfkv},\black
and~$L^2\xhookrightarrow{}V'$, we have
\begin{align}
&\norm{h_k}{L^2((0,T),V')}^2\le   \ovlineC{C_Q }
\left(\norm{f_k}{L^2((0,T),V')}^2
+\norm{v}{W((0,T),H^1,V')}^2\right)\\
&\hspace{2em}\le \ovlineC{C_Q }
\left(\norm{f_k}{L^2((0,T),V')}^2
+\norm{\fkE\chi}{\clW_{(0,T)}}^2+\norm{\psi}{W((0,T),H^1,V')}^2\right),\label{esthkL2L2}
\intertext{ (cf.~\eqref{clWt1t2}), and}
&\norm{y_k}{L^\infty((0,T),L^2)}^2+\norm{y_k}{L^2((0,T), H^1)}^2
\notag\\
&\hspace{2em}\le 2\norm{p_k}{L^\infty((0,T),L^2)}^2 
+2\norm{v}{L^\infty((0,T),L^2)}^2
+2\norm{p_k}{L^2((0,T),V)}^2
+2\norm{v}{L^2((0,T), H^1)}^2\notag\\
&\hspace{2em}\le\ovlineC{C_Q ,T}
\left(\norm{p_k(0)}{H}^2+\norm{\fkE\chi}{\clW_{(0,T)}}^2
+\norm{f_k}{L^2((0,T),V')}^2+\norm{\psi}{W((0,T), H^1,V')}^2\right).
\intertext{Notice also that}
&\norm{p_k(0)}{H}^2= \norm{y_k(0)-v(0)}{L^2}^2
\le 2\norm{y_\circ}{L^2}^2+2\norm{\fkE\chi(0) -(\fkE\chi(0)-\psi(0))^+}{L^2}^2.
\end{align}
\end{subequations}
Hence, the result follows from~\eqref{weakLinfHL2V} and~\eqref{weak-fobs}.
\end{proof}

The following lemma establishes that we are able to identify a pseudo-distance function with an strictly negative normal derivative.

\begin{lemma}\label{L:xi-hat}
Let Assumption~\ref{A:char-dom} hold true. Then, there
exists~$\xi\in H^2(\Omega)\bigcap C^2(\Omega)\bigcap C^1(\overline\Omega)$ and constant~$c_\xi<0$ satisfying
\begin{subequations}
 \begin{align}
&\xi(x)\ge0 \quad\mbox{for all}\quad x\in\overline\Omega,\label{xige0}\\
&\tfrac{\p}{\p\bfn}\xi \rest\Gamma(\overline x)\le c_\xi
\quad\mbox{for almost all}\quad \overline x\in\overline\Gamma.\label{normalxileC}
\end{align}
\end{subequations}
\end{lemma}
\begin{proof}
In the case $\Omega$ is of class~$C^2$, we can choose~$\xi=\rho d_\Gamma$ as the product of
the distance to the boundary function,
$d_\Gamma(x)=\min_{z}\{\norm{x-z}{\bbR^d}\min z\in\Gamma\}$,
and of a suitable cut-off function~$\rho$. From~\cite[Appendix, Lem.~1 and Eq.~(A7)]{GilbargTrudinger98},
see also~\cite[Sect.~13.3.4]{Katopodes19}, we know
that~$d_\Gamma\in C^2(\Gamma_\delta)$ for a suitable small enough~$\delta>0$
and~$\Gamma_\delta\coloneqq\{x\in\overline\Omega\mid d_\Gamma(x)\le\delta\}$, and also
that~$\frac{\p d_\Gamma}{\p\bfn}=1$. For~$\rho$ we choose a smooth function satisfying
 $0\le\rho\le1$,   such that~$\rho(x)=0$ for all~$x\in\Omega\setminus\Gamma_{\frac{2\delta}3}$,
and~$\rho(x)=1$ for all~$x\in\Gamma_{\frac\delta3}$.

In the case~$\Omega$ is a convex polygonal domain we can choose $x_0\in\Omega$ and
\[
\xi(x)=-\norm{x-x_0}{\bbR^d}^2 +\max_{z\in\overline\Omega}
\norm{z-x_0}{\bbR^d}^2, \quad x\in\overline\Omega,
\]
It is clear that~$\xi\in C^2(\overline\Omega)$ and that~$\xi\ge0$. It remains to prove that~$\xi$ strictly
decreases on~$\Gamma$
in the direction of the outward normal~$\bfn$. To this purpose let~$\overline x\in\Gamma$ and let~$F$ be a
face of~$\Gamma$ contained in the affine hyperplane $\bbH$ and such that~$\overline x\in F$. Up to
an affine change of variables (a translation and a
rotation) we can suppose that~$0\in\Omega$ and
\[
 x_0=0\quad\mbox{and}\quad\bbH=\{(s,x_2,x_3,\dots,x_d)\mid (x_2,x_3,\dots,x_d)\in\bbR^{d-1}\}
 \quad\mbox{with}\quad s>0.
\]
In this case, we find that
\[
 \xi(x)=-\norm{x}{\bbR^d}^2 +\max_{z\in\overline\Omega}
\norm{z}{\bbR^d}^2,  \quad\bfn=(1,0,0,\dots,0)\quad\mbox{and}\quad \tfrac{\p}{\p\bfn} \xi\rest\Gamma
= \tfrac{\p}{\p x_1} \xi\rest\Gamma
=-2x_1.
\]
Therefore at an arbitrary point~$\overline x\in\bbH$
we find that~$\tfrac{\p}{\p\bfn} \xi\rest\Gamma(\overline x)=-2\overline x_1=-2s$.
Note that~$s$ is the distance from~$0$ to~$\bbH$.

Therefore we can conclude that for every point~$\overline x$ in the (boundary)
interior of a face~$F$ we have that
$\tfrac{\p}{\p\bfn}\xi \rest\Gamma(\overline x)=-2s_F$ where $s_F>0$ is the distance from~$x_0$ to the hiperplane~$\bbH_F$
containing~$F$. Since the number of faces is finite,
$\tfrac{\p}{\p\bfn}\xi \rest\Gamma\le\max\{-2s_F\mid F\mbox{ is a face of }\Gamma\}\eqqcolon c_\xi<0$,
for all  boundary points\black living in one face only. Note that if~$\overline x$ lives in the intersection of two
faces then the normal
derivative is not well defined (not continuously, at least), however
the set of such points has vanishing (boundary) measure. That is,
$\tfrac{\p}{\p\bfn}\xi \rest\Gamma(\overline x)\le c_\xi<0$ for almost every boundary point~$\overline x$.
\end{proof}

\begin{lemma}\label{L:beta.bdry} Let~$c_\xi<0$ and~$\xi\in H^2$ be as in Lemma~\ref{L:xi-hat}, 
and~$\eta\ge \chi-\clG\psi\rest\Gamma$ be as in Assumption~\ref{A:obst-lb}.
 Then, for
\begin{align}\label{zetakxi}
&\zeta_k\coloneqq y_k-\psi+\eta\widehat \xi,\quad\mbox{with}\quad\widehat \xi\coloneqq\begin{cases}
                   0,&\mbox{ if }\clG=\Id,\\
                   -c_\xi^{-1}\xi,&\mbox{ if }\clG=\frac{\p}{\p\bfn},
                 \end{cases}
\end{align}
 where $y_k$ is the solution
 for \eqref{sys:MYyk}, we have that
\begin{align}\notag
(\tfrac{\p}{\p\bfn}\fkE\chi,\zeta_k^+)_{L^2(\Gamma)}-(\psi- \eta\widehat \xi,\zeta_k^+)_{H^1}
\le 2\norm{\psi-\eta\widehat \xi}{H^2}\norm{\zeta_k^+}{L^2},\qquad\clG\in\{\Id,\tfrac{\p}{\p\bfn}\}.
\end{align}
\end{lemma}
\begin{proof}
Observe that
\begin{align}
&(\tfrac{\p}{\p\bfn}\fkE\chi,\zeta_k^+)_{L^2(\Gamma)}-(\psi- \eta\widehat \xi,\zeta_k^+)_{H^1}\notag\\
&\hspace{3em}=(\tfrac{\p}{\p\bfn}\fkE\chi,\zeta_k^+)_{L^2(\Gamma)}
+((\Delta-\Id)(\psi- \eta\widehat \xi),\zeta_k^+)_{L^2}
-(\tfrac{\p}{\p\bfn}(\psi- \eta\widehat \xi) ,\zeta_k^+)_{L^2(\Gamma)}\notag\\
&\hspace{3em}=( \tfrac{\p}{\p\bfn}\fkE\chi-\black\tfrac{\p}{\p\bfn}\psi+\eta\tfrac{\p}{\p\bfn}\widehat \xi,\zeta_k^+)_{L^2(\Gamma)}
+((\Delta-\Id)(\psi- \eta\widehat \xi),\zeta_k^+)_{L^2}.\label{beta.bdry1}
\end{align}
Note that
\begin{subequations}\label{beta.bdry2}
\begin{equation} 
\zeta_k^+\rest\Gamma=0,\quad\mbox{if}\quad \clG=\Id,\qquad\mbox{and}\qquad\tfrac{\p}{\p\bfn}\fkE\chi=\chi,\quad\mbox{if}\quad \clG=\tfrac{\p}{\p\bfn}.
\end{equation}
Now, by using~\eqref{normalxileC} and~\eqref{zetakxi},
\begin{equation}
\begin{aligned}
& \tfrac{\p}{\p\bfn}\fkE\chi-\tfrac{\p}{\p\bfn}\psi+\eta\tfrac{\p}{\p\bfn}\widehat \xi=\black\chi-\tfrac{\p}{\p\bfn}\psi\rest\Gamma+\eta\tfrac{\p}{\p\bfn}\widehat \xi\rest\Gamma
 \le\chi-\tfrac{\p}{\p\bfn}\psi\rest\Gamma-\eta\le0,&\mbox{if}\quad \clG=\tfrac{\p}{\p\bfn}
\end{aligned}
\end{equation}
\end{subequations}
and, by~\eqref{beta.bdry2}, we have that
\begin{align}\label{beta.bdry3}
 &( \tfrac{\p}{\p\bfn}\fkE\chi- \tfrac{\p}{\p\bfn}\psi+\eta\tfrac{\p}{\p\bfn}\widehat \xi,\zeta_k^+)_{L^2(\Gamma)}\le0,
 \quad\mbox{if}\quad \clG\in\{\Id,\tfrac{\p}{\p\bfn}\},
\end{align}
with an equality in the case~$\clG=\Id$. Thus, by~\eqref{beta.bdry1} and~\eqref{beta.bdry3} we obtain
\begin{align}
&(\tfrac{\p}{\p\bfn}\fkE\chi,\zeta_k^+)_{L^2(\Gamma)}-(\psi- \eta\widehat \xi,\zeta_k^+)_{H^1}
\le\norm{(\Delta-\Id)(\psi- \eta\widehat \xi)}{L^2}\norm{\zeta_k^+}{L^2}
\le2\norm{\psi- \eta\widehat \xi}{H^2}\norm{\zeta_k^+}{L^2},
\end{align}
which ends the proof.
\end{proof}

\begin{lemma}\label{L:weak-viol}
 Let Assumptions~\ref{A:char-dom}--\ref{A:obst-lb} hold true. Then, the solution~$y_k$
 for \eqref{sys:MYyk} satisfies
\begin{align}
&k^2\norm{(y_k-\psi)^+}{L^2((0,T),L^2(\Omega))}^2+\norm{\tfrac{\rmd}{\rmd t}(y_k-\fkE\chi)}{L^2((0,T),V')}^2\notag\\
&\hspace*{2em}\le\ovlineC{C_Q ,T}\left(\norm{y_\circ}{ L^2}^2
+\norm{\fkE\chi}{\clW_{(0,T)}}^2
+\norm{f_k}{L^2((0,T), L^2)}^2
+\norm{\psi}{\clW_{(0,T)}}^2+\norm{\eta}{W^{1,2}(0,T)}^2
\right),\notag
\end{align}
with~$\ovlineC{C_Q ,T}$ independent of~$k$.
\end{lemma}

\begin{proof}
 Let us choose~$c_\xi<0$ and~$\xi$ as in Lemma~\ref{L:xi-hat} implying in particular that $\xi \in H^2$. 
We also have~$ \eta\ge\chi-\clG\psi\rest\Gamma$, due to Assumption~\ref{A:obst-lb}.
 Then, we set~$\zeta_k$ as in~\eqref{zetakxi}.

Observe that both~$\zeta_k$ and~$\zeta_k^+$ are in~$H^1$.
 Furthermore, in the case of Dirichlet boundary conditions we also have~$\zeta_k^+\in H^1_0$ as a corollary of
 Assumption~\ref{A:obst-lb}. Therefore, 
\begin{equation}\label{+inV}
\zeta_k^+\in V,\quad\mbox{for}
\quad\clG\in\{\tfrac{\p}{\p\bfn},\Id\}.
\end{equation}

 Let us denote now~$\varkappa_k=y_k-\fkE\chi$. We find
\begin{subequations}\label{sys:MYkappak}
\begin{align}
&\hspace{-2em}\dot \varkappa_k +A\varkappa_k+Q\varkappa_k+k(y_k-\psi)^+=g_k,
\qquad\varkappa_k(0)=\varkappa_\circ,\qquad \clG \varkappa_k\rest\Gamma=0,
\intertext{with}
\varkappa_\circ &=y_\circ-\fkE\chi(0),\qquad
g_k\coloneqq f_k-\tfrac{\rmd}{\rmd t} \fkE\chi -(-\Delta+\Id)\fkE\chi
-Q\fkE\chi.\label{kappa0gk}
\end{align}
\end{subequations}

Testing the dynamics
with~$\zeta_k^+$, gives us
\begin{align}
0&=(\dot \varkappa_k,\zeta_k^+)_{L^2} +(\varkappa_k,\zeta_k^+)_V+k((y_k-\psi)^+,\zeta_k^+)_{L^2}
+(Q\varkappa_k-g_k,\zeta_k^+)_{L^2}\notag\\
&=(\dot \varkappa_k+\tfrac{\rmd}{\rmd t} \fkE\chi-\dot\psi+\dot\eta\widehat \xi,\zeta_k^+)_{L^2}
+(\varkappa_k+\fkE\chi-\psi+\eta\widehat \xi,\zeta_k^+)_{H^1}
+k((y_k-\psi)^+,\zeta_k^+)_{L^2}\notag\\
&\quad+(Q\varkappa_k-g_k-\tfrac{\rmd}{\rmd t} \fkE\chi+\dot\psi-\dot\eta\widehat \xi,\zeta_k^+)_{L^2}
+(-\fkE\chi+\psi-\eta\widehat \xi,\zeta_k^+)_{H^1}\notag
\end{align}
which is equivalent to
\begin{align}
0&=(\dot \zeta_k,\zeta_k^+)_{L^2}+(\zeta_k,\zeta_k^+)_{H^1}
+k((y_k-\psi)^+,\zeta_k^+)_{L^2}\notag\\
&\quad+(Q\varkappa_k-g_k-\tfrac{\rmd}{\rmd t} \fkE\chi+\dot\psi-\dot\eta\widehat \xi,\zeta_k^+)_{L^2}
+(-\fkE\chi+\psi-\eta\widehat \xi,\zeta_k^+)_{H^1}.\notag
\end{align}
Then, using  Stampacchia Lemma~\cite[Lem.~1.1]{Stampacchia63})
and Lions-Magenes Lemma~\cite[Ch.~3, Sect.~1.4, Lem.~1.2]{Temam01}, we arrive at
\begin{align}
&\tfrac{\rm d}{\rmd t}\norm{\zeta_k^+}{L^2}^2 +2\norm{\zeta_k^+}{V}^2+2k((y_k-\psi)^+,\zeta_k^+)_{L^2}\notag\\
&\hspace{3em}=2(-Q\varkappa_k+g_k+\tfrac{\rmd}{\rmd t} \fkE\chi-\dot\psi+\dot\eta\widehat \xi,\zeta_k^+)_{L^2}
-2(-\fkE\chi+\psi-\eta\widehat \xi,\zeta_k^+)_{H^1}.\notag
\end{align}

Next, we use the relations in~\eqref{sys:MYkappak} to obtain
\begin{align}
&\tfrac{\rm d}{\rmd t}\norm{\zeta_k^+}{L^2}^2 +2\norm{\zeta_k^+}{V}^2 +2k((y_k-\psi)^+,\zeta_k^+)_{L^2}\notag\\
&\hspace{3em}=2(-Qy_k+f_k -(-\Delta+\Id)\fkE\chi
-\dot\psi+\dot\eta\widehat \xi,\zeta_k^+)_{L^2}
-2(-\fkE\chi+\psi-\eta\widehat \xi,\zeta_k^+)_{H^1}\notag\\
&\hspace{3em}=2(-Qy_k+f_k 
-\dot\psi+\dot\eta\widehat \xi,\zeta_k^+)_{L^2}
-2(\psi-\eta\widehat \xi,\zeta_k^+)_{H^1}+2(\tfrac{\p}{\p\bfn}\fkE\chi,\zeta_k^+)_{L^2(\Gamma)}\label{est:for+1}
\end{align}
and, using Lemma~\ref{L:beta.bdry} we find
\begin{align}
&\tfrac{\rm d}{\rmd t}\norm{\zeta_k^+}{L^2}^2 +2\norm{\zeta_k^+}{V}^2 +2k((y_k-\psi)^+,\zeta_k^+)_{L^2}\\
&\hspace{3em}\le2(-Qy_k+f_k 
-\dot\psi+\dot\eta\widehat \xi,\zeta_k^+)_{L^2}+4\norm{\psi- \eta\widehat \xi}{H^2}\norm{\zeta_k^+}{L^2}\\
&\hspace{3em}
\le  2\left(\norm{-Qy_k+f_k-\dot\psi+\dot\eta\widehat \xi}{L^2}
+2\norm{\psi- \eta\widehat \xi}{H^2}\right)
\norm{\zeta_k^+}{H^2}.\label{est:key+1}
\end{align}

Time integration of~\eqref{est:key+1} gives us
\begin{align}\notag
&\norm{\zeta_k^+(T)}{L^2}^2-\norm{\zeta_k^+(0)}{L^2}^2+ 2\norm{\zeta_k^+}{L^2((0,T),V)}^2
 +2k((y_k-\psi)^+,\zeta_k^+)_{L^2((0,T),L^2)}
\le 2\Xi
\norm{\zeta_k^+}{L^2((0,T),L^2)}
\end{align}
with
\begin{align}\notag
& \Xi\coloneqq\left(\norm{-Qy_k+f_k-\dot\psi+\dot\eta\widehat \xi}{L^2((0,T),L^2)}
+2\norm{\psi- \eta\widehat \xi}{L^2((0,T),H^2)}\right),
\end{align}
from which, together with the fact that, due to Assumption~\ref{A:fchi},  at time~$t=0$
we have~$\zeta_k^+(0)=(y_\circ-\psi(0))^+=0$, we obtain
\begin{align}\notag
2k\norm{((y_k-\psi)^+,\zeta_k^+)_{L^2((0,T),L^2)}}{\bbR}=2k((y_k-\psi)^+,\zeta_k^+)_{L^2((0,T),L^2)}
\le 2\Xi
\norm{\zeta_k^+}{L^2((0,T),L^2)},
\end{align}
 which, together with~$L^2((0,T),L^2)=(L^2((0,T),L^2))'$, give
us~$\norm{(y_k-\psi)^+}{L^2((0,T),L^2)}
\le {k}^{-1}{\Xi}$, thus
\begin{align}
&k\norm{(y_k-\psi)^+}{L^2((0,T),L^2)}\notag\\
&\hspace{3em}\le \Xi\le \ovlineC{C_Q }\left(\norm{y_k}{L^2((0,T),L^2)}+\norm{f_k}{L^2((0,T),L^2)}+\norm{\psi}{\clW(0,T)}
+\norm{\eta\widehat\xi}{\clW(0,T)}\right).\label{bound+}
\end{align}
 
Next, from~\eqref{sys:MYkappak} we also find that
\[
 \norm{\dot \varkappa_k}{V'}^2=\norm{A\varkappa_k+Q\varkappa_k+k(y_k-\psi)^+-g_k}{V'}^2
\]
which together with~\eqref{bound+}, $\varkappa_k=y_k-\fkE\chi$,~and~$L^2\xhookrightarrow{}V'$, give us
\[
\norm{\dot \varkappa_k}{L^2((0,T),V')}^2\le \overline C
\left(\norm{y_k}{L^2((0,T),H^1)}^2+\norm{\fkE\chi}{\clW(0,T)}^2+\norm{f_k}{L^2((0,T),L^2)}^2
 +\norm{\psi}{\clW(0,T)}^2+\norm{\eta}{W^{1,2}(0,T)}^2\right).
\]
with~$\overline C=\ovlineC{C_Q,\norm{\widehat\xi}{H^2} }$.
Finally, we can finish the proof by using Lemma~\ref{L:weak-parab1k}.
\end{proof}

\begin{remark}
We can see that the constant~$\ovlineC{C_Q,T}$ in the statement of the Lemma~\ref{L:weak-viol} will also
depend on~$\norm{\widehat\xi}{H^2}$
as~$\ovlineC{C_Q,T,\norm{\widehat\xi}{H^2}}$, but since essentially~$\widehat\xi$ depends only on the spatial
domain~$\Omega$, we omit the dependence on~$\norm{\widehat\xi}{H^2}$ in the statement of
Lemma~\ref{L:weak-viol} and throughout the manuscript.
\end{remark}

\begin{lemma}\label{L:strong-parab1k}
 Let Assumptions~\ref{A:char-dom}--\ref{A:obst-lb} hold true,  with in
 addition~$y_\circ-\fkE\chi(0)\in V$. Then the solution~$y_k$ for \eqref{sys:MYyk} satisfies
\begin{align}
&\norm{y_k}{L^2((0,T),H^2)}^2 +\norm{y_k}{L^\infty((0,T),H^1)}^2\notag\\
&\hspace{3em}\le  \overline C \left(\norm{y_\circ}{H^1}^2+\norm{\fkE\chi}{\clW_{(0,T)}}^2
+\norm{f_k}{L^2((0,T),L^2)}^2 +\norm{\psi}{\clW_{(0,T)}}^2
\right),\label{est:strong-parab1k}
\end{align}
with a constant~$\ovlineC{T,C_Q }$ independent of~$k$.
\end{lemma}
\begin{proof}
Testing the dynamics in~\eqref{sys:MYkappak} with~$2A\varkappa_k$,
 where~$\varkappa_k=y_k-\fkE\chi$, it follows that
 \begin{align}\notag
2\norm{\varkappa_k}{\rmD(A)}^2+\tfrac{\rmd}{\rmd t} \norm{\varkappa_k}{V}^2
=2(g_k-Q\varkappa_k -k(y_k-\psi)^+,A\varkappa_k)_{L^2}.
\end{align}
Then, the Young inequality gives us
 \begin{align}\notag
 &\norm{\varkappa_k}{\rmD(A)}^2+\tfrac{\rmd}{\rmd t} \norm{\varkappa_k}{V}^2
 \le\norm{g_k-Q\varkappa_k-k(y_k-\psi)^+)}{L^2}^2,
\end{align}
 and from the Gronwall Lemma and integration over~$(0,T)$ we obtain
 \begin{align}\notag
&\norm{\varkappa_k}{L^2((0,T),\rmD(A))}^2 +\norm{\varkappa_k}{L^\infty((0,T),V)}^2
\le \norm{\varkappa_\circ}{V}^2+
\norm{g_k-Q\varkappa_k-k(y_k-\psi)^+)}{L^2((0,T),L^2)}^2.
\end{align}
Finally, we can conclude the proof by using Lemmas~\ref{L:weak-parab1k} and~\ref{L:weak-viol},
and recalling the identities in~\eqref{kappa0gk}.
\end{proof}

In Lemma~\ref{L:strong-smooth1k} we require the extra regularity 
for the initial condition in order to have strong solutions for
the parabolic equation. This extra requirement is needed due to the compatibility
conditions mentioned in Remark~\ref{R:compcond_y0}.
However, due to the smoothing property of parabolic equations, it
turns out that for strictly positive time~$t>0$ we will have that~$y_k(t)\in V$ when~$y_\circ\in H$. 
This fact is explored in the following result.

\begin{lemma}\label{L:strong-smooth1k}
 Let Assumptions~\ref{A:char-dom}--\ref{A:obst-lb} hold true and let~$y_k$ solve~\eqref{sys:MYyk}.
 Then, it follows that
\begin{align}
&\norm{r y_k}{L^2((0,T),H^2)}^2 +\norm{ry_k}{L^\infty((0,T),H^1)}^2
+\norm{\tfrac{\rmd}{\rmd t}(ry_k)}{L^2((0,T),L^2)}^2
\notag\\
&\hspace{2em}\le  \overline C \left(\norm{y_\circ}{L^2}^2+\norm{r\fkE\chi}{\clW_{(0,T)}}^2
+\norm{rf_k}{L^2((0,T),L^2)}^2 +\norm{r\psi}{\clW_{(0,T)}}^2
\right),\notag
\end{align}
with a constant~$\ovlineC{T,C_Q }$ independent of~$k$.
\end{lemma}
\begin{proof}
Multiplying the dynamics in~\eqref{sys:MYkappak} by~$2r^2A\varkappa_k$,  it follows that
 \begin{align}\notag
\tfrac{\rmd}{\rmd t} \norm{r\varkappa_k}{V}^2
-(\tfrac{\rmd}{\rmd t}r^2)\norm{\varkappa_k}{V}^2
+2\norm{r\varkappa_k}{\rmD(A)}^2
=2(rg_k-rQ\varkappa_k -rk(y_k-\psi)^+,rA\varkappa_k)_{L^2}.
\end{align}
Then, the Young inequality together with~$\max\{\norm{r}{L^\infty(\bbR_+)},\norm{\dot r}{L^\infty(\bbR_+)}\}=1$
give us
 \begin{align}\notag
 &\norm{r\varkappa_k}{\rmD(A)}^2+\tfrac{\rmd}{\rmd t} \norm{r\varkappa_k}{V}^2
 \le\norm{g_k-Q\varkappa_k-k(y_k-\psi)^+)}{L^2}^2+\norm{r\varkappa_k}{V}^2,
\end{align}
 and from the Gronwall Lemma and integration over~$(0,T)$ we obtain
 \begin{align}\notag
&\norm{r\varkappa_k}{L^2((0,T),\rmD(A))}^2 +\norm{r\varkappa_k}{L^\infty((0,T),V)}^2\le 
\norm{g_k-Q\varkappa_k-k(y_k-\psi)^+)}{L^2((0,T),L^2)}^2.
\end{align}

Further we have that
\[
 \norm{\tfrac{\rmd}{\rmd t}(r\varkappa_k)}{L^2}^2=
 \norm{Ar\varkappa_k+Qr\varkappa_k+rk(\varkappa_k-\phi)^+-rg_k-(\dot r)\varkappa_k}{L^2}^2.
\]
We can conclude the proof by 
using~$ry_k=r\varkappa_k+r\fkE\chi$, \eqref{kappa0gk}, and
Lemmas~\ref{L:weak-parab1k} and~\ref{L:weak-viol}. 
\end{proof}

We are now ready to conclude the proof of Theorem~\ref{T:MY-approx}.
\begin{proof}[Proof of Theorem~\ref{T:MY-approx}] 
 \textbf{Existence:} From Lemmas~\ref{L:weak-parab1k}  and~\ref{L:strong-smooth1k}, there exists a
 subsequence~$y_{\fkn(k)}$ of~$y_k$,  such that the following weak limits hold
\begin{subequations}\label{MY:wklimitsVI}
\begin{align}
&y_{\fkn(k)}-\fkE\chi\xrightharpoonup[L^2((0,T), V)]{} y-\fkE\chi,\quad&&
\dot y_{\fkn(k)}-\tfrac{\rmd}{\rmd t}\fkE\chi\xrightharpoonup[L^2((0,T), V')]{}
\dot y-\tfrac{\rmd}{\rmd t}\fkE\chi,\\
&r(y_{\fkn(k)}-\fkE\chi)\xrightharpoonup[L^2((0,T), \rmD(A))]{} z,\quad &&
\tfrac{\rmd}{\rmd t}(r(y_{\fkn(k)}-\fkE\chi))\xrightharpoonup[L^2((0,T), L^2)]{} \dot  z,
 \end{align}
 \end{subequations} 
for suitable~$y\in W((0,T), H^1,V')$ and~$z\in W((0,T), \rmD(A),L^2)$.
Necessarily we have $z=r(y-\fkE\chi)$ and the strong limits
\begin{subequations}\label{MY:stlimitsVI}
\begin{align}
&y_{\fkn(k)}\xrightarrow[L^2((0,T), L^2)]{}  y,&&  r(y_{\fkn(k)}-\fkE\chi)
\xrightarrow[L^2((0,T), V)]{}  r(y-\fkE\chi),\\
& &&r(y_{\fkn(k)}-\fkE\chi)\xrightarrow[\clC({[0,T]}, L^2)]{}  r(y-\fkE\chi),
 \end{align}
\end{subequations} 
where we have used, in particular the Aubin-Lions-Simon Lemma~\cite[Sect.~8, Cor.~4]{Simon87}.

For the sake of simplicity, let us still denote the subsequence~$y_{\fkn(k)}$ by~$y_k$.
By Lemma~\ref{L:weak-parab1k}, it follows that~$(k^2\norm{(y_k-\psi)^+}{L^2((0,T), L^2)}^2)_{k\in\bbN}$
is bounded, thus
\begin{equation}\notag
        \norm{(y-\psi)^+}{L^2((0,T), L^2)}^2
        =\lim_{k\to+\infty}\norm{(y_{k}-\psi)^+}{L^2((0,T), L^2)}^2
	=0
\end{equation}
 and, since $y\in L^2((0,T);H^1)$, we obtain that~$y\in\bfC^\psi_T$, see~\eqref{convex}.\black
Now, for an arbitrary~$v\in\bfC^\psi_T$, 
we find, for almost every~$t\in(0,T)$,
 \begin{align}
&\Bigl(r\left(\tfrac{\p}{\p t} y_{k } +(-\Delta+\Id)y_{k}
+Qy_{k }-f_{k }\right),r (v-y_{k })\Bigr)_{L^2}\notag\\
&\hspace{3em}
=-{k}\left(r(y_{k}-\psi)^+,r(v-y_{k })\right)_{L^2}\notag\\
&\hspace{3em}
={ k}\left((y_{ k }-\psi)^+,r^2(y_{k}-\psi)\right)_{L^2}
+{ k }\left((y_{ k }-\psi)^+,r^2(\psi-v)\right)_{L^2},\notag
\end{align}
which gives us
\begin{align}
&\Bigl(r\left(\tfrac{\p}{\p t} y_{ k } +(-\Delta+\Id)y_{ k }
+Qy_{ k }-f_{ k }\right),r(v-y_{ k })\Bigr)_{L^2}\ge0,\label{VIk>0}
\end{align}
because~$r^2{ k }(y_{ k }-\psi)^+(y_{ k }-\psi)\ge0$
and~$r^2{ k }(y_{ k }-\psi)^+(\psi-v)\ge0$, due  to~$v\in\bfC^\psi_T$. 

Observe that, with~$q_{ k }\coloneqq r(y_k-\fkE\chi)$ and~$q\coloneqq r(y-\fkE\chi)$,
 for the left-factor in~\eqref{VIk>0}, we find 
\begin{align}
 &r\left(\tfrac{\p}{\p t} y_{ k } +(-\Delta+\Id)y_{ k }
+Qy_{ k }-f_{ k }\right)\notag\\
&\hspace{4.25em}=\dot q_k +Aq_{ k }
+Qq_{ k }-rf_{ k }
+r(\tfrac{\rmd}{\rmd t}+A+Q)\fkE\chi -(\dot r)(y_{ k }-\fkE\chi),\notag
\end{align}
and we have the weak limit in $L^2((0,T),L^2)$ given by
\begin{align}\notag
\dot q +Aq
+Qq-rf
+r(\tfrac{\rmd}{\rmd t}+A+Q)\fkE\chi -(\dot r)(y-\fkE\chi)=r\left(\tfrac{\p}{\p t} y +(-\Delta+\Id)y
+Qy-f\right)
\end{align}
and also the strong limit for the right-factor in~\eqref{VIk>0} as follows
\begin{align}\notag
q_{ k }\xrightarrow[L^2((0,T),L^2)]{}q.
\end{align}
These limits allow us to take the limit for the integrated product in~\eqref{VIk>0}, and obtain
 \begin{align}
&\int_0^T\Bigl(r\left(\tfrac{\p}{\p t} y +(-\Delta+\Id)y
+Qy-f\right),r(v-y)\Bigr)_{L^2}\,\rmd t\notag\\
&\hspace{3em}=\lim_{k\to+\infty}\int_0^T\left(r(\tfrac{\p}{\p t} y_{ { k } }
+(-\Delta+\Id)y_{ { k } }
+Qy_{ { k} }-f_{ { k } }),r(v-y_{ { k } })\right)_{H}\,\rmd t\notag\\
&\hspace{3em}\ge0,\qquad\mbox{for all}\quad v\in\bfC^\psi_T.\label{eq:niVI}
 \end{align}
Let us fix arbitrary~$v\in\bfC^\psi_T$, ~$\overline t\in(0,T)$,
$\delta\in(0,\min\{\overline t,T-\overline t\})$. 
Note that the integrand~$\xi_v\coloneqq\Bigl(r\left(\tfrac{\p}{\p t} y +(-\Delta+\Id)y
+Qy-f\right),r(v-y)\Bigr)_{L^2}$ is an integrable function, $\xi_v\in L^1(0,T)$.
By the Lebesgue differentiation theorem~\cite[Ch.~7, Thm.~7.7]{Rudin87},
the set of Lebesgue points
\[
\fkL_{v}\coloneqq\left\{t^*\in(0,T)\mid \xi_h(t^*)
=\lim_{\delta\searrow0}\frac1{2\delta}{\int_{t^*-\delta}^{t^*+\delta}}\xi_v(t)\,\rmd t\right\},
\]
has full measure. 
{ {We define the functions
\[
v_{\overline t,\delta}\coloneqq\begin{cases}
v,\quad&\mbox{if }t\in(\overline t-\delta,\overline t+\delta)\\
y,\quad&\mbox{if }t\in(0,\overline t-\delta)\bigcup(\overline t+\delta,T).
\end{cases}
\] 
We have}}
$
v_{\overline t,\delta}(t,x)\in\bfC^\psi_T.
$
From~\eqref{eq:niVI}, it follows that
\[
{\int_{\overline t-\delta}^{\overline t+\delta}}\xi_v(t)\,\rmd t
=\int_0^T\Bigl(r\left(\tfrac{\p}{\p t} y +(-\Delta+\Id)y
+Qy-f\right),r(v_{ \overline t,\delta}-y)\Bigr)_{L^2}(t)\,\rmd t\ge0
\]
and as a consequence we have
 \begin{align}\notag
&\Bigl(r\left(\tfrac{\p}{\p t} y +(-\Delta+\Id)y
+Qy-f\right),r(v-y)\Bigr)_{L^2}(t^*)\ge0,\quad\mbox{for all}\quad t^*\in 
\fkL_{v},
 \end{align}
which implies the inequality in~\eqref{eq:limMYyk}, because $r^2=\min\{t^2,1\}>0$ for time~$t>0$.

\textbf{Uniqueness:} Let us assume
that~$w\in\bfC^\psi_T\bigcap W((0,T),H^1,V')$, with~$rw\in W((0,T),H^2,L^2)$
also satisfies~\eqref{eq:limMYyk}.
In this case we find the relations
\begin{align}\notag
&\left(\dot y +(-\Delta+\Id)y
+Qy-f, w-y\right)_{L^2}\ge0,\qquad
\left(\dot w +(-\Delta+\Id)w
+Qw-f,y-w\right)_{L^2}\ge0,
\end{align}
which lead us to, with~$z\coloneqq y-w$, 
\begin{align}\notag
&\left(\dot z +Az+Qz,z\right)_{L^2}\le0,\quad\mbox{for almost all}\quad t\in(0,T),
\quad z(0)=0,
\end{align}
 with~$z(t)\in V$ for all~$t\in[0,T]$.
Thus
\begin{align}
&\tfrac{\rmd}{\rmd t}\norm{z}{L^2}^2+ 2\norm{z}{V}^2
\le 2C_Q \norm{z}{H^1}\norm{z}{L^2}\le\norm{z}{V}^2 + C_Q ^2\norm{z}{L^2}^2,
\end{align}
and the uniqueness follows from Gronwall's Lemma.

\textbf{Convergence:} Finally we show that the strong limits in~\eqref{MY:stlimitsVI} hold for the (entire) sequence~$y_k$.
We argue by contradiction. Let us denote~$\bbS\coloneqq \{L^2((0,T),V),\clC([0,T],L^2)\}$.
\begin{align}\label{assu.contra.seq}
\mbox{Suppose that } r(y_{k}-\fkE\chi)\xrightarrow[\quad\clS\quad]{}  r(y-\fkE\chi) \mbox{ does not hold,
for some }\clS\in\bbS.
\end{align}
Under assumption~\eqref{assu.contra.seq}, there would exist~$\varepsilon>0$ and a
subsequence~$y_{\fks_1(k)}$ of~$y_{k}$ such that
\begin{equation}\label{contra.seq.s1}
 \norm{r(y_{\fks_1(k)}-\fkE\chi)-  r(y-\fkE\chi)}{\clS}\ge\varepsilon.
\end{equation}
However since~$\{\overline y_{k}\}\coloneqq\{y_{\fks_1(k)}\}$ is a subsequence of~$\{y_{k}\}$ we would be able to
follow the arguments above and arrive to analogous limits as in~\eqref{MY:wklimitsVI} and~\eqref{MY:stlimitsVI},
for a suitable subsequence~$\{\overline y_{\fks_2(k)}\}$ of~$\{\overline y_{k}\}$ and a
limit~$\overline y$ in the place of~$y$.
In particular, we would arrive to
\begin{align}\notag
&y_{\fks_2(\fks_1(k))}\xrightarrow[\quad\clS\quad]{}  \overline y,
 \end{align}
where moreover~$\overline y$ solves~\eqref{eq:limMYyk}. By~\eqref{contra.seq.s1} we would have that~$\overline y\ne y$,
which contradicts the uniqueness of the solution proven above. That is, the assumption in~\eqref{assu.contra.seq}
leads us to a contradiction. Therefore,
we can conclude that~\eqref{MY:stlimitsVI-full} holds true.
The proof is finished.
 \end{proof}

\section{Stabilization of a sequence of parabolic equations}\label{sec:stabilization_sequ}
The solution of~\eqref{sys-haty} can
be approximated by the sequence $({ y}_k)_{k\in\mathbb N}$ as stated in Theorem~\ref{T:MY-approx}, where~${ y}_k$ solves
\begin{subequations}\label{sys-yk}
\begin{align}
&\tfrac{\p}{\p t}{ y}_k -\nu\Delta { y}_k+a y_k+b\cdot\nabla y_k+k({ y}_k-\psi)^+=f,\\
&{y}_k(0)={ y}_\circ ,\quad\clG  y\rest{\Gamma}=\chi.
\end{align}
\end{subequations}
This follows from Theorem~\ref{T:MY-approx} with~$Q=a\Id+b\cdot\nabla$, and~$f_k=f$.

Here, we investigate the stabilizability to trajectories for system~\eqref{sys-yk}.
We consider the sequence~$({ w}_k)_{k\in\mathbb N}$, where~${ w}_k$ solves
\begin{subequations}\label{sys-wk}
\begin{align}
&\tfrac{\p}{\p t}{ w}_k -\nu\Delta { w}_k+a w_k+b\cdot\nabla w_k+k({ w}_k-\psi)^+
=f-\lambda P_{\clU_M}^{\clE_M^\perp}AP_{\clE_M}^{\clU_M^\perp}(w_k-y_k),\\
&{ w}_k(0)= w_\circ,\quad\clG w\rest{\Gamma}=\chi,
\end{align}
\end{subequations}
where~$P_{\clE_M}^{\clU_M^\perp}\in\clL(  L^2)$ and~$P_{\clE_M}^{\clU_M^\perp}\in\clL(  L^2)$ are suitable oblique
projections in~$  L^2$, which we shall construct so
that~$P_{\clU_M}^{\clE_M^\perp}AP_{\clE_M}^{\clU_M^\perp}\in\clL(  L^2)$. 
Then again from Theorem~\ref{T:MY-approx}, with~$Q
=a\Id+b\cdot\nabla+\lambda P_{\clU_M}^{\clE_M^\perp}AP_{\clE_M}^{\clU_M^\perp}$,
and~$f_k=f+\lambda P_{\clU_M}^{\clE_M^\perp}AP_{\clE_M}^{\clU_M^\perp}y_k$,
it follows that the solution of~\eqref{sys-tildey-u} can
be approximated by the sequence~$({ w}_k)_{k\in\mathbb N}$.  At this point, it is important to underline that
the triple~$(\lambda,{\clU_M},{\clE_M})$ can be chosen independently of~$k$, as we shall show later on.

In this section we will see~$ y_k$ as our target solution and consider the
difference~${z}_k \coloneqq w_k- y_k$ from the controlled solution~$w_k$ to the target.
With initial condition~${z}_\circ \coloneqq w_\circ- y_\circ$, we find that~${z}_k$ satisfies
\begin{subequations}\label{sys-zk}
 \begin{align}
&\tfrac{\p}{\p t}{z}_k  -\nu\Delta {z}_k +a {z}_k +b\cdot\nabla {z}_k +k\left(({z}_k +{y}_k-\psi)^+-({y}_k-\psi)^+\right)
=-\lambda P_{\clU_M}^{\clE_M^\perp}AP_{\clE_M}^{\clU_M^\perp}z_k,\\
&{z}_k (0)={z}_\circ ,\qquad\clG   {z}_k\rest{\Gamma}=0.
\end{align}
\end{subequations}

 For a given $\mu>0$,  our goal here, see ~\eqref{goal-intro}, is to find   a scalar~$\lambda>0$,
a space of actuators~${\clU_M}$, and an auxiliary
 space~${\clE_M}$,  such that
\begin{equation}\label{goal-intro-k}
 \left| w_k(t)- y_k(t)\right|_{L^2}\le C{\rm e}^{-\mu t}\left| w_\circ- y_\circ\right|_{L^2},
 \quad\mbox{for all}\quad(w_\circ,y_\circ)\in L^2\times L^2,\quad t\ge0
\end{equation}
 for a suitable~$C\ge1$.

 \subsection{The oblique projections}\label{sS:oblProj}
We specify here how we can appropriately choose the spaces of actuators~$\clU_{M}$ and auxiliary
eigenfunctions~$\clE_{M}$,  so that the feedback
operator~$-\lambda P_{\clU_M}^{\clE_M^\perp}AP_{\clE_M}^{\clU_M^\perp}$ is stabilizing for large enough~$\lambda>0$. 
Since the stabilization results will hold for large enough~$M$, we will rather consider a sequence of
pairs of subspaces~$(\clU_{M},\clE_M)_{M\in\bbN}$ as in~\eqref{seqActEig}.

In the one-dimensional case, $\Omega^1=(0,L_1)\subset\bbR$, $L_1>0$, as actuators we take
the indicator functions~$1_{\omega_j^1}(x_1)$, $j\in\{1,2,\dots,M\}$,
defined as follows,
\begin{equation}\label{Act-mxe}
 1_{\omega_j^1}(x_1)\coloneqq\begin{cases}
                            1,&\mbox{if }x_1\in\Omega^1\textstyle\bigcap\omega_j^1,\\
                            0,&\mbox{if }x_1\in\Omega^1\setminus\overline{\omega_j^1},
                           \end{cases}
\quad\! \omega_j^1\coloneqq(c_j-\tfrac{rL_1}{2M},c_j\!+\!\tfrac{rL_1}{2M}),\quad\! c_j\coloneqq\tfrac{(2j-1)L_1}{2M}.
\end{equation}
As eigenfunctions we take the first~$M$ eigenfunctions~$e_j^1$ of~$-\nu\Delta+\Id\colon \rmD(A)\to L^2(\Omega^1)$
(i.e., the first eigenfunctions of~$\Delta$),
\begin{equation}\label{firsteigs}
(-\nu\Delta+\Id)e_j^1=\alpha_j^1e_j^1,\quad\clG e_i^1\rest\Gamma=0,\qquad j\in\{1,2,\dots,M\},
\end{equation}
where the~$\alpha_j^1$s are the ordered eigenvalues, repeated accordingly to their multiplicity,
\[
0<1\le \alpha_1^1< \alpha_2^1< \dots<\alpha_j^1<\alpha_{j+1}^1< \dots,\qquad j\in\bbN.
\]

In the higher-dimensional case, for nonempty rectangular
domains~$\Omega^\times=\prod\limits_{n=1}^{d}(0,L_n)\subset\bbR^{d}$, $L_n>0$  we take
cartesian product actuators of the above actuators~$1_{\omega_j^n}$ and eigenfunctions~$e_j^n$ as follows.
 We define~$\bbM\coloneqq\{1,2,\dots,M\}$ and take 
\begin{align}\label{Act-mxe-cart}
\clU_{M}=\linspan\{1_{\omega_{\mathbf j}^\times}\mid {\mathbf j}\in\bbM^{d}\}&\quad\mbox{and}\quad
\clE_{M}=\linspan\{e_{\mathbf j}^\times\mid {\mathbf j}\in\bbM^{d}\},
\end{align}
and $\omega_{\bf j}^\times\coloneqq\{(x_1,x_2,\dots,x_{ d})\in\Omega^\times\mid x_n\in\omega_{\mathbf j_n}^n\}$
and $e_{\mathbf j}^\times(x_1,x_2,\dots,x_{\mathbf d})\coloneqq\textstyle\prod\limits_{n=1}^{ d}e_{\mathbf j_n}^n(x_n)$.
Notice that we can also
write~$1_{\omega_{\mathbf j}^\times}=\textstyle\prod\limits_{n=1}^{ d}1_{\omega_{\mathbf j_n}^n}(x_n)$.

In particular, by setting the eigenvalue
\begin{subequations}\label{bound.Proj.ratLam}
 \begin{align}
 \widehat\alpha_{M}&\coloneqq \max\{\alpha_i\mid \mbox{there is }\phi\in\clE_{M}
 \mbox{ such that } A\phi=\alpha_i\phi\}, 
 \intertext{and the Poincar\'e-like constant}
 \beta_{M_+}&\coloneqq \min\left\{\left.\tfrac{\norm{h}{V}}{\norm{h}{L^2}}
 \;\right|\; h\in \clU_{M}^\perp{\textstyle\bigcap}V,\quad h\ne0\right\}, 
\end{align}
we have
\begin{align}
 & L^2=\clU_{M}\oplus\clE_{M}^\perp,\qquad
 \lim_{M\to+\infty}\beta_{M_{+}}=+\infty,
\end{align}
and also 
 \begin{equation}
 \sup\limits_{M\ge1}\norm{P_{\clU_{M}}^{\clE_{M}^\perp}}{\clL(L^2)}
   \eqqcolon C_P< +\infty.
 \end{equation}
\end{subequations}

See~\cite[Sect.~2.2]{Rod20-eect} and~\cite[Sect.~5]{Rod21-sicon}  for more details.
For the one-dimensional case we refer to~\cite[Thms.~4.4 and~5.2]{RodSturm20}, for higher-dimensional rectangular
domains see~\cite[Sect.~4.8.1]{KunRod19-cocv}.

\begin{remark}
 For nonrectangular domains~$\Omega\subset\bbR^{d}$, with~${d}\ge2$, we still not
know whether we can choose the actuators (as indicator functions) so that the properties in~\eqref{bound.Proj.ratLam}
are satisfied. So we cannot guarantee that an oblique projection based feedback will stabilize our system.
In spite of this fact, we refer the reader to~\cite{KunRod19-cocv,KunRod19-dcds}, where  numerical simulations show the 
the stabilizing performance of such a  feedback 
for equations evolving in a spatial nonrectangular domain.
\end{remark}

 \subsection{On the nonlinearity} We gather key properties of the  nonlinear
operator in~\eqref{sys-zk}.
\begin{equation}\label{NNkdef}
\mathcal N_k(z)\in\clC(L^2,L^2),\qquad \mathcal N_k(z)\coloneqq k\left((z+{y}_k-\psi)^+-({y}_k-\psi)^+\right).
\end{equation}
\begin{lemma}\label{L:bdNN}
The nonlinear operator~\eqref{NNkdef} is bounded, as
 \begin{align}\notag
 &\norm{\mathcal N_k(z_1)-\mathcal N_k(z_2)}{L^2}\le k\norm{z_1-z_2}{L^2},\quad\mbox{for all}
 \quad (z_1,z_2)\in L^2\times L^2.
\end{align}
\end{lemma}
\begin{proof}
 With~$(z_1,z_2)\in L^2\times L^2$, we find that
 \begin{align}\label{DiffNN}
 \mathcal N_k(z_1)-\mathcal N_k(z_2)\coloneqq k\left((z_1+{y}_k-\psi)^+-(z_2+{ y }_k-\psi)^+\right). 
 \end{align}
 
  Note that $h\mapsto h^+=\max(h,0)$ is a globally Lipschitz continuous functions with
  unitary Lipschitz constant, and thus $|h_1^+-h_2^+|_{L^2}\leq |h_1-h_2|_{L^2}$
  for all $h_1,h_2\in {L^2}$. Therefore,
  \begin{align}\notag
 |\mathcal N_k(z_1)-\mathcal N_k(z_2)|_{L^2}
 \leq k|(z_1+{ y }_k-\psi)-(z_2+{ y }_k-\psi)|_{L^2}= k|z_1-z_2|_{L^2},
 \end{align}
which finishes the proof.
 \end{proof}

\begin{lemma}\label{L:monot}
The nonlinear operator~\eqref{NNkdef} is monotone, 
 \begin{align}\notag
   & (\mathcal N_k(z_1)-\mathcal N_k(z_2),z_1-z_2)_{L^2}\ge 0,\qquad\mbox{for all}
 \quad (z_1,z_2)\in {L^2}\times {L^2}.
 \end{align} 
\end{lemma}
\begin{proof}
 Note that $z\mapsto G(z):=z^+$ is monotone in $L^2(\Omega)$. Hence, $z\mapsto G(z-\zeta_1)-\zeta_2$ is also monotone for arbitrary $\zeta_1$ and $\zeta_2$ in $L^2(\Omega)$,
which finishes the proof. 
\end{proof}

\subsection{Stabilizability result}
For simplicity, let us denote
\begin{align}
& A_{\rm rc}\coloneqq a\Id+b\cdot\nabla,\qquad C_{\rm rc}\coloneqq \norm{A_{\rm rc}}{L^\infty(\bbR_+,\clL(V,  L^2))},\notag\\
&\clK_{M}^{\lambda}\coloneqq-\lambda P_{\clU_{M}}^{\clE_{M}^\perp}A P_{\clE_{M}}^{\clU_{M}^\perp}.\label{lamK_M}
\end{align}
\begin{theorem}\label{T:stab-appN}
Let Assumptions~\ref{A:char-dom}--\ref{A:obst-lb} hold true, with~$\clB=a\Id$. Let the
sequence~$(\clU_M,\clE_M)_{M\in\bbN}$ be constructed as in Section~\ref{sS:oblProj}. Then, for every given~$\mu>0$, there are large enough constants~$\lambda>0$ and~$M\in\bbN$ such that, for every~$k\in\bbN$,
the system
 \begin{align}
&\dot{z}_k  + A{z}_k +A_{\rm rc}{z}_k +\mathcal N_k({z}_k )
=\clK_{M}^{\lambda}{z}_k ,
\qquad	{z}_k (0)={z}_\circ ,
\label{sys-y-P-k}\stepcounter{equation}\tag{\theequation$[k]$}
\end{align}
is exponentially stable with rate~$-\mu$.
For all~${z}_\circ \in L^2$, the solution satisfies
\begin{equation}\label{expyk}
 \norm{{z}_k (t)}{L^2}\le\ex^{-\mu (t-s)}\norm{{z}_k (s)}{L^2},\qquad t\ge s\ge0.
\end{equation}
Moreover, the feedback operator~$\clK_{M}^{\lambda}$ and control input~$\clK_{M}^{\lambda}{z}_k$ satisfy
the estimate
\begin{align}\label{normFeedz_k}
\norm{\clK_M^\lambda}{\clL(L^2)}
&\le\lambda\widehat\alpha_MC_P^2\quad\mbox{and}\quad
 \norm{\clK_M^\lambda{z}_k }{L^2(\bbR_+,L^2)}
\le\lambda\widehat\alpha_M\mu^{-1} C_P^2\norm{{z}_\circ }{L^2}.
 \end{align}
where~$\widehat\alpha_M$ and~$C_P$ are as in~\eqref{bound.Proj.ratLam}. Furthermore, we can choose
\begin{equation}\label{choice-Mk}
\lambda\sim\ovlineC{\mu, C_{\rm rc}}\quad\mbox{and}\quad M\sim\ovlineC{\mu, C_{\rm rc}}.
 \end{equation}
\end{theorem}

\begin{remark}
Note that the feedback operator~$\clK_{M}^{\lambda}$ in~\eqref{lamK_M} is independent of~$(k,\psi)$,
because the pair~$(\lambda,M)$ in~\eqref{choice-Mk}
can be chosen independently of~$(k,\psi)$. The upper bound in~\eqref{normFeedz_k} for the norm of the
control input~$\clK_{M}^{\lambda}{z}_k$ is also independent of~$(k,\psi)$.  The monotonicity stated in\black Lemma~\ref{L:monot} plays a key role on
such independences on~$k$.
\end{remark}

\begin{remark}
Inequality~\eqref{expyk} implies that~$t\mapsto\norm{{z}_k (t)}{L^2}^2$ is strictly
decreasing at time~$t=s$, if~$\norm{{z}_k (s)}{L^2}^2>0$.
Of course, if~$\norm{{z}_k (s)}{L^2}^2=0$ then~$\norm{{z}_k (t)}{L^2}^2=0$
for all~$t\ge0$, see \cite[Sect.~4]{Rod21-sicon}.
\end{remark}

\begin{proof}[Proof of Theorem~\ref{T:stab-appN}]
Following the arguments in~\cite[Sect.~4]{Rod21-sicon},
we decompose the solution of system~\eqref{sys-y-P-k} into oblique components as
\[
 {z}_k =\theta_k+\varTheta_k,\qquad\mbox{with}\qquad \theta_k\coloneqq P_{\clE_{M}}^{\clU_{M}^\perp}{z}_k 
 \quad\mbox{and}\quad \varTheta_k\coloneqq P_{\clU_{M}^\perp}^{\clE_{M}}{z}_k .
\]

Observe that form~\eqref{sys-y-P-k}, Lemma~\ref{L:monot}, and the Young inequality, we obtain that
\begin{align}
\tfrac{\ed}{\ed t}\norm{{z}_k }{L^2}^2&=-2\norm{{z}_k }{V}^2-2\langle A_{\rm rc}{z}_k ,{z}_k \rangle_{V',V}
-2\left(\mathcal N_k({z}_k ),{z}_k \right)_{L^2}
+2\left(\clK_{M}^{\lambda}{z}_k ,{z}_k \right)_{L^2}\label{est-Nk1}\\
&\le -2\norm{{z}_k }{V}^2-2\langle A_{\rm rc}{z}_k ,{z}_k \rangle_{V',V}
-2\lambda\left(A \theta_k,\theta_k\right)_{L^2}\label{est-Nk0}\\
&\le -2\norm{{z}_k }{V}^2+\gamma_1\norm{{z}_k }{V}^2+\gamma_1^{-1} C_{\rm rc}^2\norm{{z}_k }{L^2}^2
 -2\lambda\norm{\theta_k}{V}^2,\notag\\
 &\le -(2-\gamma_1)\norm{{z}_k }{V}^2+\gamma_1^{-1} C_{\rm rc}^2\norm{{z}_k }{L^2}^2
 -2\lambda\norm{\theta_k}{V}^2,\quad \mbox{for all}\quad  \gamma_1>0.\label{est_norm1}
\end{align}

Now we observe that, by the young inequality, we obtain for all~$\gamma_2>0$
\begin{align}
-\norm{{z}_k }{V}^2&= -\norm{\varTheta_k+\theta_k}{V}^2
=-\norm{\varTheta_k}{V}^2-\norm{\theta_k}{V}^2-2(\varTheta_k,\theta_k)_V\notag\\
&\le-\norm{\varTheta_k}{V}^2-\norm{\theta_k}{V}^2+\gamma_2\norm{\varTheta_k}{V}^2+\gamma_2^{-1}\norm{\theta_k}{V}^2
=-(1-\gamma_2)\norm{\varTheta_k}{V}^2-(1-\gamma_2^{-1})\norm{\theta_k}{V}^2.\label{est_norm2}
\end{align}
Combining~\eqref{est_norm1} and~\eqref{est_norm2} we obtain, for all~$(\gamma_1,\gamma_2)\in (0,2)\times\bbR_+$,
 \begin{align}
\tfrac{\ed}{\ed t}\norm{{z}_k }{L^2}^2
&\le -(2-\gamma_1)(1-\gamma_2)\norm{{\varTheta}_k}{V}^2
-\left(2\lambda+(2-\gamma_1)(1-\gamma_2^{-1})\right)\norm{{\theta}_k}{V}^2
+\gamma_1^{-1} C_{\rm rc}^2\norm{{z}_k }{L^2}^2\notag\\
&\le -(2-\gamma_1)(1-\gamma_2)\norm{{\varTheta}_k}{V}^2
-\left(2\lambda-(2-\gamma_1)(\gamma_2^{-1}-1)\right)\norm{{\theta}_k}{V}^2
+2\gamma_1^{-1} C_{\rm rc}^2(\norm{{\varTheta}_k }{L^2}^2+\norm{{\theta}_k }{L^2}^2)\notag
\end{align}
Now, we can choose~$\gamma_1\in(0,2)$ and~$\gamma_2\in(0,1)$, and
$\lambda$ satisfying~$2\lambda-(2-\gamma_1)(\gamma_2^{-1}-1)>0$.
For such choices,  using~\eqref{bound.Proj.ratLam}, we find
 \begin{align}
\tfrac{\ed}{\ed t}\norm{{z}_k }{L^2}^2
&\le -(2-\gamma_1)(1-\gamma_2)\beta_{M_+}\norm{{\varTheta}_k}{L^2}^2
-\left(2\lambda-(2-\gamma_1)(\gamma_2^{-1}-1)\right)\alpha_1\norm{{\theta}_k}{L^2}^2\notag\\
&\quad+2\gamma_1^{-1} C_{\rm rc}^2(\norm{{\varTheta}_k }{L^2}^2+\norm{{\theta}_k }{L^2}^2)\notag\\
&\le -\Xi_1(M)\norm{{\varTheta}_k}{V}^2-\Xi_2(M)\norm{{\theta}_k}{V}^2,\label{est_norm3}
\end{align}
where~$\alpha_1\coloneqq\min\left\{\left.\tfrac{\norm{h}{V}}{\norm{h}{L^2}}\;\right|\; h\in V\setminus\{0\}\right\}$,
and
\begin{subequations}\label{Xis}
\begin{align}
 \Xi_1(M)&\coloneqq(2-\gamma_1)(1-\gamma_2)\beta_{M_+}-2\gamma_1^{-1} C_{\rm rc}^2,\\
 \Xi_2(\lambda)&\coloneqq\left(2\lambda-(2-\gamma_1)(\gamma_2^{-1}-1)\right)\alpha_1-2\gamma_1^{-1} C_{\rm rc}^2,
\end{align}
\end{subequations}
\begin{subequations}\label{choice.gammalamM}
 Recall that,  due to~\eqref{bound.Proj.ratLam} we have that~$\lim\limits_{M\to+\infty}\beta_{M_+}=+\infty$.
Let us be given an arbitrary given~$\mu>0$ and let us choose~$\gamma_1$ and~$\gamma_2$ as above, satisfying
 \begin{align}
 &\gamma_1\in(0,2)\quad\mbox{and}\quad \gamma_2\in(0,1).
 \intertext{Then, subsequently we can choose~$\lambda>0$ and~$M\in\bbN$ large enough satisfying}
 &2\lambda-(2-\gamma_1)(\gamma_2^{-1}-1)>0,\quad \Xi_2(\lambda)\ge 4\mu,\quad\mbox{and}\quad
 \Xi_1(M)\ge 4\mu.
\end{align}
\end{subequations}

Form~\eqref{est_norm3}, with the choices in~\eqref{choice.gammalamM}, we arrive at
 \begin{align}
\tfrac{\ed}{\ed t}\norm{{z}_k }{L^2}^2
&\le -4\mu\left(\norm{{\varTheta}_k}{L^2}^2+\norm{{\theta}_k}{L^2}^2\right)
\le -2\mu\norm{{z}_k }{L^2}^2,\label{est_norm4}
\end{align}
which implies~\eqref{expyk}.

It remains to show the boundedness of the feedback control, with~$(\gamma_1,\gamma_2,\lambda,M)$
as in~\eqref{choice.gammalamM}.

We see that~$P_{\clU_{M}}^{\clE_{M}^\perp}=P_{\clU_{M}}^{\clE_{M}^\perp}P_{\clE_{M}}$, because
$P_{\clU_{M}}^{\clE_{M}^\perp} h=P_{\clU_{M}}^{\clE_{M}^\perp}(P_{\clE_{M}}h+P_{\clE_{M}^\perp}h)
=P_{\clU_{M}}^{\clE_{M}^\perp} P_{\clE_{M}}h$,
for all~$h\in L^2$. Here~$P_{\clE_{M}}\coloneqq P_{\clE_{M}}^{\clE_{M}\perp}$ stands for the orthogonal
projection in~$L^2$ onto~${\clE_{M}}$.
Using~\eqref{expyk}
we obtain that the feedback operator~$\clK_M^\lambda$
satisfy
\begin{subequations}\label{bound.controlM}
\begin{align}
\norm{\clK_M^\lambda}{\clL(L^2)}
&=\lambda\norm{P_{\clU_{M}}^{\clE_{M}^\perp}AP_{\clE_{M}}^{\clU_{M}^\perp}}{\clL(L^2)}
=\lambda\norm{P_{\clU_{M}}^{\clE_{M}^\perp}P_{\clE_{M}}A
P_{\clE_{M}}P_{\clE_{M}}^{\clU_{M}^\perp}}{\clL(L^2)}\notag\\
&\le \lambda\norm{P_{\clU_{M}}^{\clE_{M}^\perp}}{\clL(L^2)}
\norm{P_{\clE_{M}}AP_{\clE_{M}}}{\clL(L^2)}
\norm{P_{\clE_{M}}^{\clU_{M}^\perp}}{\clL(L^2)}
\le\lambda\widehat\alpha_M\norm{P_{\clU_{M}}^{\clE_{M}^\perp}}{\clL(L^2)}^2
\end{align}
and corresponding control~$\clK_M^\lambda{z}_k $ 
\begin{align}
 \norm{\clK_M^\lambda{z}_k }{L^2(\bbR_+,L^2)}
 &\le \lambda\widehat\alpha_M \norm{P_{\clU_{M}}^{\clE_{M}^\perp}}{\clL(L^2)}^2\norm{{z}_k }{L^2(\bbR_+,L^2)}
 \le \lambda\widehat\alpha_M \norm{P_{\clU_{M}}^{\clE_{M}^\perp}}{\clL(L^2)}^2\norm{{z}_\circ }{L^2}
 \int_0^{+\infty}\ex^{-\mu t}\,\ed t\notag\\
 &=\lambda\widehat\alpha_M\mu^{-1} \norm{P_{\clU_{M}}^{\clE_{M}^\perp}}{\clL(L^2)}^2\norm{{z}_\circ }{L^2},
 \end{align}
 \end{subequations}
where~$\widehat\alpha_M$ is as in~\eqref{bound.Proj.ratLam}. Finally,
with~$C_P$ is as in~\eqref{bound.Proj.ratLam},
we also obtain the bounds
\begin{equation}\label{bound.control}
\norm{\clK_M^\lambda}{\clL(L^2)}
\le \lambda\widehat\alpha_M C_{P}^2,\quad\mbox{and}\quad
 \norm{\clK_M^\lambda{z}_k }{L^2(\bbR_+,L^2)}
 \le \lambda\widehat\alpha_M\mu^{-1} C_{P}^2\norm{{z}_\circ }{L^2}. 
\end{equation}
 The proof is finished.
 \end{proof}

\section{Stabilization of the variational inequality}\label{sec:stabil-var}
Here we prove the main result, which we can write now in a more precise form as follows.

\begin{theorem}\label{T:main}
Let Assumptions~\ref{A:char-dom}--\ref{A:obst-lb} hold true,  let~$\mu>0$,  and let the pairs~$(\clU_M,\clE_M)$
be constructed as in Section~\ref{sS:oblProj}.
Further let~$y\in W_{\rm loc}(\bbR_+;H^1,V')$ 
with~$ry\in W_{\rm loc}(\bbR_+;H^2,L^2)$
solve~\eqref{sys-haty}. Then for~$M$ and~$\lambda$ large enough
the solution~$w$ of system~\eqref{sys-tildey-K-intro}
satisfies
\begin{equation}\label{goal}
 \left| w(t)- y(t)\right|_{L^2}\le {\rm e}^{-\mu t}\left| w_\circ- y_\circ\right|_{L^2},\quad t\ge0.
\end{equation}
Furthermore, with~$\widehat\alpha_M$ and~$C_P$ as in~\eqref{bound.Proj.ratLam} the control satisfies 
\begin{align}\label{prop.K-traj}
\norm{\clK_M^\lambda}{\clL(L^2)}
\le\lambda\widehat\alpha_MC_P^2\quad\mbox{and}\quad \norm{\clK_M^\lambda (w-y) }{L^2(\bbR^+,L^2)}
 \le \lambda\widehat\alpha_M\mu^{-1} C_P^2\norm{w_\circ- y_\circ }{L^2},
 \end{align}
\end{theorem}

\begin{proof}
Let us fix~$\lambda>0$ and~$M\in\bbN$ so that Theorem~\ref{T:stab-appN} holds true.
Note that~$\lambda>0$ and~$M\in\bbN$
are independent of~$k$.

Let~$y_k$ and~$w_k$ be the solutions of the Moreau--Yosida approximations~\eqref{sys-yk} and~~\eqref{sys-wk},
respectively.

\begin{subequations}\label{Dwy}
For the difference between the solution~$w$ of~\eqref{sys-tildey-K-intro} and the
solution~$y$ of~\eqref{sys-haty} we find
\begin{align}
 \norm{w(t)-y(t)}{L^2}&\le \norm{w(t)-w_k(t)}{L^2}+\norm{w_k(t)-y_k(t)}{L^2}+\norm{y_k(t)-y(t)}{L^2}
  \end{align}

Let us now be given arbitrary~$\epsilon>0$, ~$\varrho>1$, $T>0$, and~$t\in[0,T]$.

Now for the pair~$(y_k,y)$ we apply Theorem~\ref{T:MY-approx}
with $(f_k,Q)=(f,a\Id+b\cdot\nabla)$, and for the pair~$(w_k,w)$
we apply Theorem~\ref{T:MY-approx}
with~$(f_k,Q)=(f+\clK_M^\lambda y_k,a\Id+b\cdot\nabla+\clK_M^\lambda)$.
In this way we obtain that, for large enough~$k=k(\epsilon,T)$, we have
\begin{align}
&\norm{r(y_{k}-y)}{\clC([0,T], L^2)}\le \epsilon\quad\mbox{and}\quad\norm{r(w_{k}-w)}{\clC([0,T], L^2)}\le \epsilon,
\quad\mbox{with}\quad r(t)=\min\{t,1\}.
 \end{align}
and, since $z_k\coloneqq w_{k} - y_{k}$ satisfies~\eqref{sys-zk}, that is~\eqref{sys-y-P-k}, by using
Theorem~\ref{T:stab-appN}, we obtain
\begin{align}
 \norm{w_k(t)-y_k(t)}{L^2}&\le \rme^{-\mu t}\norm{w_\circ-y_\circ}{L^2},\quad\mbox{for every~$k\in\bbN$}.
\end{align}
\end{subequations}
 Hence, by selecting~$k$ large enough, from~\eqref{Dwy} we obtain that, at time~$t=T>0$, \begin{align}\notag
 \norm{w(T)-y(T)}{L^2}&\le 2\max\{\tfrac1T,1\}\epsilon +\rme^{-\mu T}\norm{w_\circ-y_\circ}{L^2}.
 \end{align}
Choosing now~$\epsilon\coloneqq\frac12\min\{T,1\}(\varrho-1)\ex^{-\mu T}\norm{w_\circ-y_\circ}{L^2}$,
we arrive at
\begin{align}\notag
 \norm{w(T)-y(T)}{L^2}&\le (\varrho-1)\rme^{-\mu T}\norm{w_\circ-y_\circ}{L^2}
 +\rme^{-\mu T}\norm{w_\circ-y_\circ}{L^2} = \varrho\rme^{-\mu T}\norm{w_\circ-y_\circ}{L^2}.\black
\end{align}
Furthermore, since~$T>0$ and~$\varrho>1$ are arbitrary we arrive at
 \begin{align}\notag
 \norm{w(t)-y(t)}{L^2}&\le\ex^{-\mu t}\norm{w_\circ-y_\circ}{L^2}, \quad t\ge0.
 \end{align}

Finally proceeding as in~\eqref{bound.controlM}, we find
 \[\norm{\clK_M^\lambda(w-p) }{L^2(\bbR^+,L^2)}
 \le \lambda\widehat\alpha_M \norm{P_{\clU_{M}}^{\clE_{M}^\perp}}{\clL(L^2)}^2\norm{w-p}{L^2(\bbR^+,L^2)}
 \le \lambda\widehat\alpha_M\mu^{-1} C_P^2\norm{w_\circ-y_\circ}{L^2},
\]
with~$\widehat\alpha_M$ and~$C_P$ as in~\eqref{bound.Proj.ratLam}, which finishes the proof.
\end{proof}

\section{Numerical simulations}\label{sec:num}
We consider Moreau--Yosida approximations of one-dimensional parabolic variational inequality in the
spatial open interval~$\Omega=(0,1)\subset\bbR$, and impose homogeneous Neumann boundary conditions, for simplicity. 
\begin{subequations}\label{sys-yk-nu}
 \begin{align}
&\tfrac{\p}{\p t}{ y_{k}} +(-\nu\Delta+\Id) y_{k}+a y_{k}+b\cdot\nabla y_{k} -f +k(y_{k}-\psi)^+= 0,
\quad t>0,\\
&\tfrac{\p}{\p\bfn} y_{k}\rest{\Gamma}=0,\qquad y_{k}(\Bigcdot,0)= y_\circ.
\end{align}
\end{subequations}

\begin{subequations}\label{simul.setting}
For the parameters, we  have chosen
\begin{align}
\nu&=0.1,\qquad& f(x,t)&=-\sin(t)x,\\
a(x,t)&=-6+x +2\norm{\sin(t+x)}{\bbR},\qquad& b(x,t)&=\cos(t)x^2
\end{align}
and 
\begin{equation}\label{Opsi-smooth}
\psi(x,t)=2+\cos(t)+\cos\left(10\pi x(x-1)\bigl(x-\tfrac{1}{4}\cos(5 t)\bigr)\right).
\end{equation}
\end{subequations}
Recall that by Theorem~\ref{T:MY-approx}, we have that~$y_{k}$ gives us an
approximation of the solution~$y$ of the variational inequality with the same data parameters.
See also Remark~\ref{R:nu}.

The targeted trajectory~$y$ is the one issued, at initial time~$t=0$, from the state
\begin{align}
y(x,0)&=y_\circ(x)=3\cos(\pi x),
\intertext{and we want to target such trajectory starting, again at time~$t=0$, from the state}
w(x,0)&=w_\circ(x)=-1.
\end{align}
Again by Theorem~\ref{T:MY-approx}, we have that~$w_{k}$ solving
\begin{subequations}\label{sys-wk-nu}
 \begin{align}
&\tfrac{\p}{\p t}{ w_{k}} +(-\nu\Delta+\Id) w_{k}+a w_{k}+b\cdot\nabla w_{k}
-f-\clK_M^\lambda(w_{k}-y_{k})+k(w_{k}-\psi)^+= 0,
\quad t>0,\\
&\tfrac{\p}{\p\bfn} w_{k}\rest{\Gamma}=0,\qquad w_{k}(\Bigcdot,0)= w_\circ,
\end{align}
\end{subequations}
gives us an
approximation of the solution~$w$ of the controlled variational inequality with the same data parameters.

Initial states are plotted in Figure~\ref{Fig:InitStates}.
\begin{figure}[hb]
\centering
\subfigure
{\includegraphics[width=0.45\textwidth,height=0.31\textwidth]{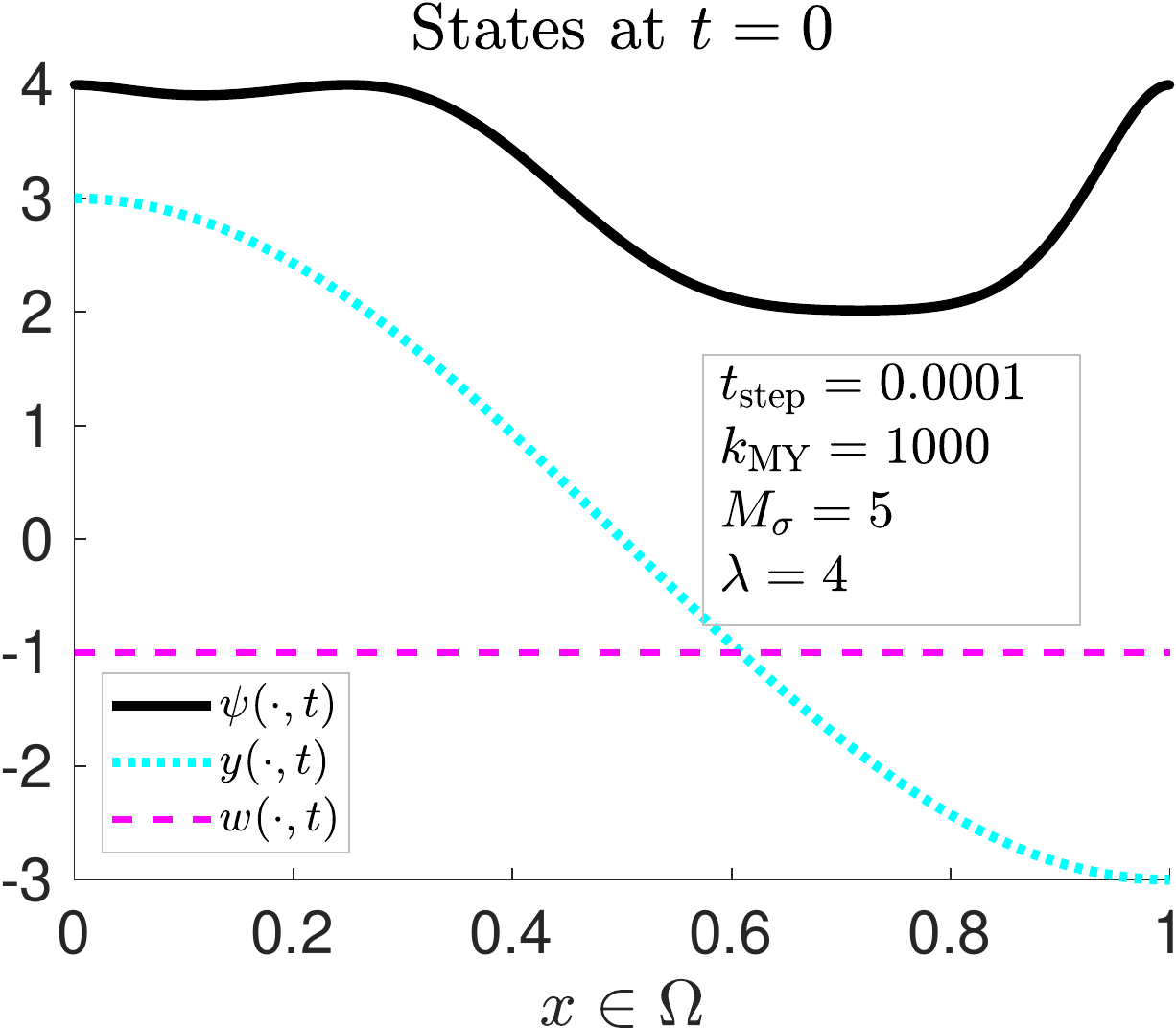}}
\caption{Initial states.}
\label{Fig:InitStates}
\end{figure}


For a fixed~$M\in\bbN$ we take~$M_\sigma=M$ actuators as in~\cite{KunRod19-cocv} which are indicator functions~$1_{\omega_j^M}$ of the subdomains
\[
\omega_j^M=(\tfrac{2j-1}{2M}-\tfrac1{20M},\tfrac{2j-1}{2M}+\tfrac1{20M}),\qquad  j\in\{1,2,\dots,M\}.
\]
In particular, note that the total volume covered by the actuators is independent of~$M$. It is given by~$\frac1{10}$, which is $10\%$ of the total volume of the spatial domain.

As auxiliary space of eigenfunctions we take the first eigenfunctions of the Laplace operator, under the  imposed Neumann boundary conditions, namely
\[
e_j^M=\cos((j-1)\pi x),\qquad  j\in\{1,2,\dots,M\}.
\]

The obstacle~$\psi(\cdot,t)$ satisfies $\tfrac{\p}{\p \bfn}\psi=0$ at every~$t\ge0$.
Recall that our Assumption~\ref{A:obst-lb}
  requires that~$\tfrac{\p}{\p \bfn}\psi\ge -\eta$ for a suitable positive function~$-\eta\in W_{\rm loc}^{1,2}(\bbR_+)\ge0$
hence it is satisfied.

 Furthermore, we can see that Assumptions~\ref{A:char-dom}--\ref{A:obst-lb} are satisfied.
 Therefore all the hypothesis of Theorems~\ref{T:stab-appN}
are satisfied. Hereafter we present the results of simulations illustrating the stability
result stated in the thesis of Theorem~\ref{T:stab-appN}. 

As we have mentioned above, by solving systems~\eqref{sys-yk-nu} and~\eqref{sys-wk-nu}, by Theorem~\ref{T:main}, 
with  a relatively large Moreau--Yosida parameter~$k=k_{MY}$ we expect to obtain a relatively good
approximation of the behavior of the limit solutions for the corresponding variational inequalities.
Depending on the simulation example, we have taken~$k_{MY}$ in the interval~$[500,20000]$.

For the discretization, we considered a finite element spatial approximation
based on the classical piecewise linear hat functions, where the closure~$[0,1]$ of
the spatial interval has been discretized with a regular mesh with 2001 equidistant points.
Subsequently the  closure~$[0,+\infty)$ of the temporal interval has been discretized with
a uniform time-step~$t_{\rm step}>0$ and a  Crank--Nicolson/Adams--Bashforth scheme was used. Depending on the simulation we
have taken~$t_{\rm step}\in\{10^{-4},10^{-5}\}$.

 In the figures below we denote~$H\coloneqq L^2(\Omega)$.

\subsection{Stabilizing performance of the feedback control}
In Figure~\ref{Fig:DF_lam4M5Feed04T4}
\begin{figure}[ht]
\centering
\subfigure
{\includegraphics[width=0.45\textwidth,height=0.31\textwidth]{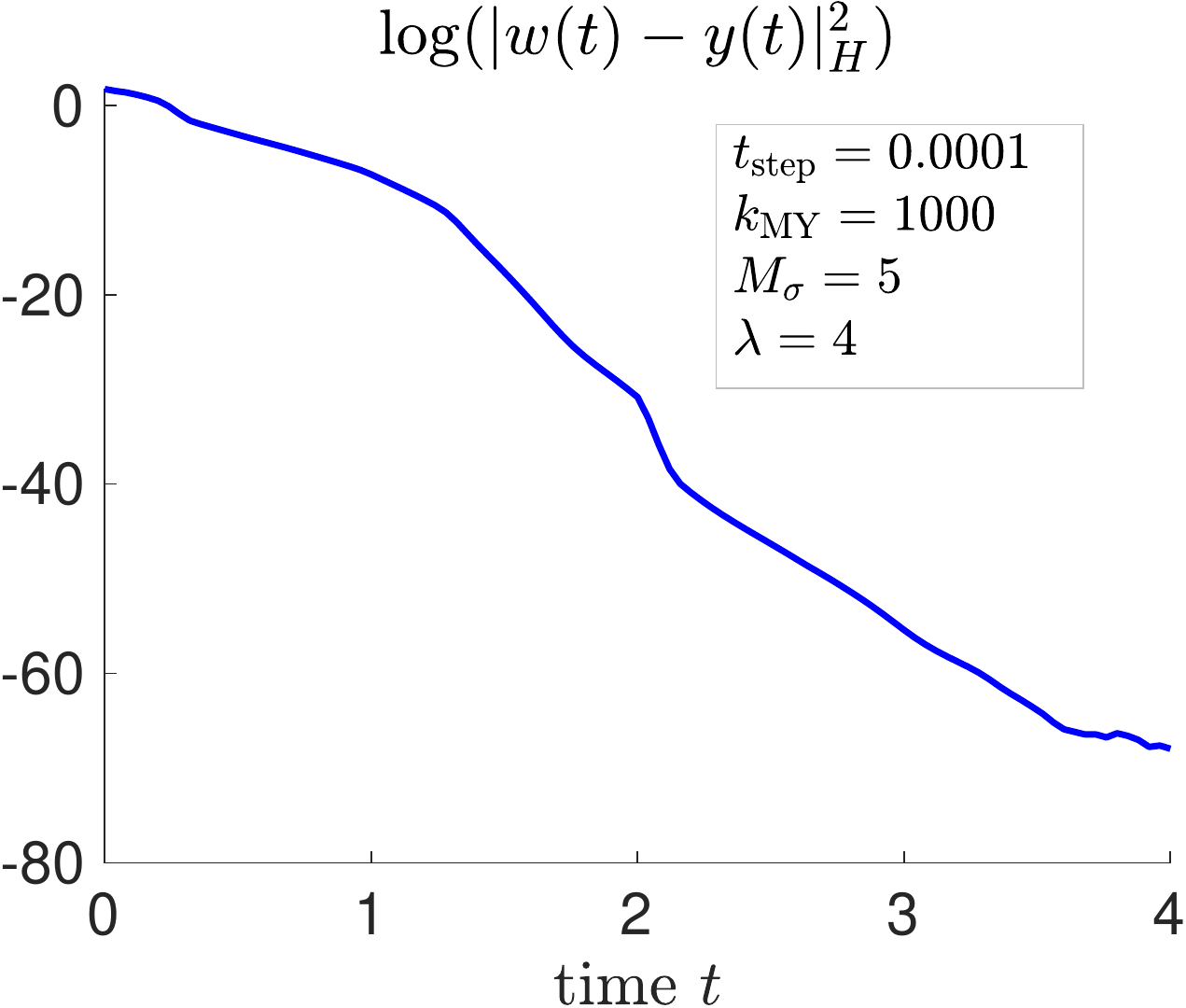}}
\qquad
\subfigure
{\includegraphics[width=0.45\textwidth,height=0.31\textwidth]{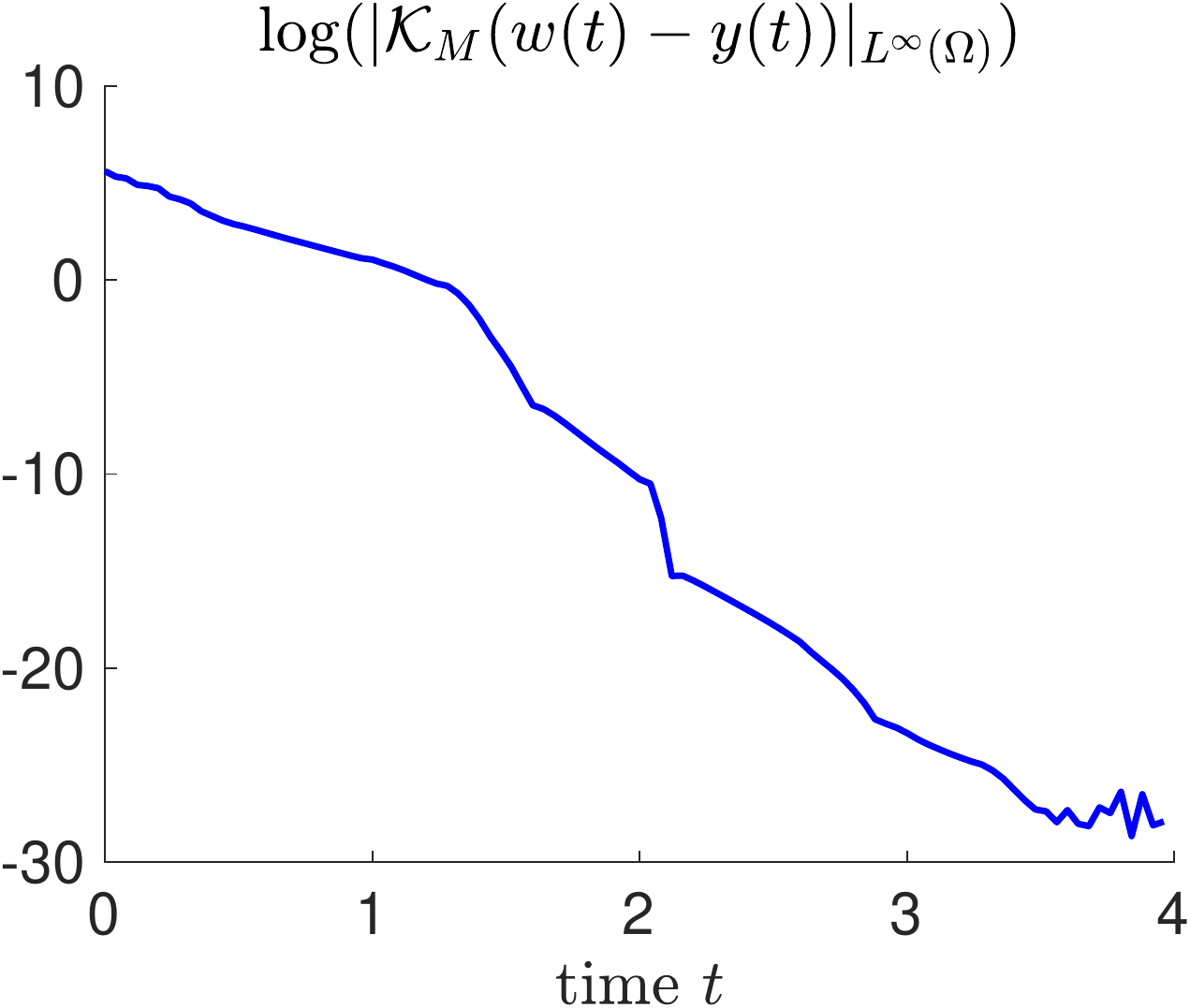}}
\caption{Norms of difference to target and control.}
\label{Fig:DF_lam4M5Feed04T4}
\end{figure}
we can see that with~$5$ actuators and~$\lambda=4$ the explicit oblique
projection feedback control we propose in this manuscript is able to stabilize
the solution~$w=w_k$ of the Moreau--Yosida approximation, with~$k=k_{\rm MY}=1000$,
to the corresponding targeted uncontrolled solution approximation~$y=y_k$.

Time snapshots of the corresponding trajectories and control
are shown in Figures~\ref{Fig:TDF_lam4M5Feed04T4-tl}.
\begin{figure}[ht]
\centering
\subfigure
{\includegraphics[width=0.45\textwidth,height=0.31\textwidth]{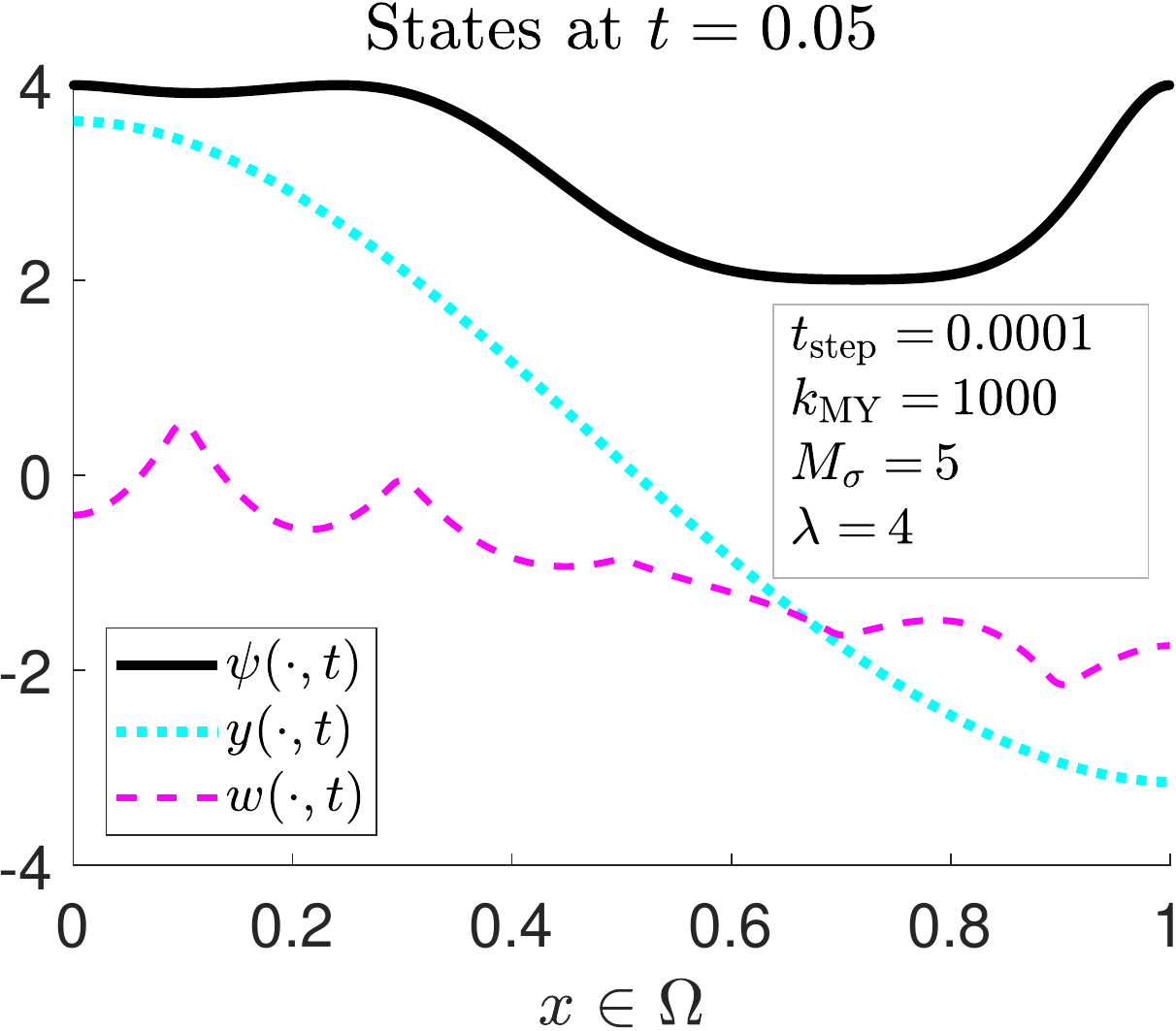}}
\qquad
\subfigure
{\includegraphics[width=0.45\textwidth,height=0.31\textwidth]{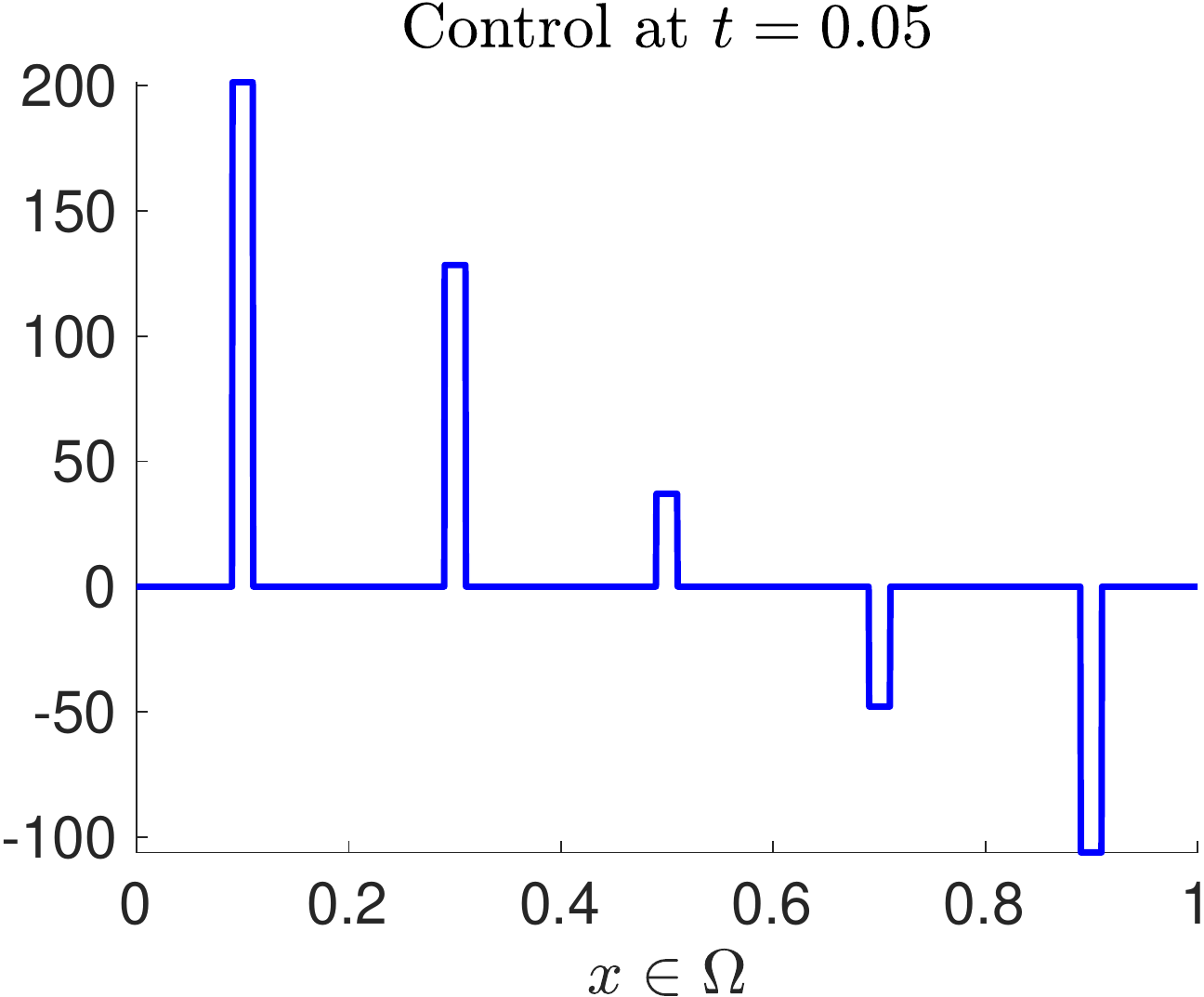}}
\subfigure
{\includegraphics[width=0.45\textwidth,height=0.31\textwidth]{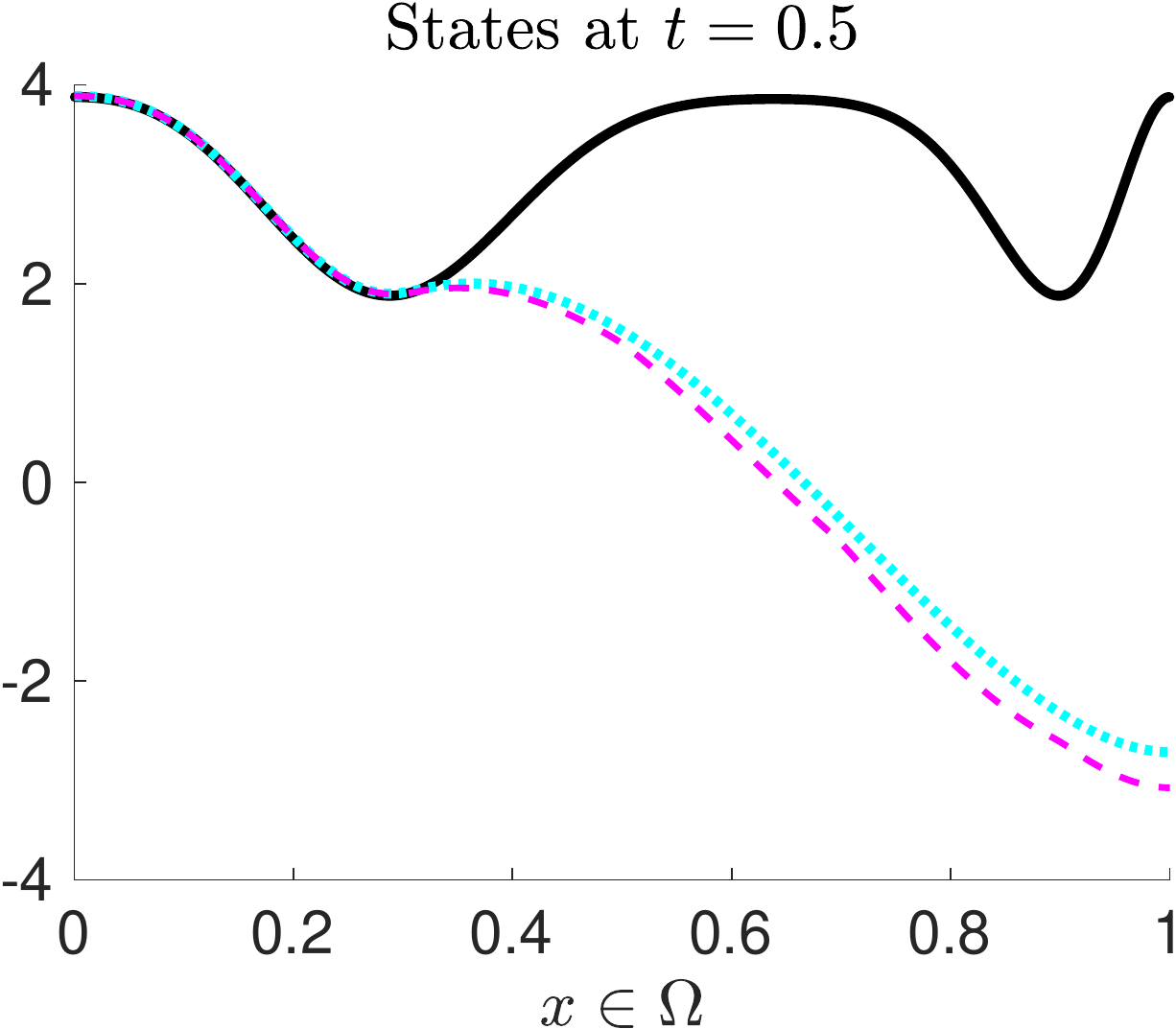}}
\qquad
\subfigure
{\includegraphics[width=0.45\textwidth,height=0.31\textwidth]{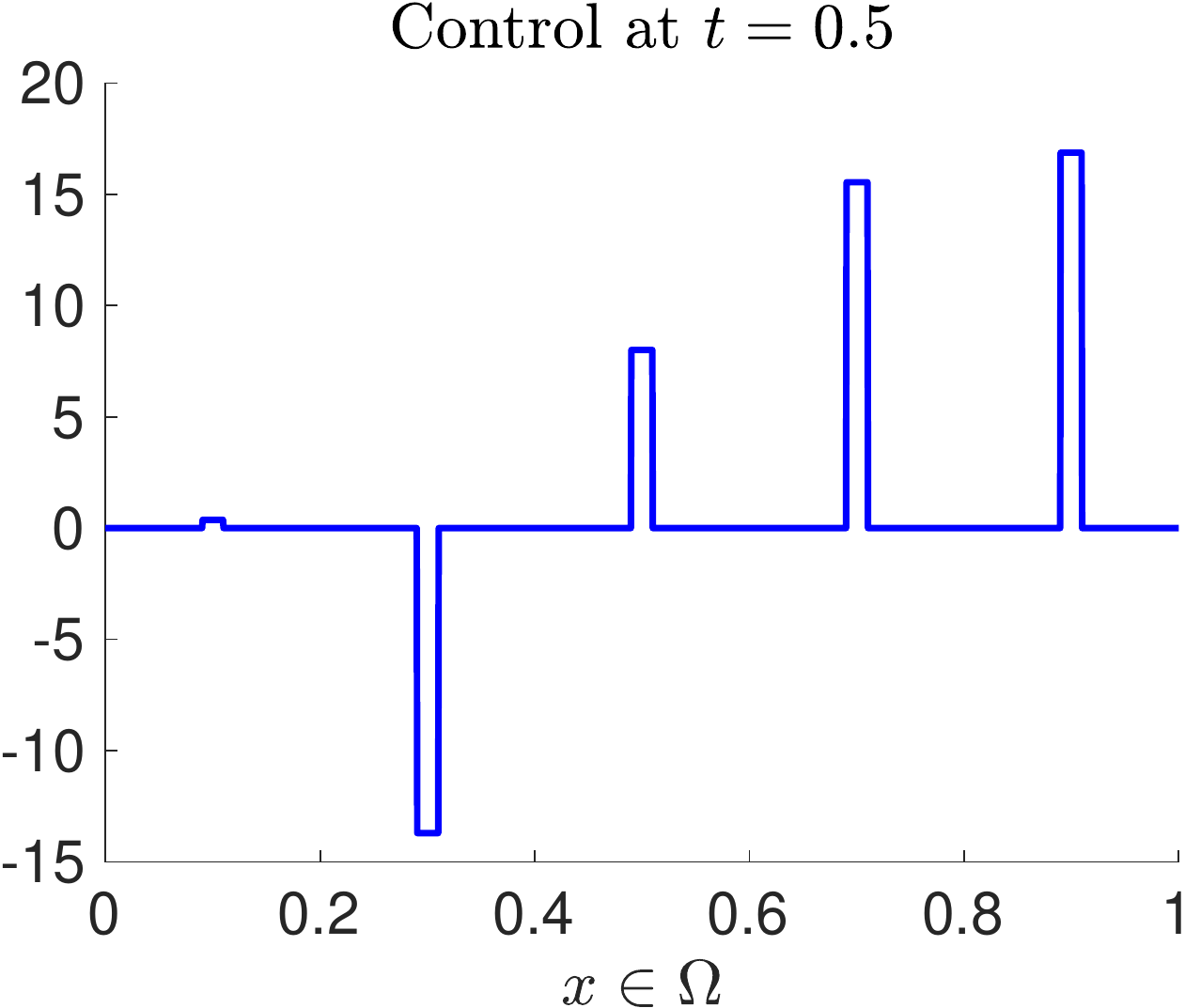}}\\
\subfigure
{\includegraphics[width=0.45\textwidth,height=0.31\textwidth]{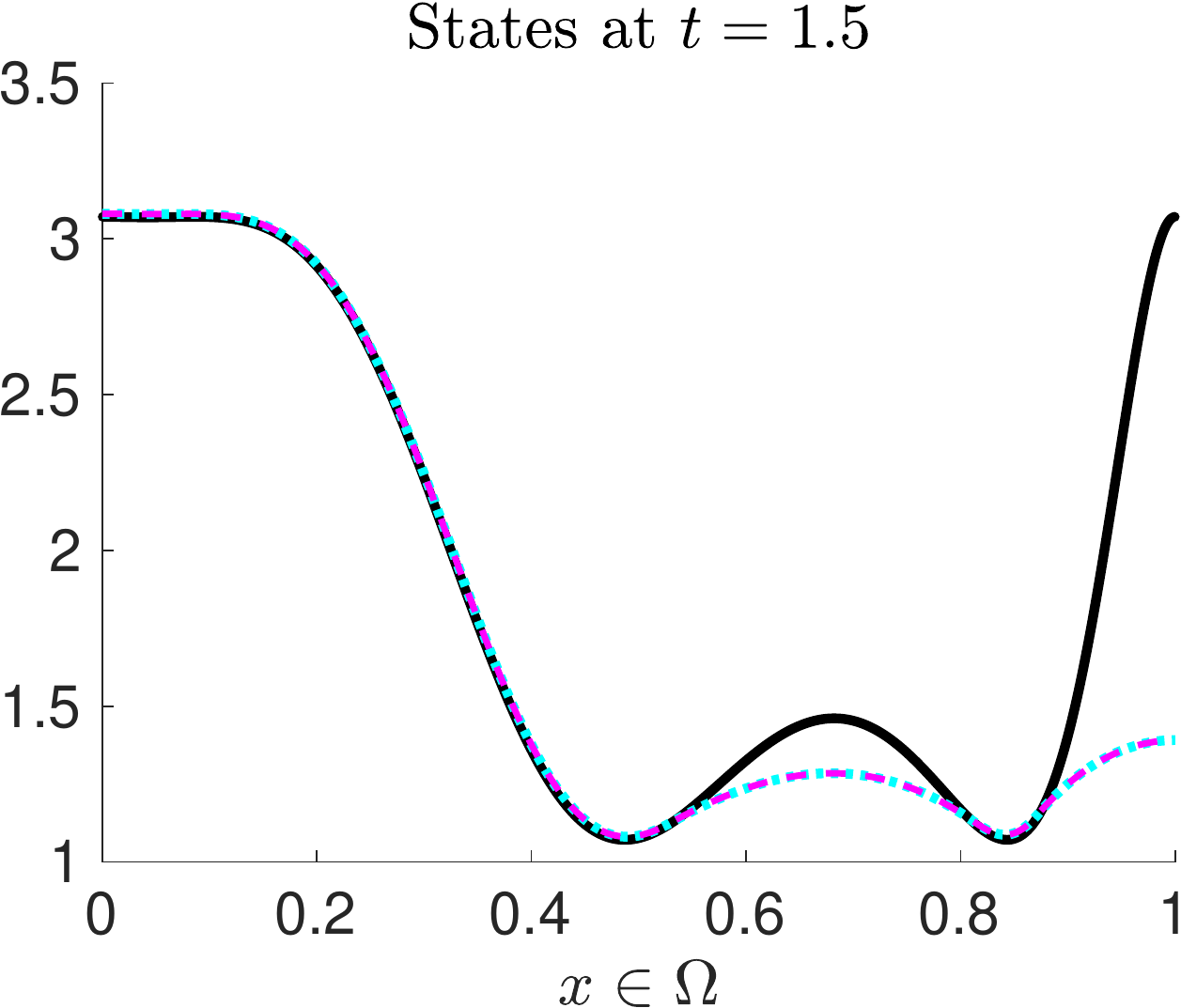}}
\qquad
\subfigure
{\includegraphics[width=0.45\textwidth,height=0.31\textwidth]{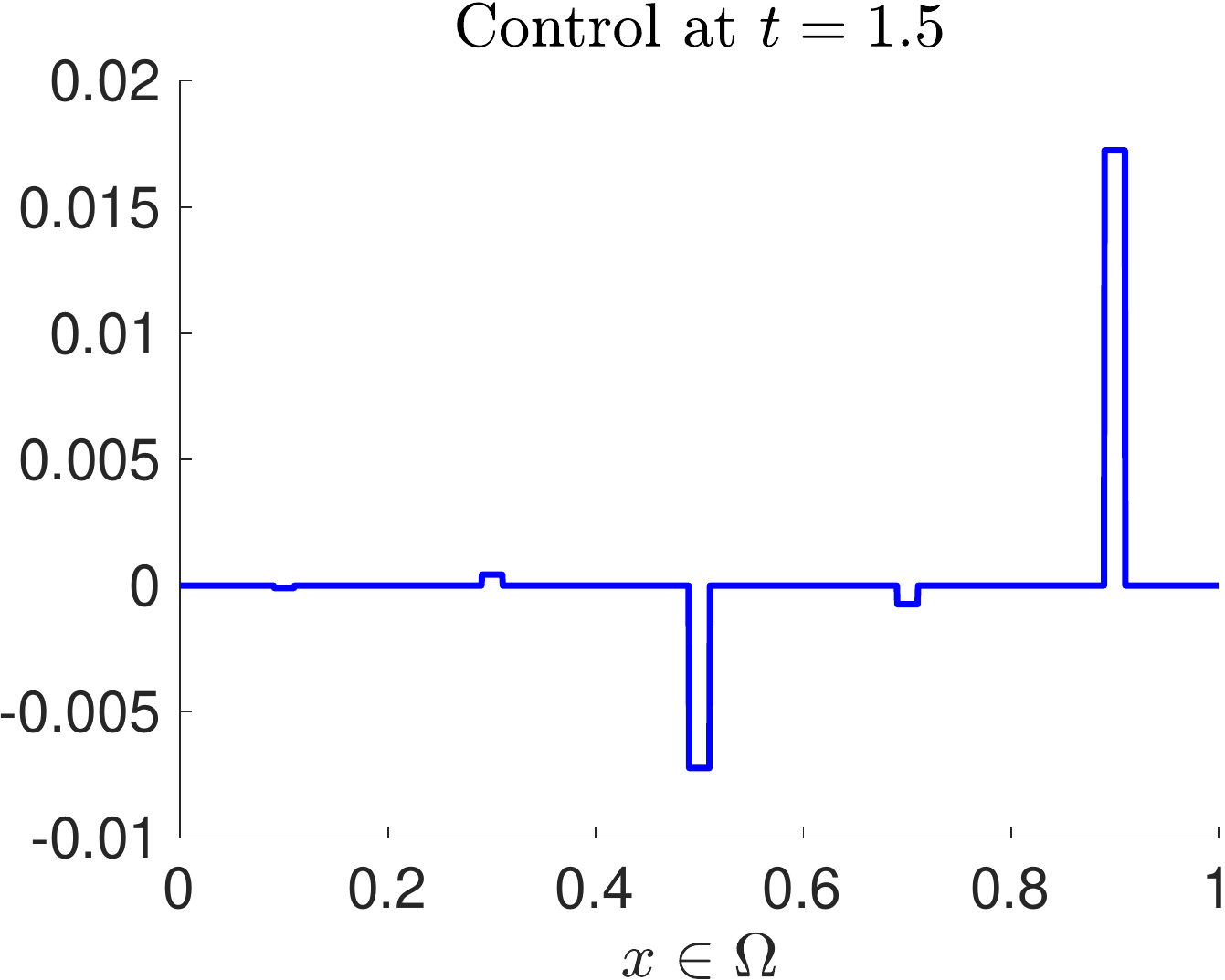}}
\caption{Time snapshots of trajectories and control. Larger time}
\label{Fig:TDF_lam4M5Feed04T4-tl}
\end{figure}
It is interesting to observe, at time~$t=0.05$, the $5$ bumps on the shape of the controlled solution, which are pointing towards the targeted one. The spatial location of these bumps coincide with spatial location of the actuators, and they show the action of the feedback control pushing the controlled solution towards the targeted one.

\subsection{On the Moreau--Yosida parameter~$k_{\rm MY}$}\label{sS:MYk}
The goal of this section is to show that it is very likely that the Moreau--Yosida approximation  
with parameter $k_{\rm MY}=500$ in the above simulation give us already a good approximation of the  
behavior of the limit solution of the variational inequality.
Indeed, in Figure~\ref{Fig:kMYDF_lam4M5Feed04T4}, we can  see that the norm 
of the difference to the target presents an analogous evolution for the considered parameters~$k_{\rm MY}s$.
\begin{figure}[ht]
\centering
\subfigure
{\includegraphics[width=0.45\textwidth,height=0.31\textwidth]{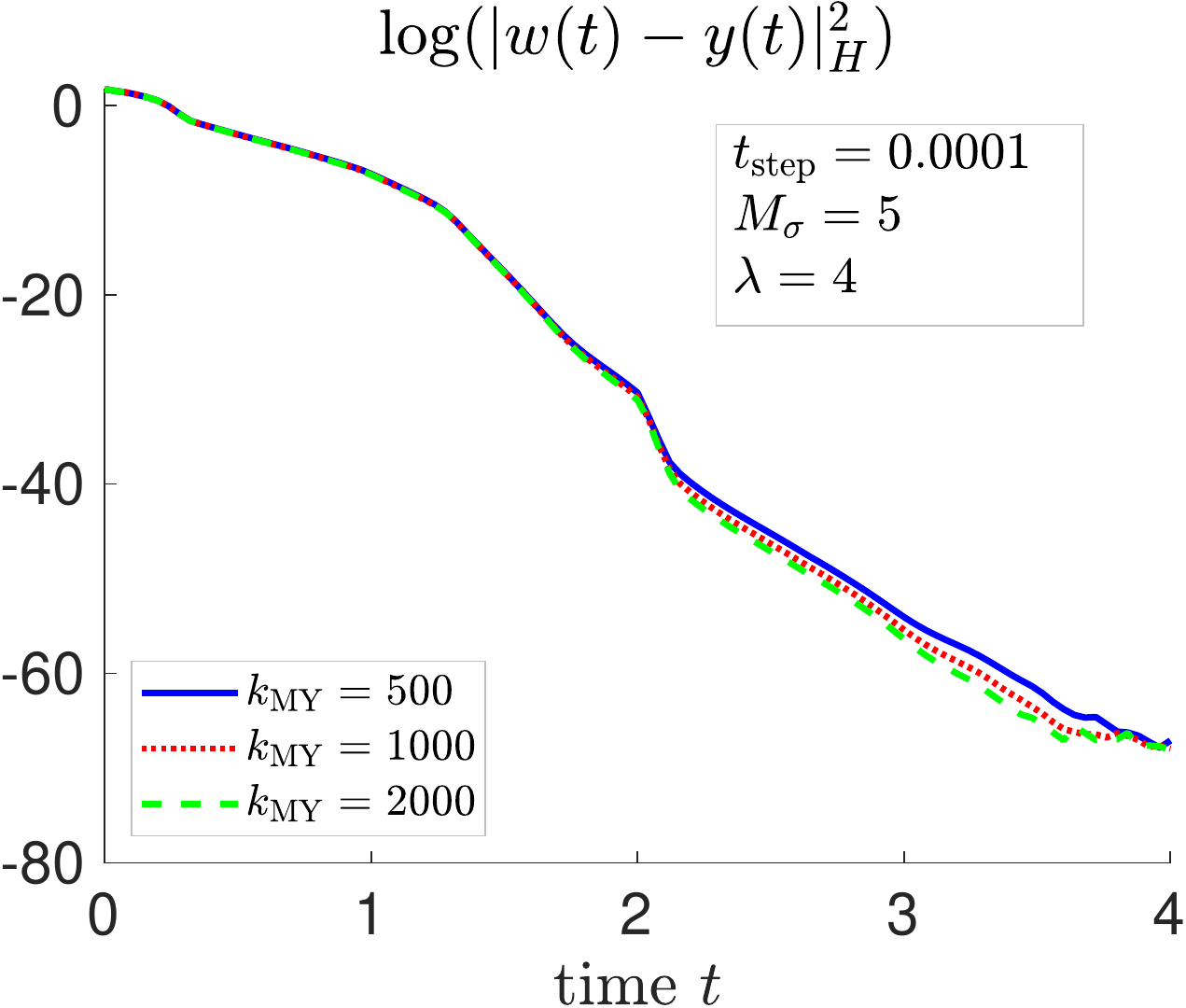}}
\qquad
\subfigure
{\includegraphics[width=0.45\textwidth,height=0.31\textwidth]{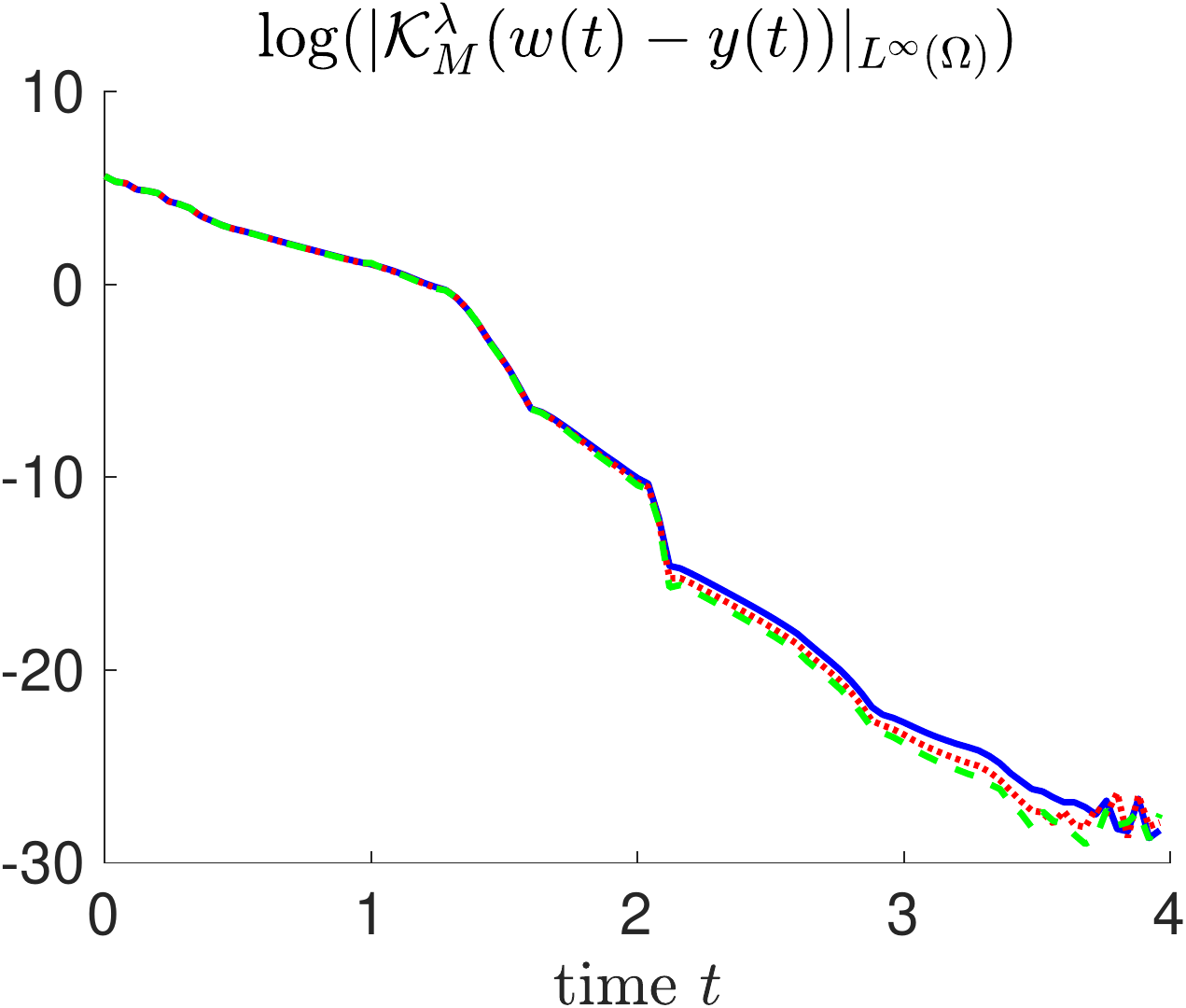}}
\caption{Norms of difference to target and control}
\label{Fig:kMYDF_lam4M5Feed04T4}
\end{figure}

In Figure~\ref{Fig:kMYOyw_lam4M5Feed04T4}
\begin{figure}[ht]
\centering
\subfigure
{\includegraphics[width=0.45\textwidth,height=0.31\textwidth]{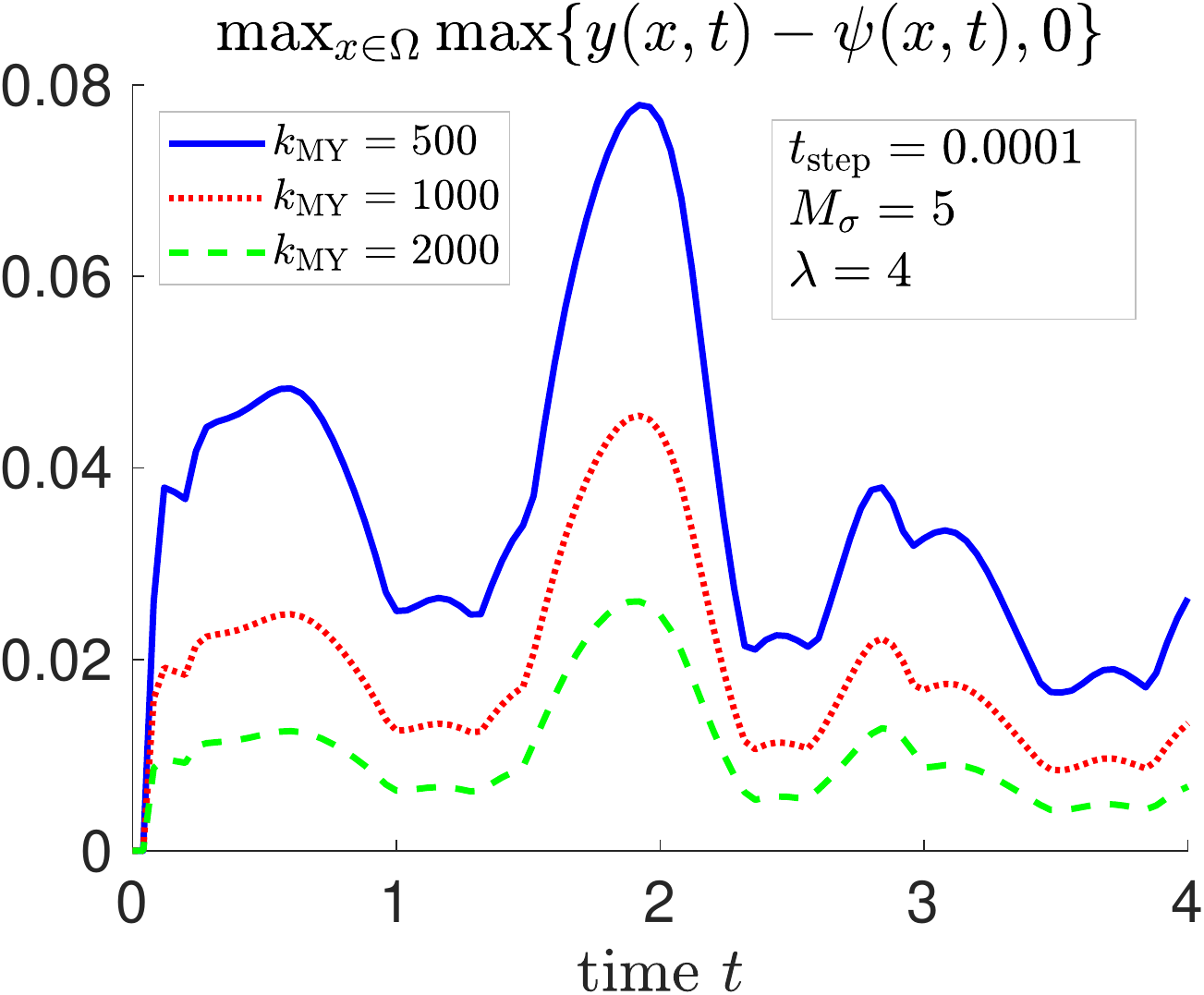}}
\qquad
\subfigure
{\includegraphics[width=0.45\textwidth,height=0.31\textwidth]{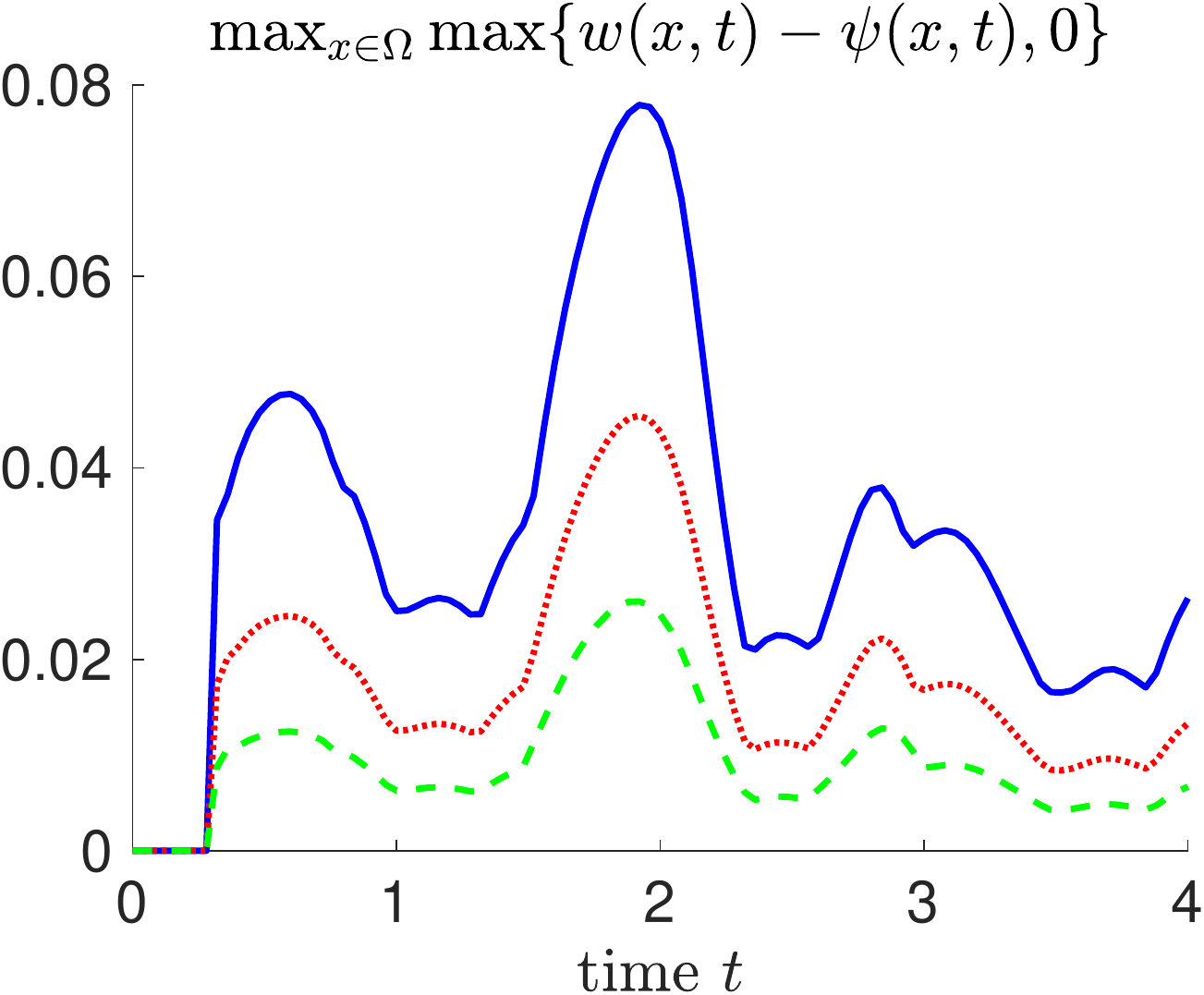}}
\caption{Largest magnitude of obstacle constraint violation}
\label{Fig:kMYOyw_lam4M5Feed04T4}
\end{figure}
 we see that the obstacle constraint violation decreases as~$k_{\rm MY}$ increases,
 as we expect, since at the limit we must have a vanishing constraint violation.
 Furthermore, from Lemma~\ref{L:weak-viol} we have that~$k\norm{(y_{k}-\psi)^+}{L^2(\Omega\times(0,T))}\le C$
 for a suitable constant~$C$ independent of~$k$.
 Figure~\ref{Fig:kMYOyw_lam4M5Feed04T4}  shows  that the violation decreases  (at each instant of time) 
 as~$k$ increases.

In Figure~\ref{Fig:kMYTwF_lam4M5Feed04T4}
\begin{figure}[hb]
\centering
\subfigure
{\includegraphics[width=0.45\textwidth,height=0.31\textwidth]{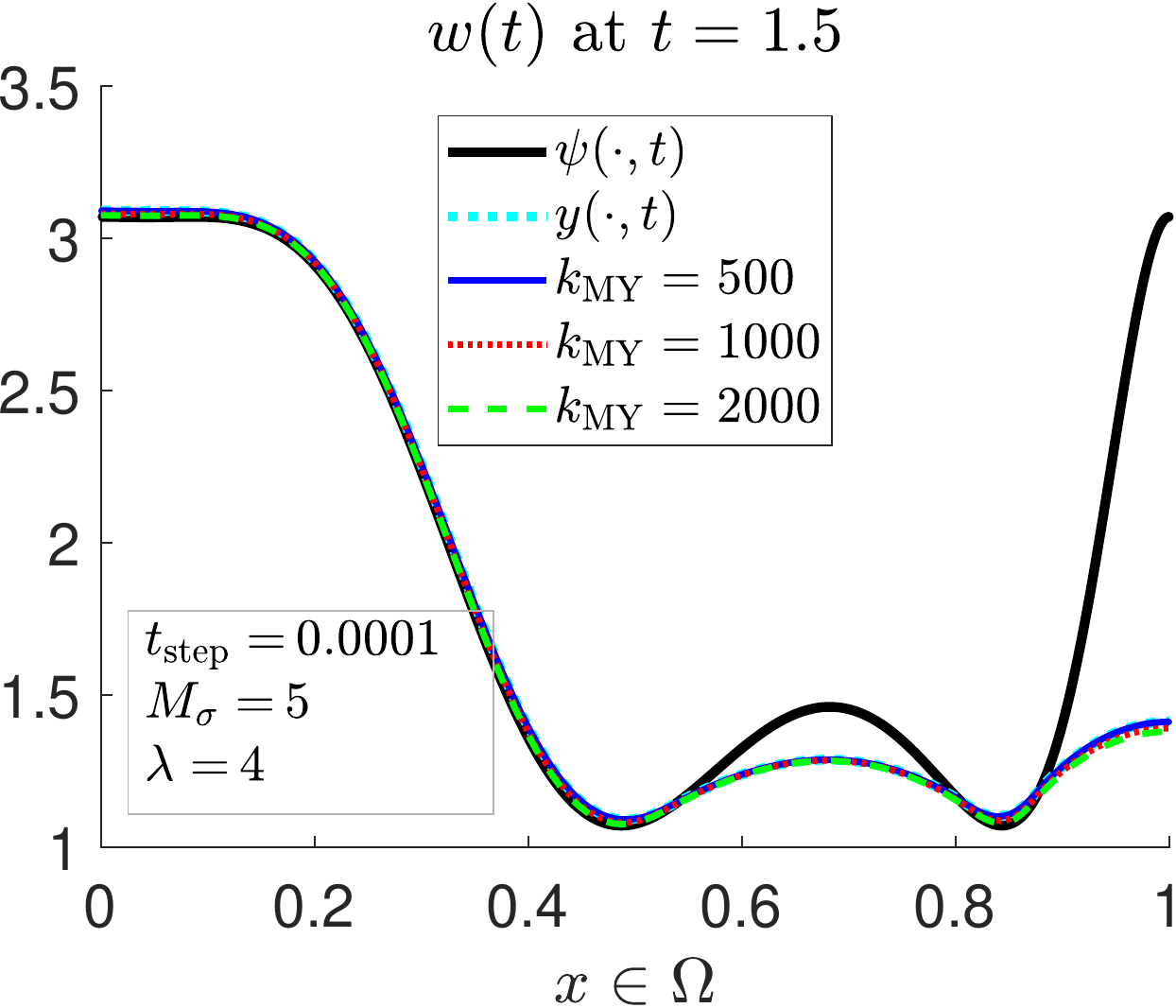}}
\qquad
\subfigure
{\includegraphics[width=0.45\textwidth,height=0.31\textwidth]{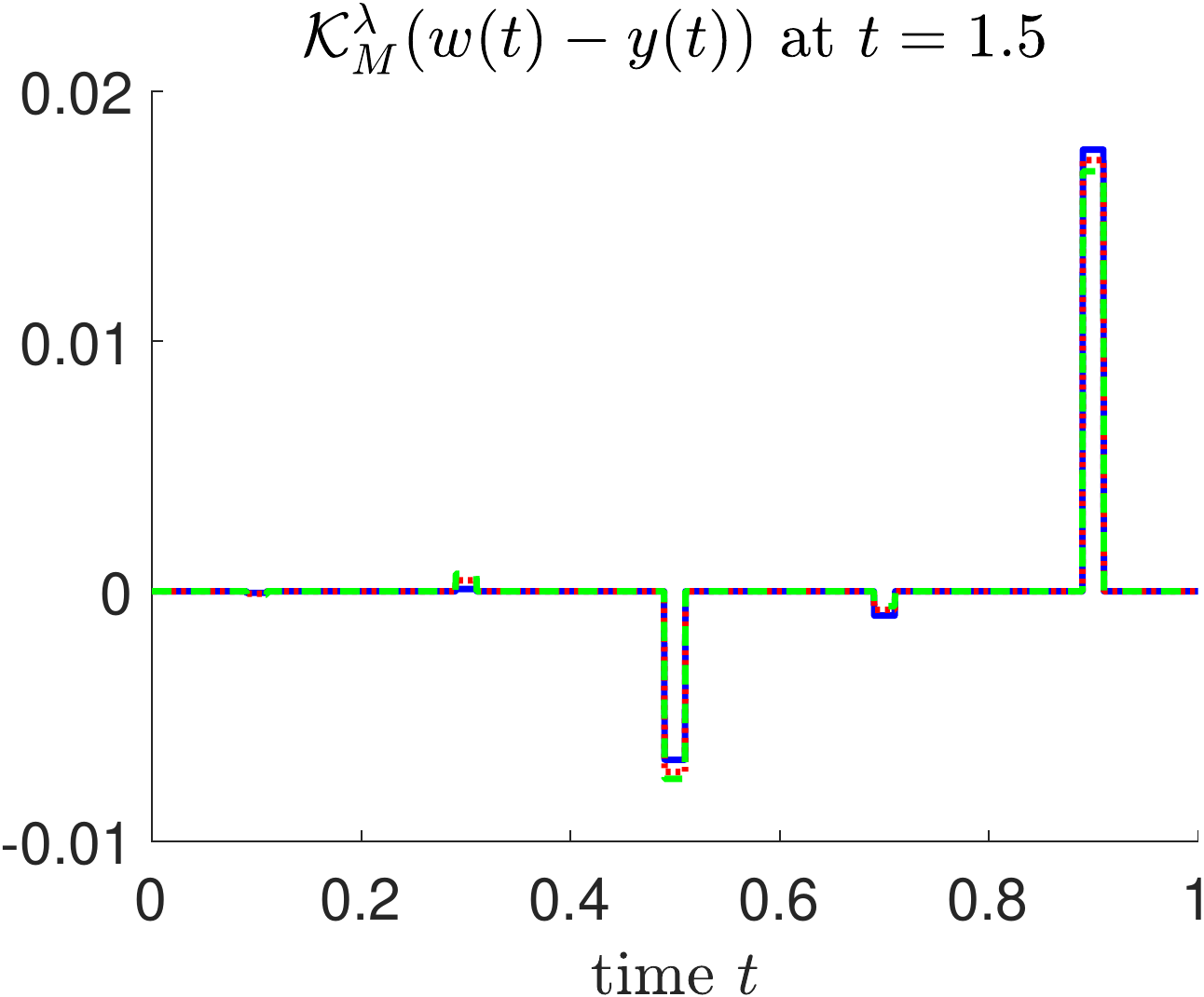}}
\caption{Time snapshots of trajectories and controls}
\label{Fig:kMYTwF_lam4M5Feed04T4}
\end{figure}
 we see a time snapshot of the controlled trajectories and control, where we see a small
 difference between the controlled trajectories for the several~$k_{MY}$s.
 A similar behavior was observed for the corresponding targeted trajectories,
 for simplicity we plotted only the targeted trajectory~$y$  corresponding to~$k_{MY}=500$
 (which, at that instant of time, is already almost indistinguishable form the controlled states with the naked eye).

\subsection{Necessity of both large~$M$ and large~$\lambda$}\label{sS:pairMLambda}
From our result, for stability it is {\em sufficient} to take large~$M$ and large~$\lambda$.
Here, we present simulations showing that such  condition is also {\em necessary}.

\subsubsection{Necessity of large enough~$M$}\label{ssS:necLM}
In Figure~\ref{Fig:LamDF_lam50M1Feed01T1}
\begin{figure}[hb]
\centering
\subfigure
{\includegraphics[width=0.45\textwidth,height=0.31\textwidth]{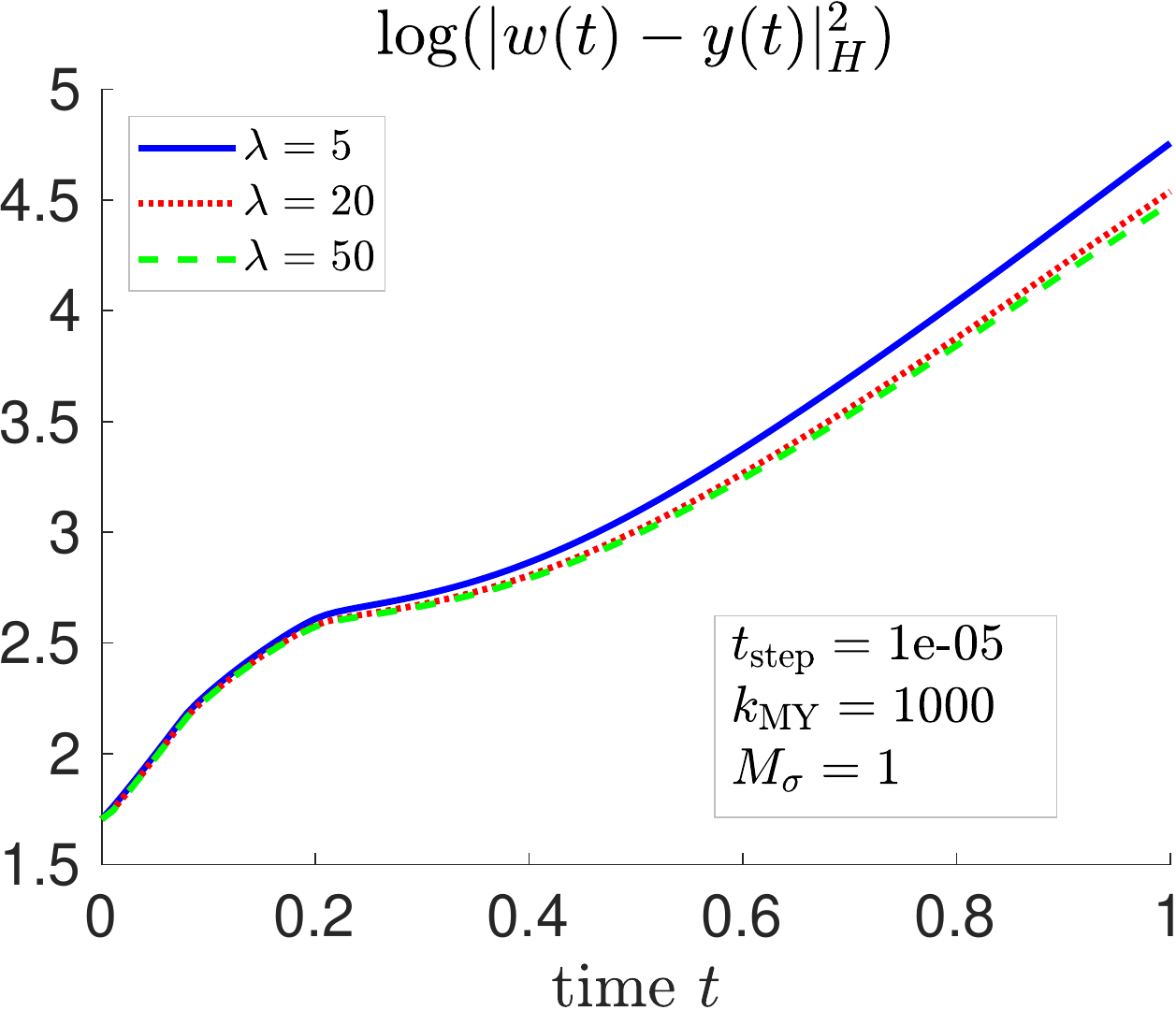}}
\qquad
\subfigure
{\includegraphics[width=0.45\textwidth,height=0.31\textwidth]{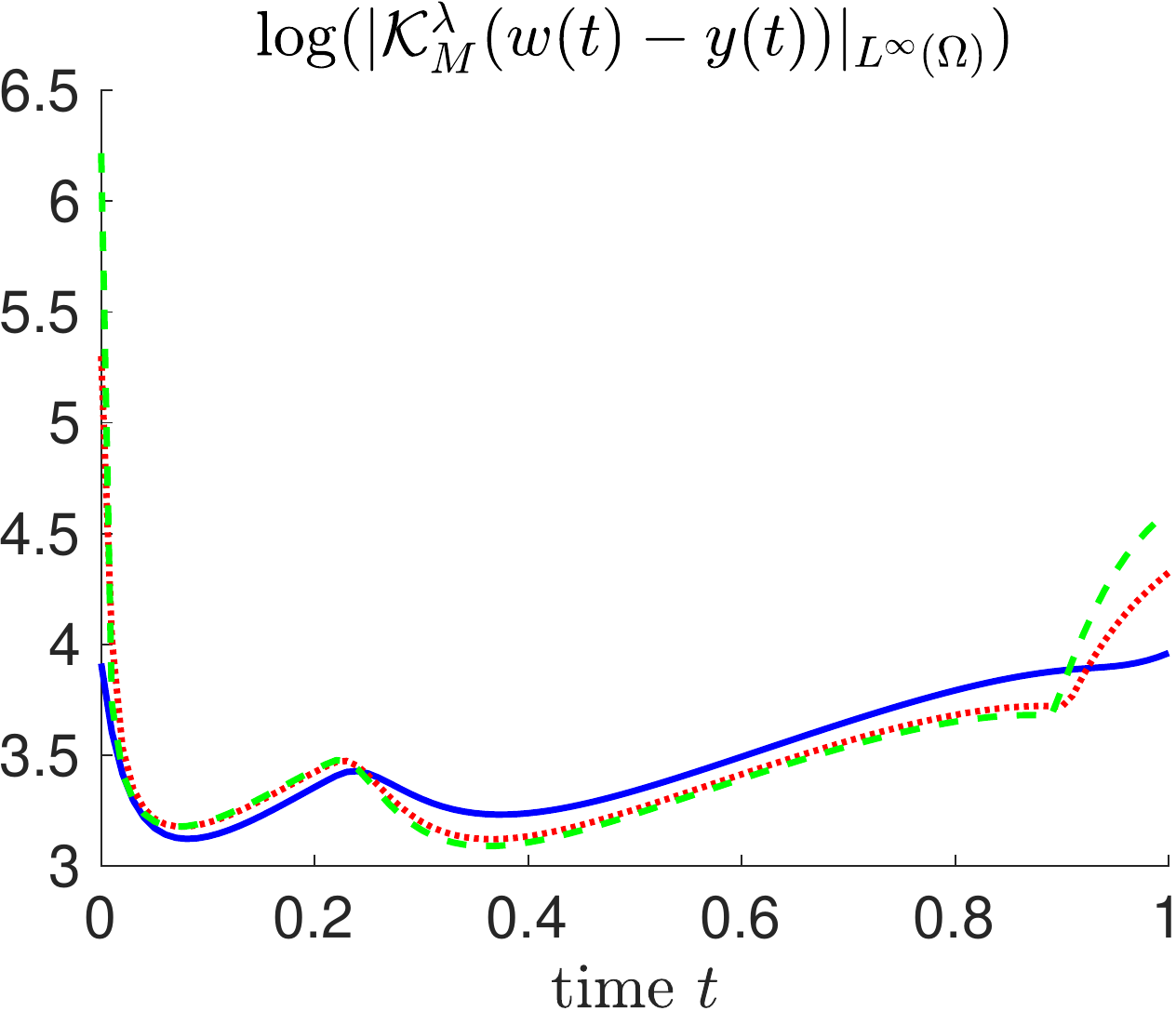}}
\caption{Norms of difference to targeted state and of control}
\label{Fig:LamDF_lam50M1Feed01T1}
\end{figure}
we see that with a single actuator we cannot stabilize the system, even for the relatively large~$\lambda=50$. Furthermore, for small time we cannot see a considerable change in the norm of the difference to the target for the several~$\lambda$s. This allow us to extrapolate that one actuator is not enough to stabilize the system.

In Figure~\ref{Fig:LamTwF_lam50M1Feed01T1}
\begin{figure}[hb]
\centering
\subfigure
{\includegraphics[width=0.45\textwidth,height=0.31\textwidth]{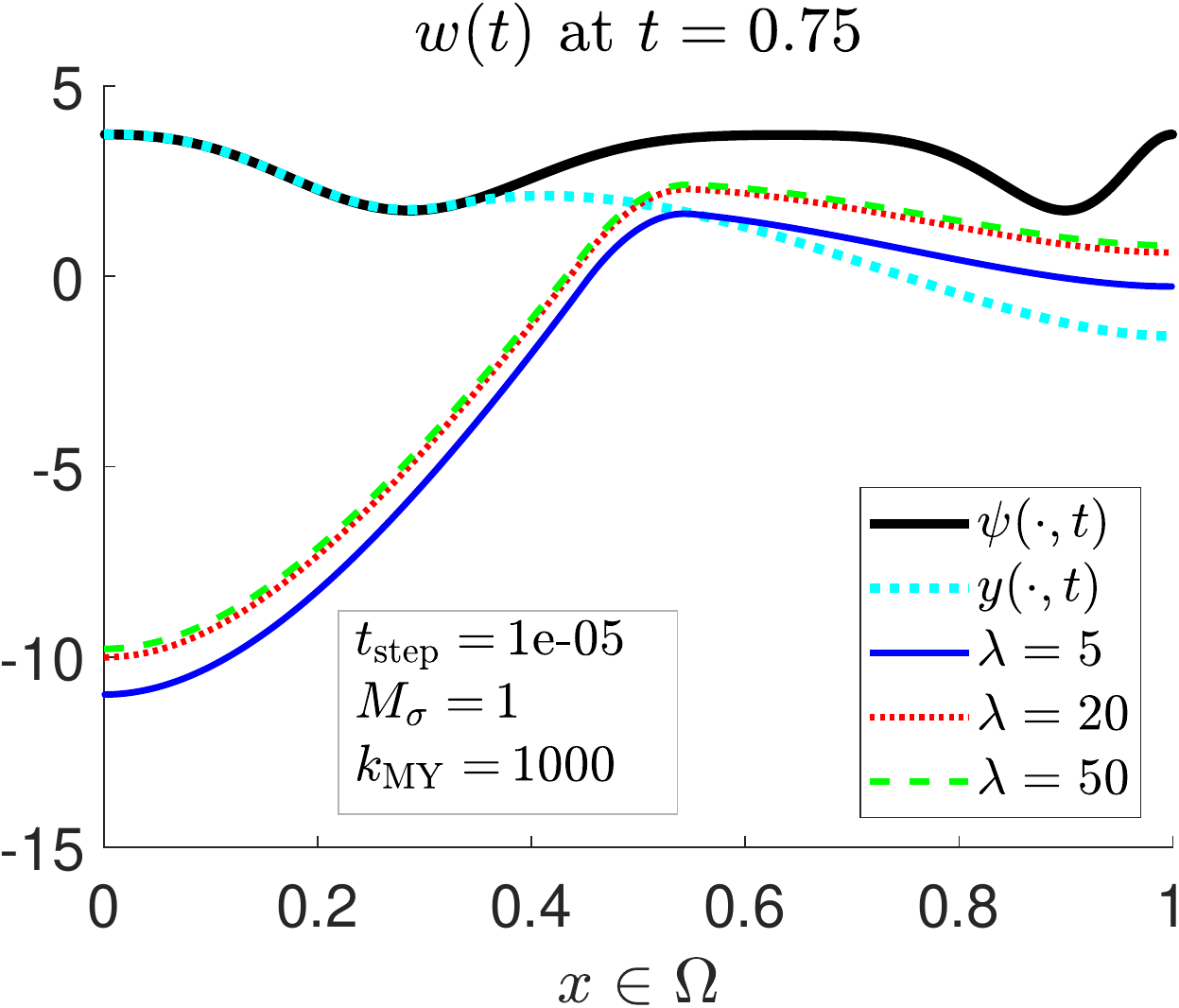}}
\qquad
\subfigure
{\includegraphics[width=0.45\textwidth,height=0.31\textwidth]{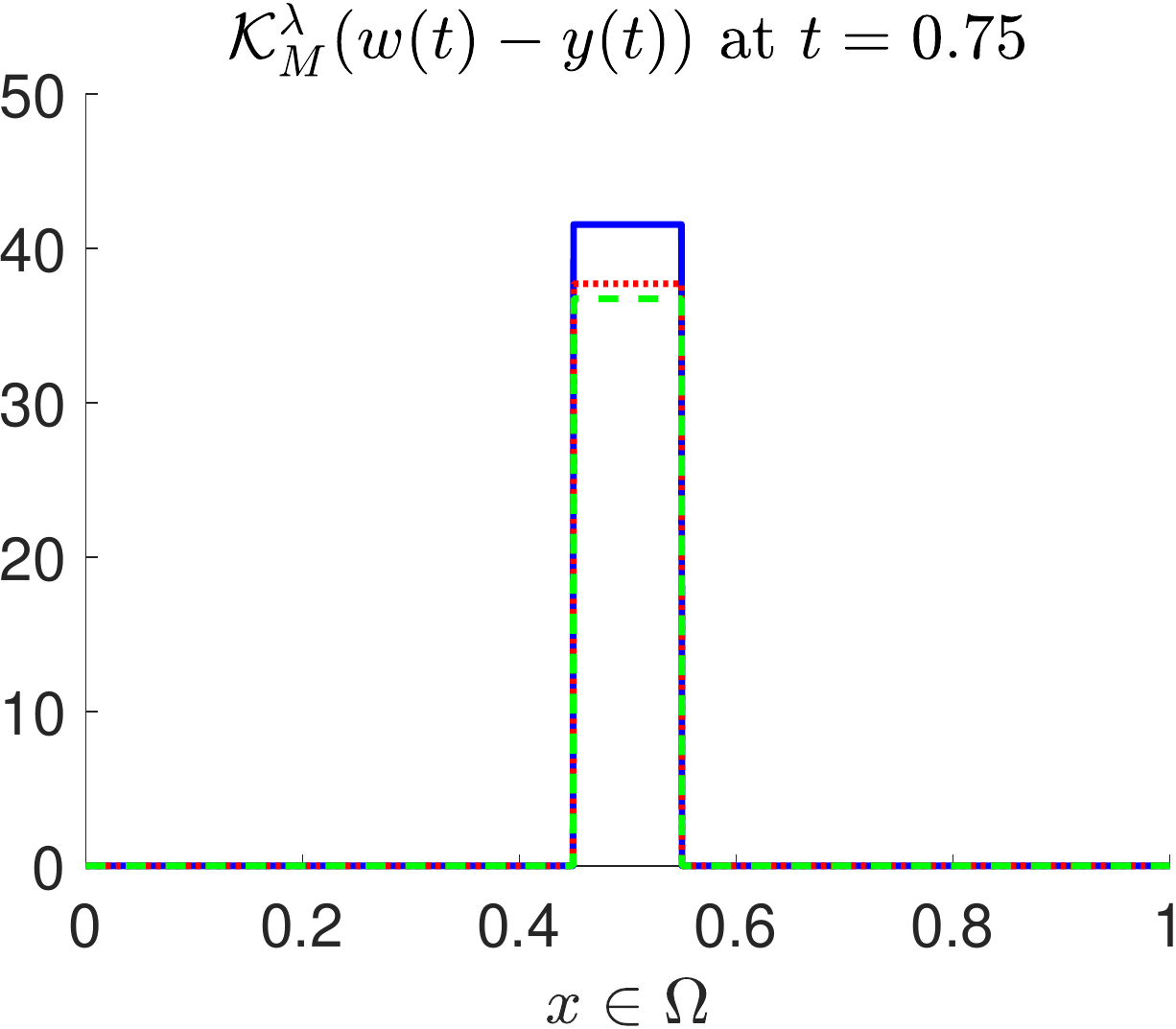}}
\caption{Time snapshots of trajectories and controls}
\label{Fig:LamTwF_lam50M1Feed01T1}
\end{figure}
 we present time snapshots of trajectories and control. We see that by taking a larger~$\lambda$ we cannot see a strong enough influence on the evolution of the trajectory to expect (or, hope for) a stabilization effect for large values of $\lambda$.

\subsubsection{Necessity of large enough~$\lambda$}\label{ssS:necLlambda}
In Figure~\ref{Fig:MDF_lam1M20Feed01T1}
\begin{figure}[ht]
\centering
\subfigure
{\includegraphics[width=0.45\textwidth,height=0.31\textwidth]{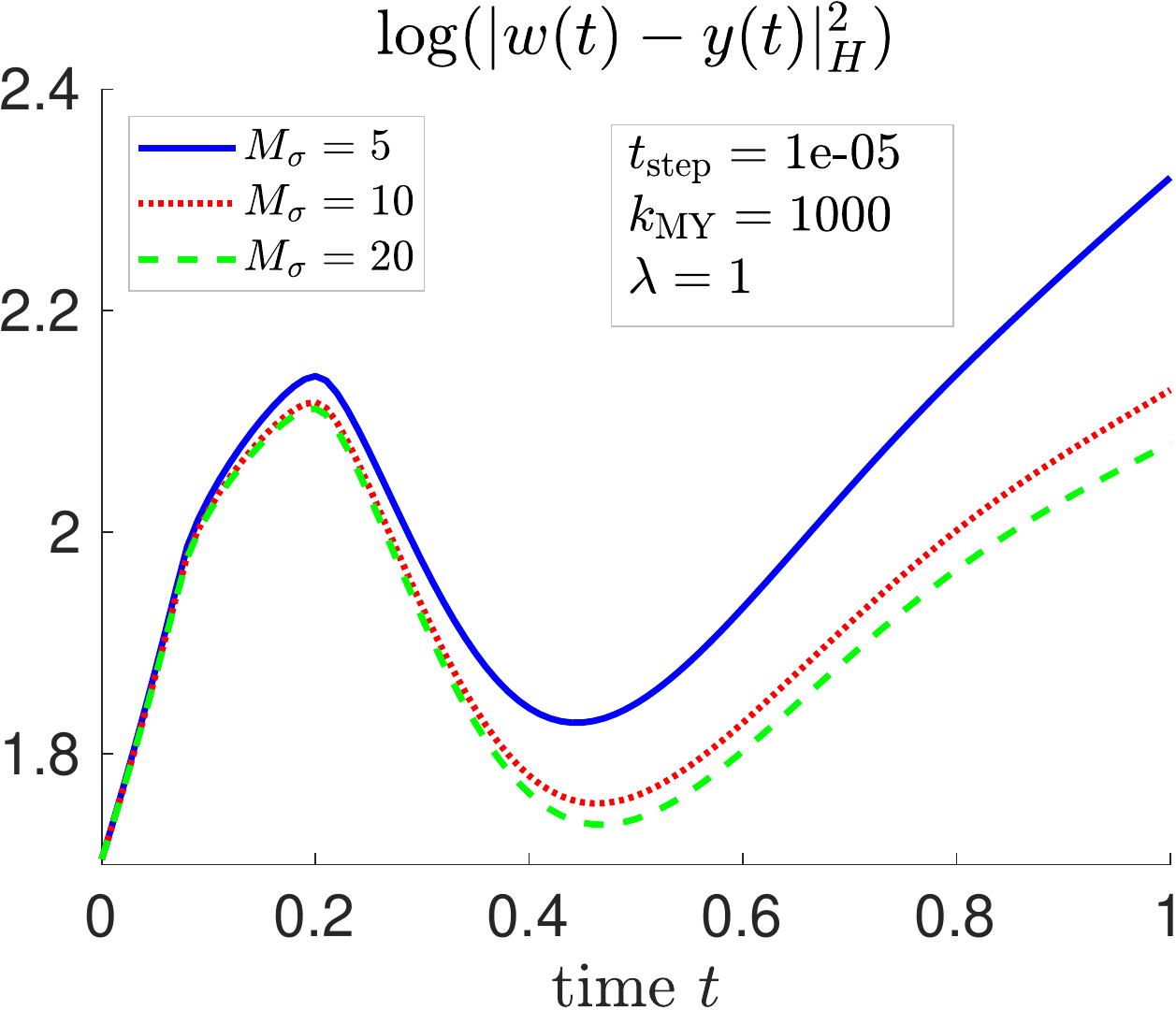}}
\qquad
\subfigure
{\includegraphics[width=0.45\textwidth,height=0.31\textwidth]{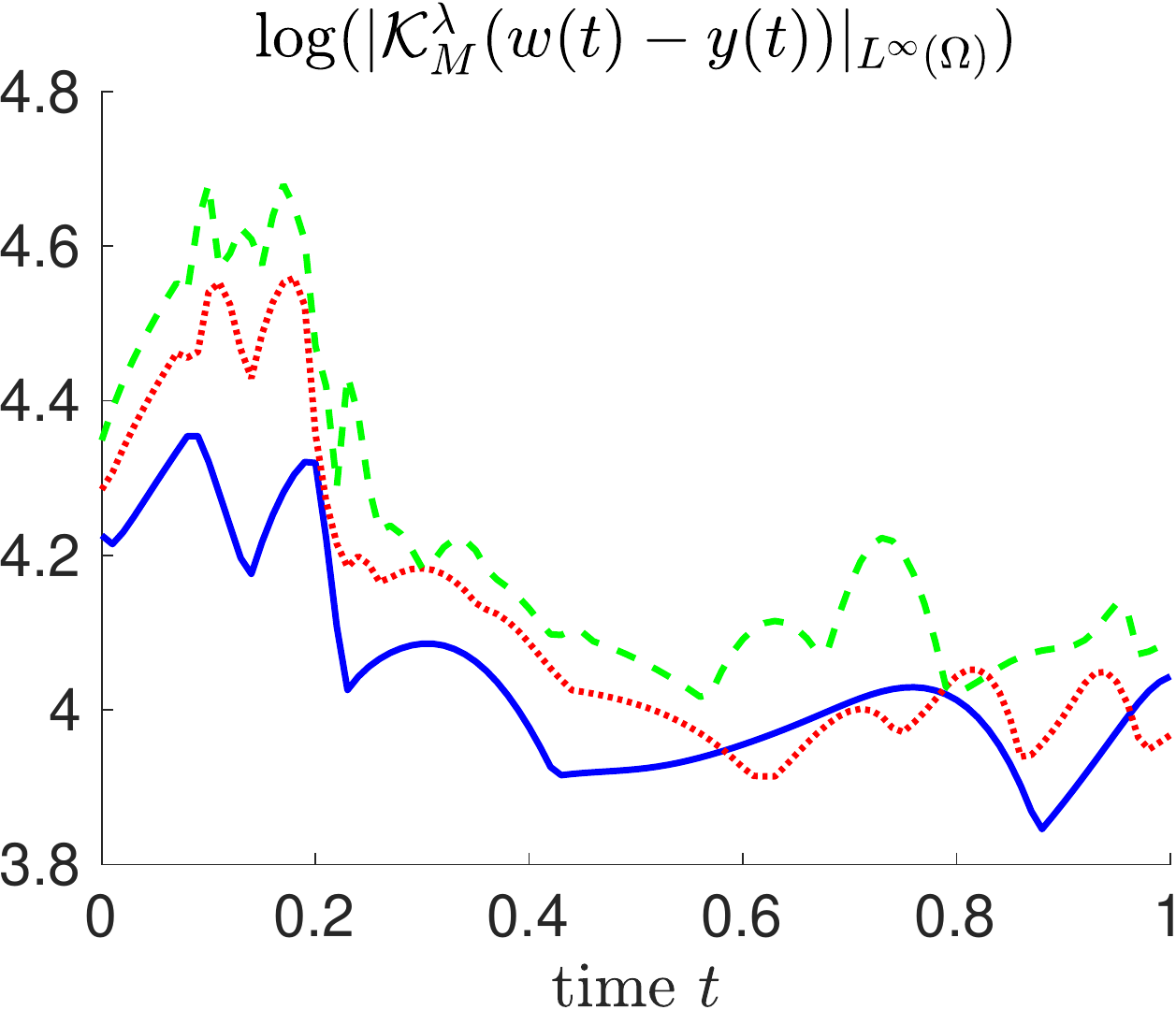}}
\caption{Norms of difference to targeted state and of  control}
\label{Fig:MDF_lam1M20Feed01T1}
\end{figure}
we see that with~$\lambda=1$
we cannot stabilize the system, even if we take~$20$ actuators. Furthermore, for small time we cannot see a considerable change in the norm of the difference to the target for the several~$M_\sigma$s. This allow us to extrapolate that it is necessary to take~$\lambda>1$ if we want to stabilize the system.

In Figure~\ref{Fig:MTw_lam1M20Feed01T1} we present time snapshots of trajectories and control. We see that with~$10$ and~$20$ actuators we cannot see a strong enough change on the evolution of the trajectory to expect (or, hope for) a stabilization effect for large values of $M_\sigma$.
\begin{figure}[ht]
\centering
\subfigure
{\includegraphics[width=0.45\textwidth,height=0.31\textwidth]{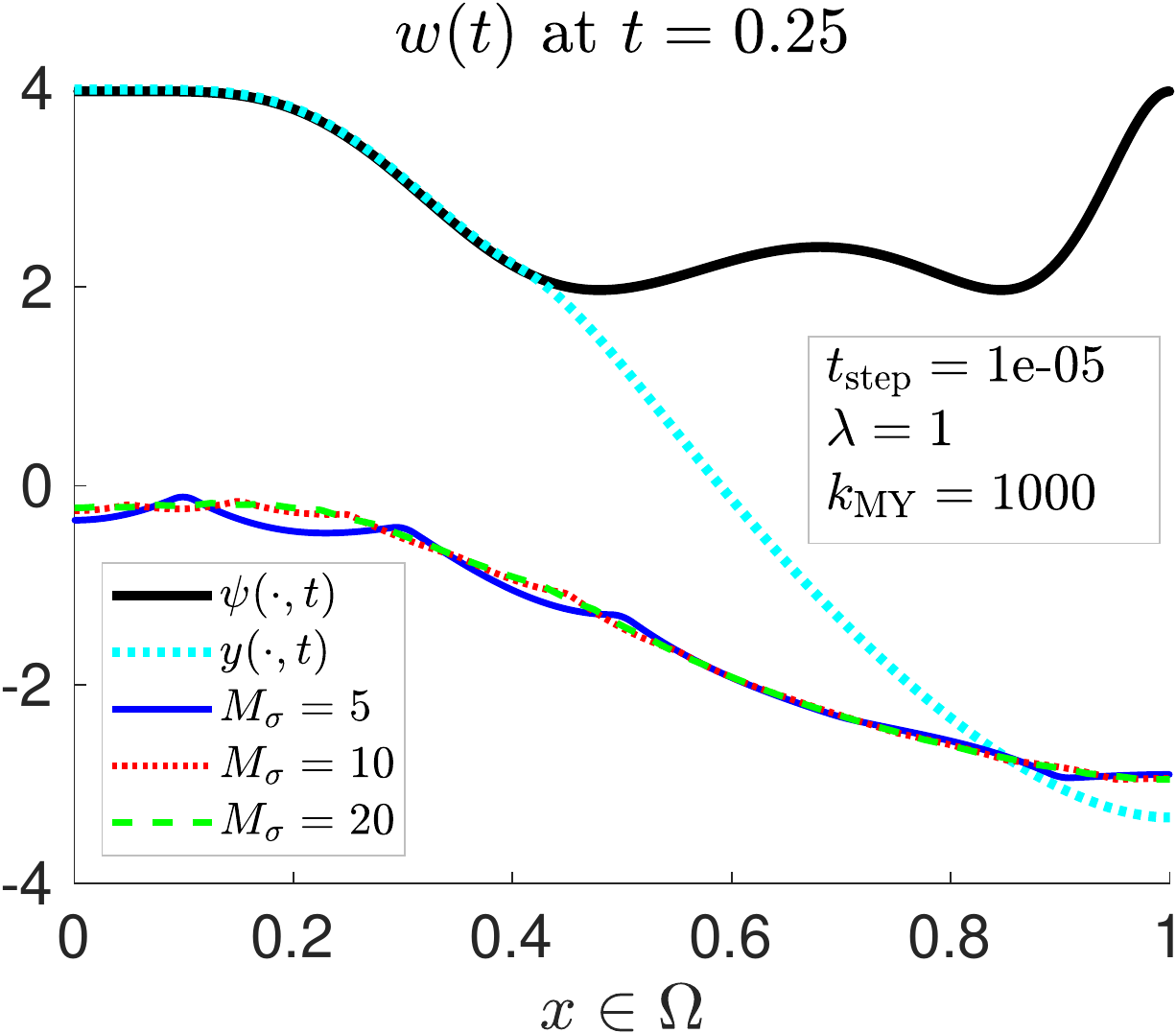}}
\qquad
\subfigure
{\includegraphics[width=0.45\textwidth,height=0.31\textwidth]{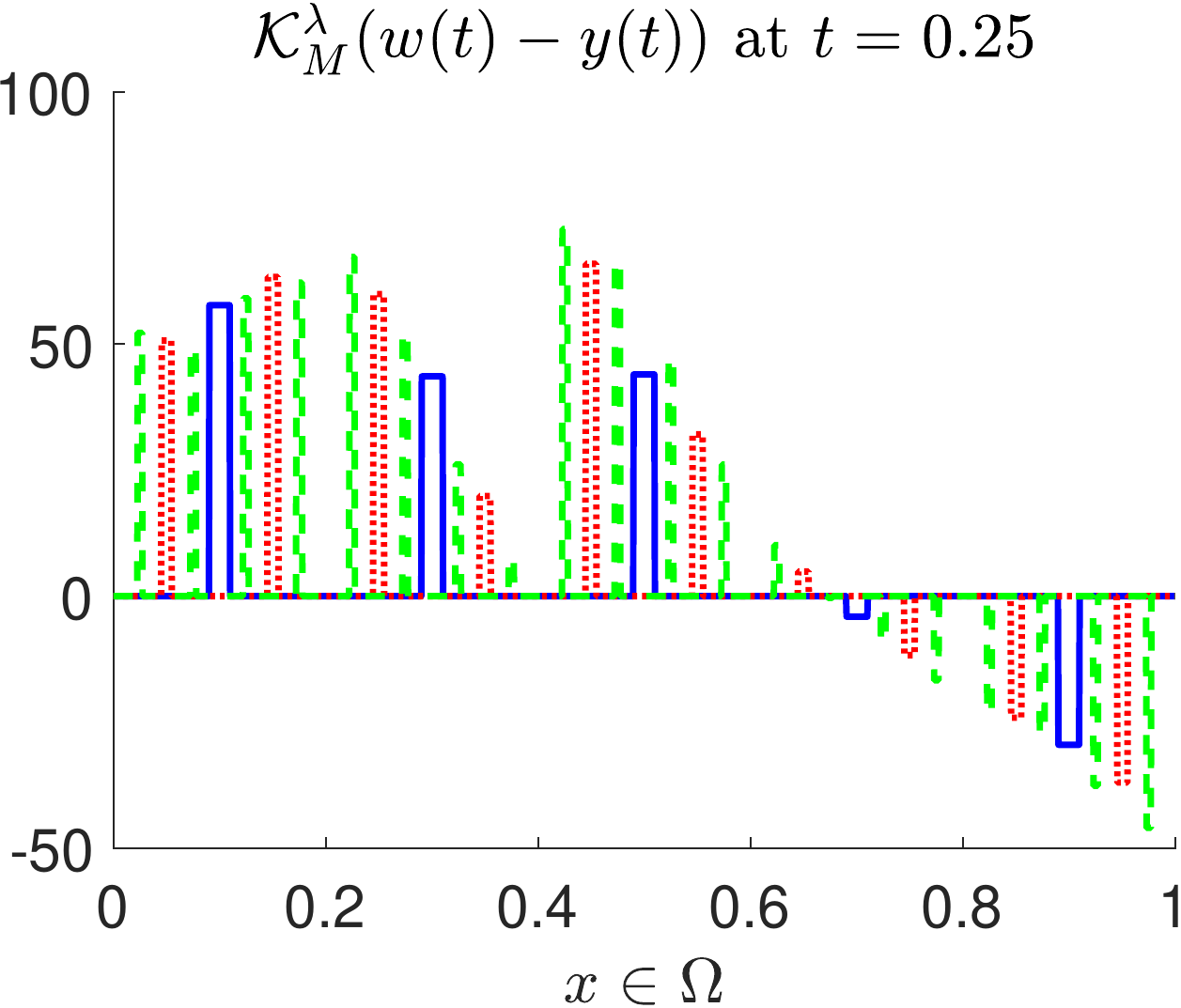}}
\caption{Time snapshots of controlled state}
\label{Fig:MTw_lam1M20Feed01T1}
\end{figure}

\subsubsection{On the achievement of an arbitrarily small exponential decreasing rate~$-\mu<0$}\label{ssS:achmu}

From our result we can reach an arbitrarily small exponential decreasing rate~$-\mu$, provided we take both~$M_\sigma$ and~$\lambda$ large enough. This is shown in Figure~\ref{Fig:rmuDF_lam6M10Feed04T4},
\begin{figure}[ht]
\centering
\subfigure
{\includegraphics[width=0.45\textwidth,height=0.31\textwidth]{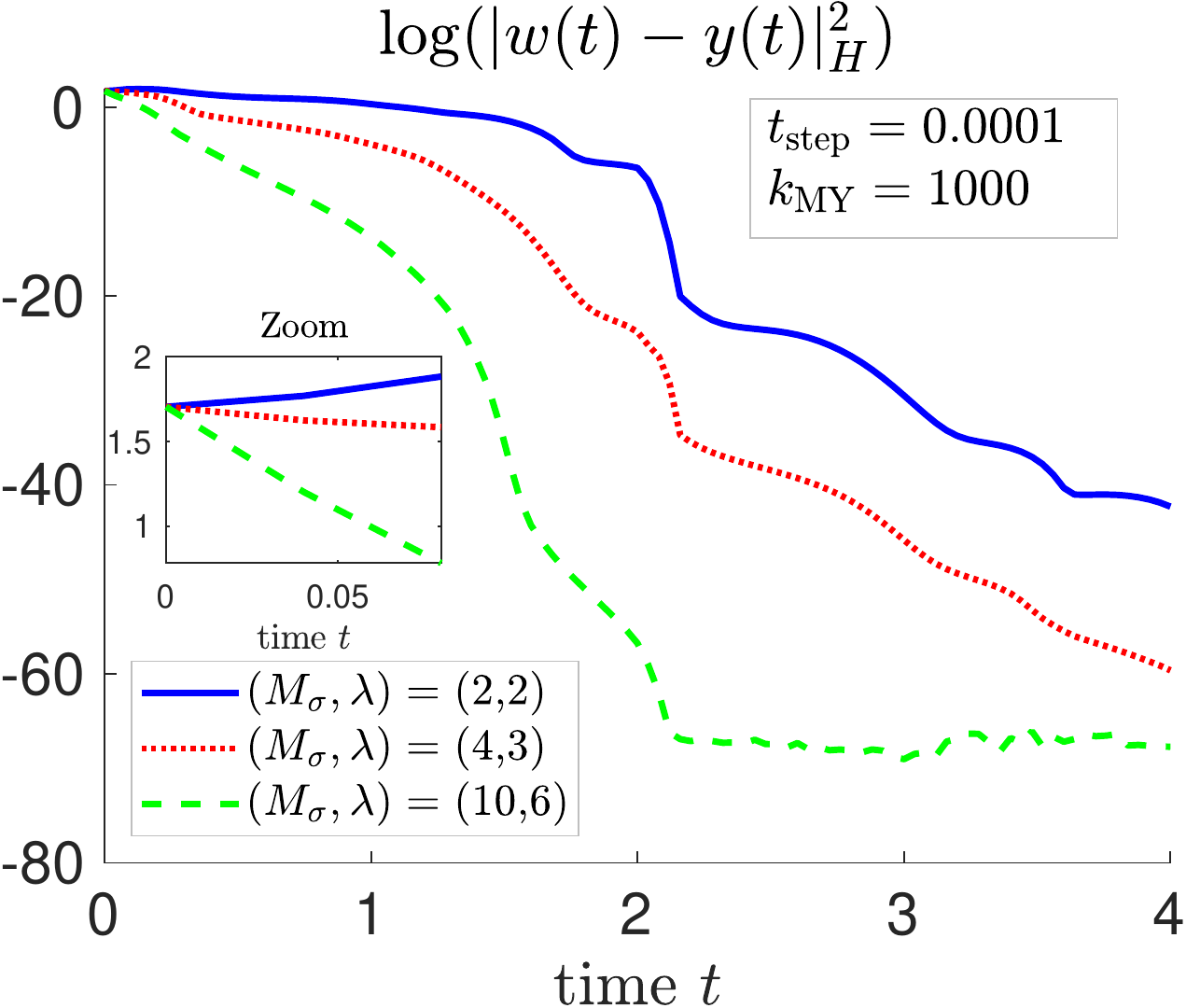}}
\qquad
\subfigure
{\includegraphics[width=0.45\textwidth,height=0.31\textwidth]{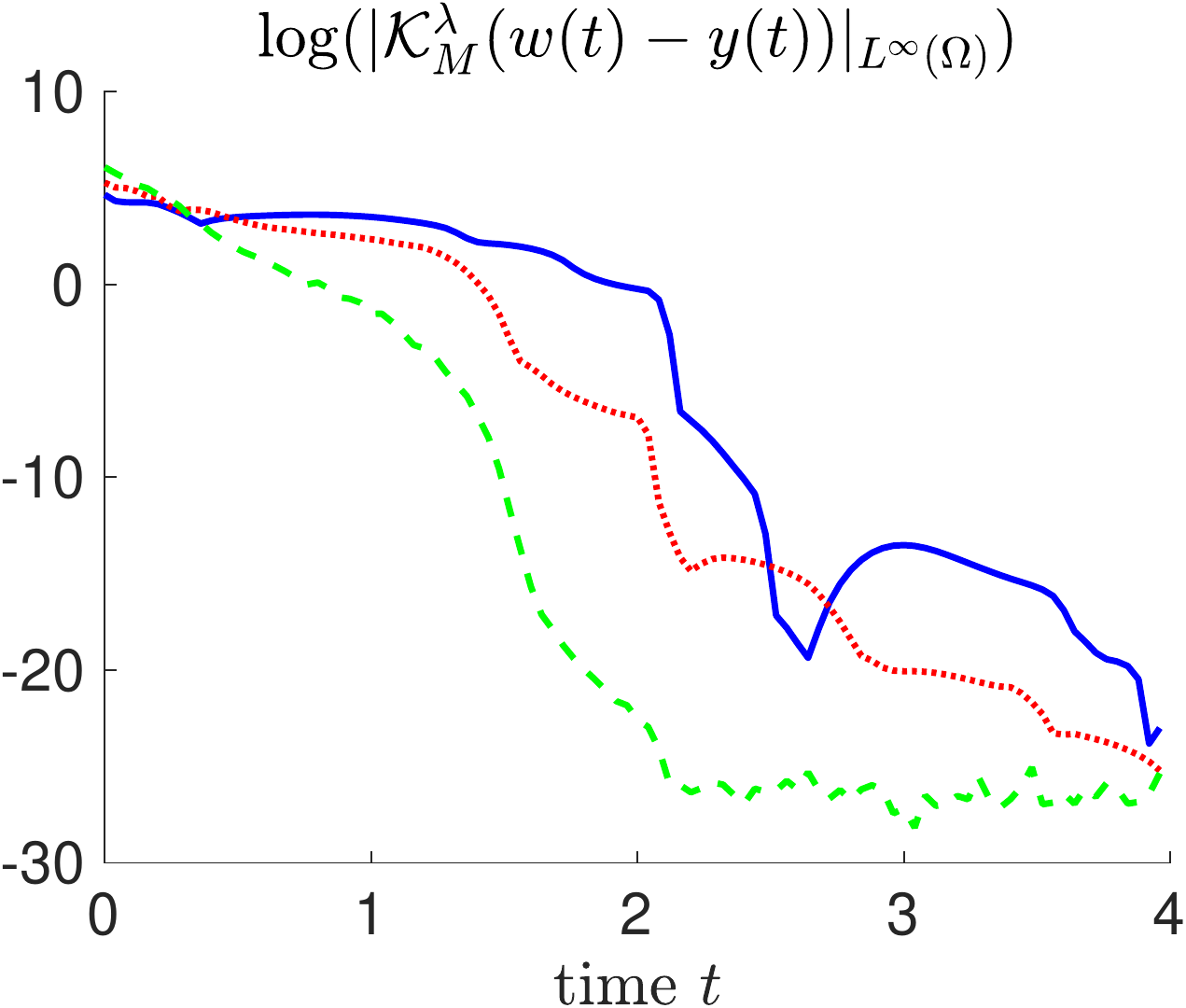}}
\caption{Norms of difference to targeted state and of  control}
\label{Fig:rmuDF_lam6M10Feed04T4}
\end{figure}
 where we see that with~$(M_\sigma,\lambda)=(10,6)$ we obtain a smaller exponential rate than with~$(M_\sigma,\lambda)=(4,3)$. We also observe that with~$(M_\sigma,\lambda)=(2,2)$ we are also able to stabilize the system, however this case does not fully confirm our result, where we can also guarantee that the norm of the difference to the targeted trajectory is strictly decreasing. In the zoomed subplot, in Figure~\ref{Fig:rmuDF_lam6M10Feed04T4}, we can see that
for small time the norm of the difference in not strictly decreasing,  for~$(M_\sigma,\lambda)=(2,2)$.

The time snapshots in Figure~\ref{Fig:rmuTwF_lam6M10Feed04T4}
\begin{figure}[ht]
\centering
\subfigure
{\includegraphics[width=0.45\textwidth,height=0.31\textwidth]{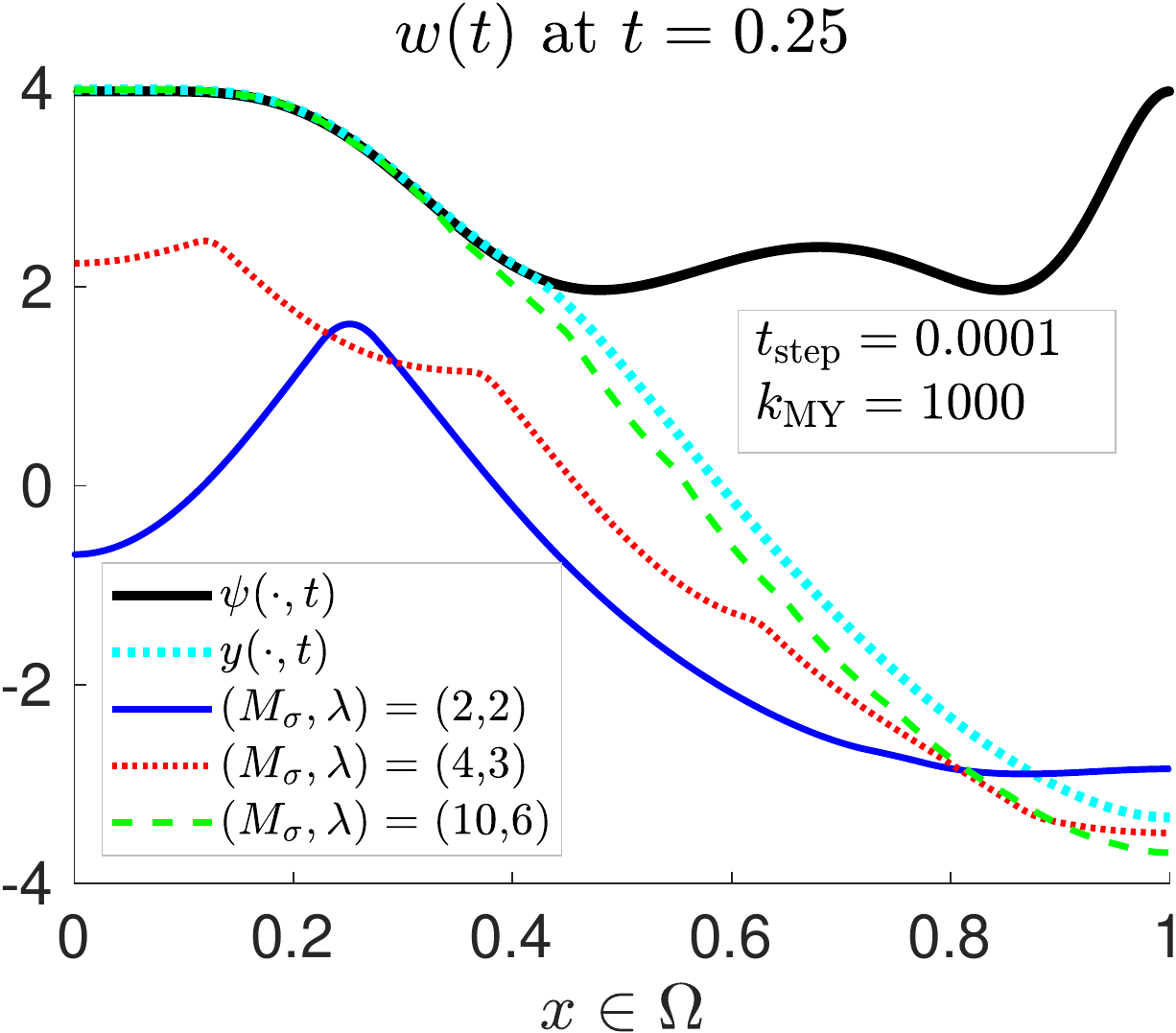}}
\qquad
\subfigure
{\includegraphics[width=0.45\textwidth,height=0.31\textwidth]{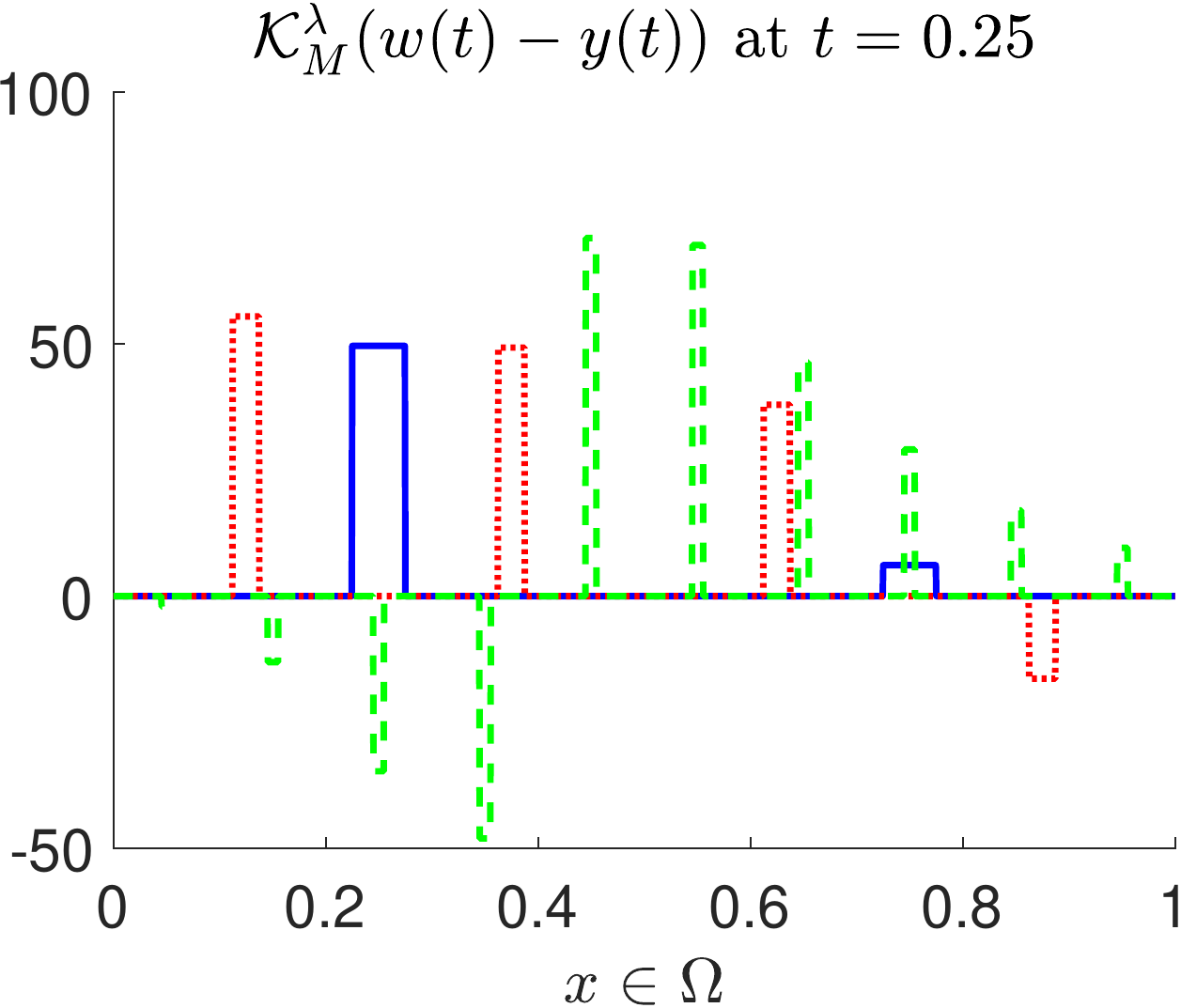}}
\caption{Time snapshots of trajectories and controls}
\label{Fig:rmuTwF_lam6M10Feed04T4}
\end{figure}
also confirm that with a pair~$(M_\sigma,\lambda)$ with larger coordinates, we obtain a faster convergence of the controlled trajectory~$w$ to the targeted one~$y$.

\subsection{The uncontrolled dynamics}\label{ssS:FeedOn}

Here we show that the uncontrolled dynamics is unstable. That is, a control is necessary to stabilize the system to the targeted trajectory. In Figure~\ref{Fig:FeedOnDF_lam5M6Feed14T4} the symbol~$FeedOn$ denotes the time interval where the feedback control is switched on. Thus, outside this time interval the free (uncontrolled) dynamics is followed.
\begin{figure}[ht]
\centering
\subfigure
{\includegraphics[width=0.45\textwidth,height=0.31\textwidth]{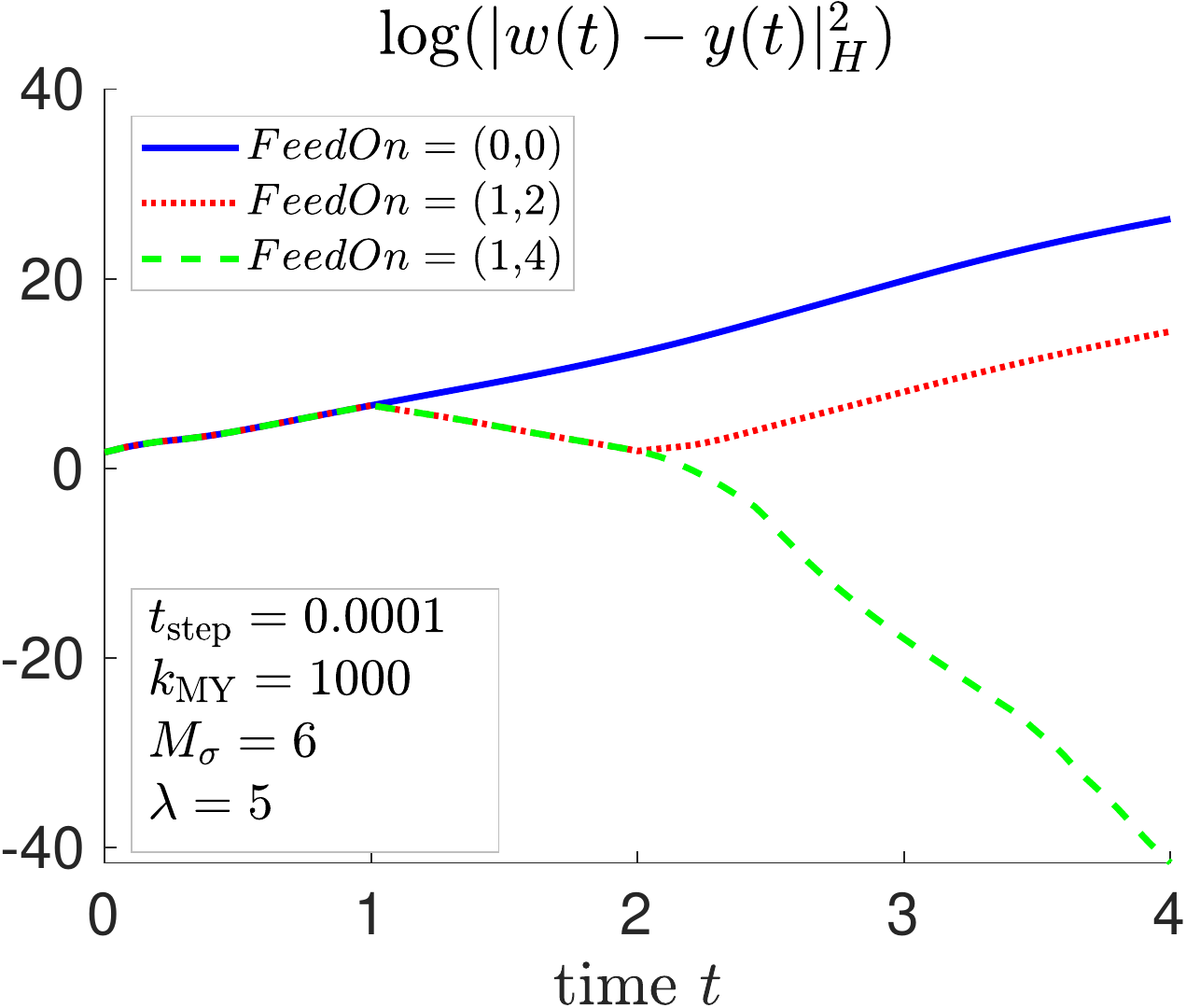}}
\qquad
\subfigure
{\includegraphics[width=0.45\textwidth,height=0.31\textwidth]{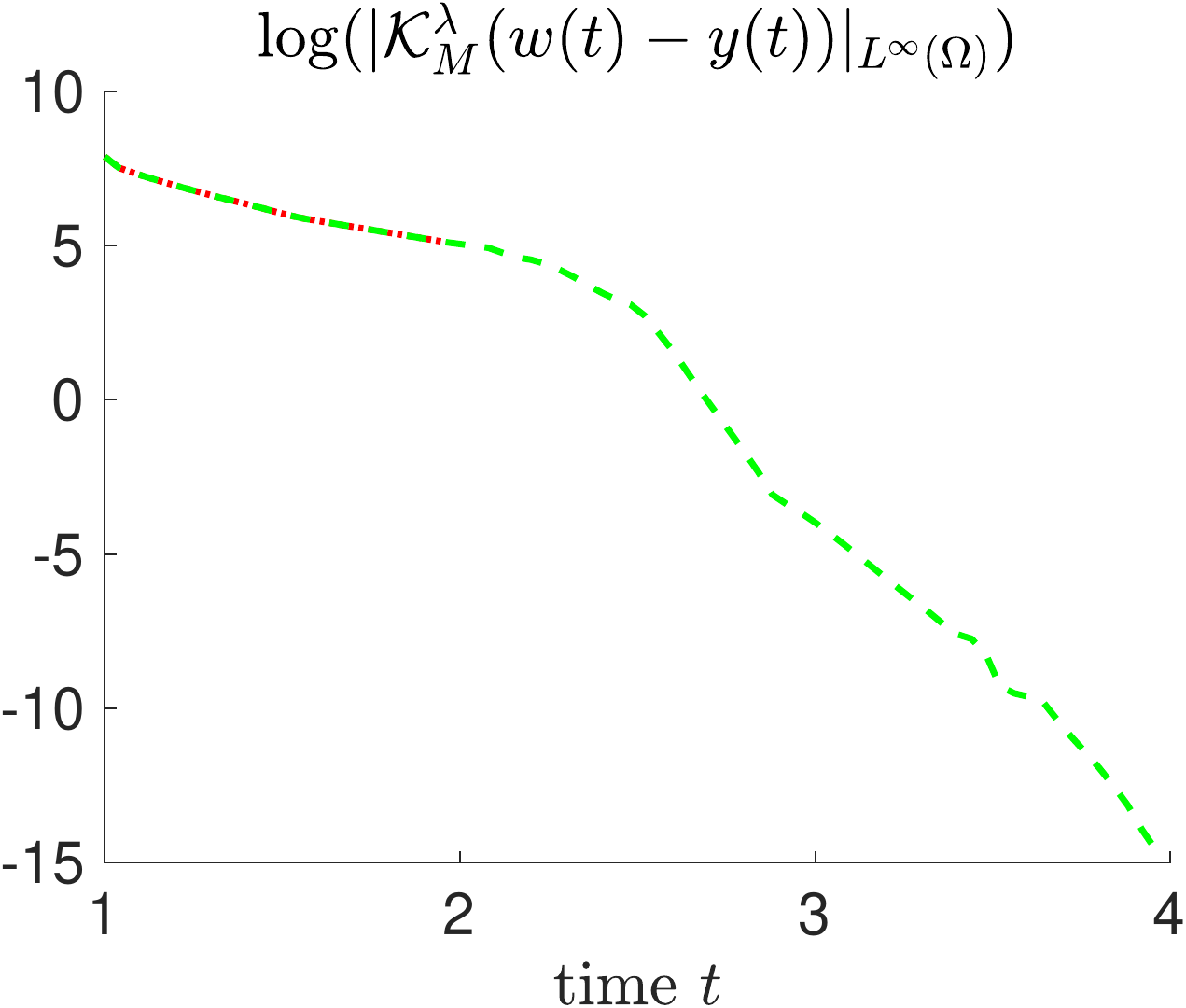}}
\caption{Norms of difference to targeted state and of  control}
\label{Fig:FeedOnDF_lam5M6Feed14T4}
\end{figure}
We see that the free dynamics is exponentially unstable, as the norm of the difference to the target {\em increases} exponentially when the control is switched off. On the other hand, when the control is switched on we see that such norm {\em decreases} exponential, confirming again our theoretical stabilizability results.

Time snapshots in Figure~\ref{Fig:FeedOnTwF_lam5M6Feed14T4} show again that the trajectory~$w$ corresponding to the free dynamics~$FeedOn=(0,0)$ is not approaching the targeted one~$y$ as time increases (cf.~Figure~\ref{Fig:InitStates}).
\begin{figure}[ht]
\centering
\subfigure
{\includegraphics[width=0.45\textwidth,height=0.31\textwidth]{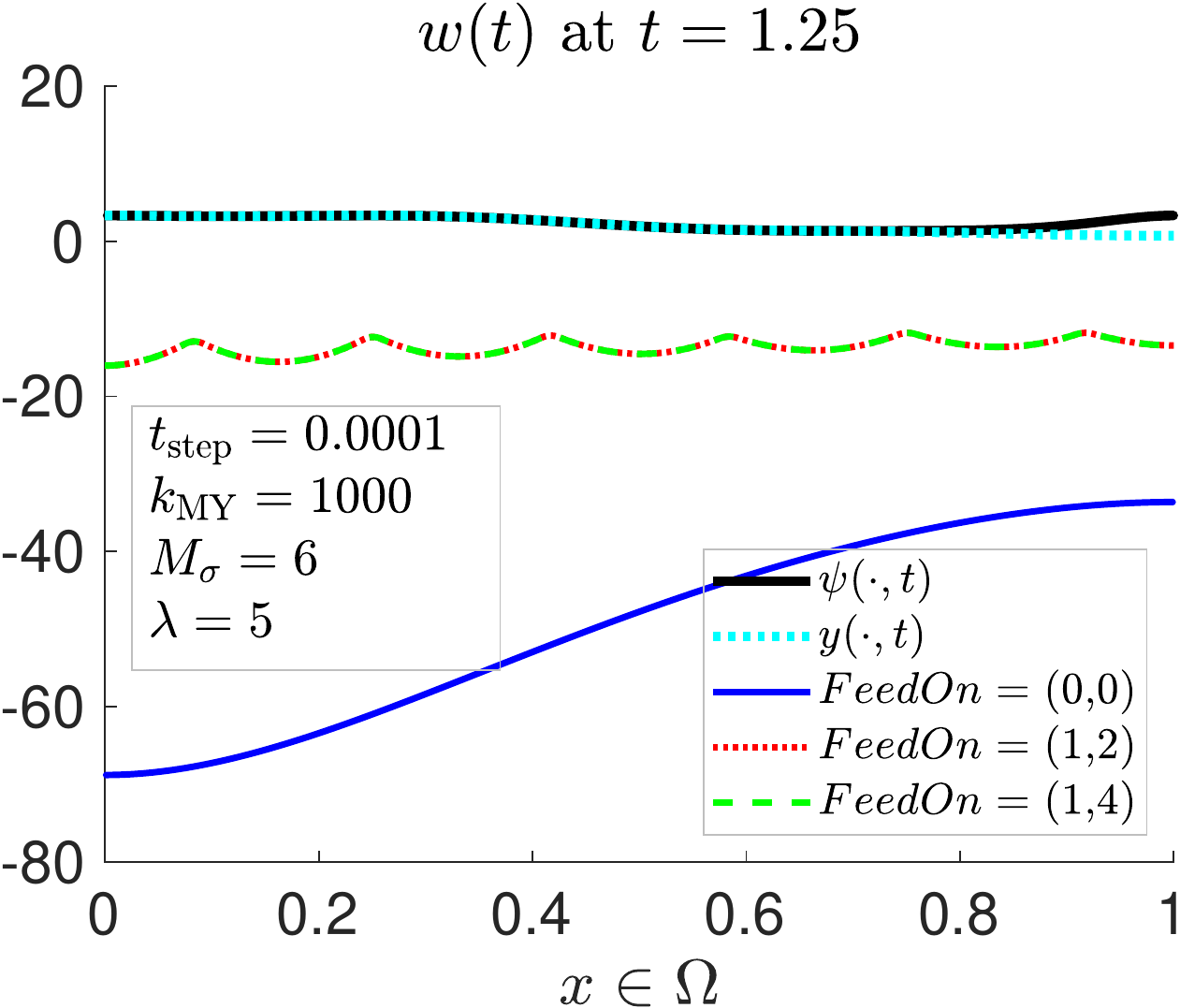}}
\qquad
\subfigure
{\includegraphics[width=0.45\textwidth,height=0.31\textwidth]{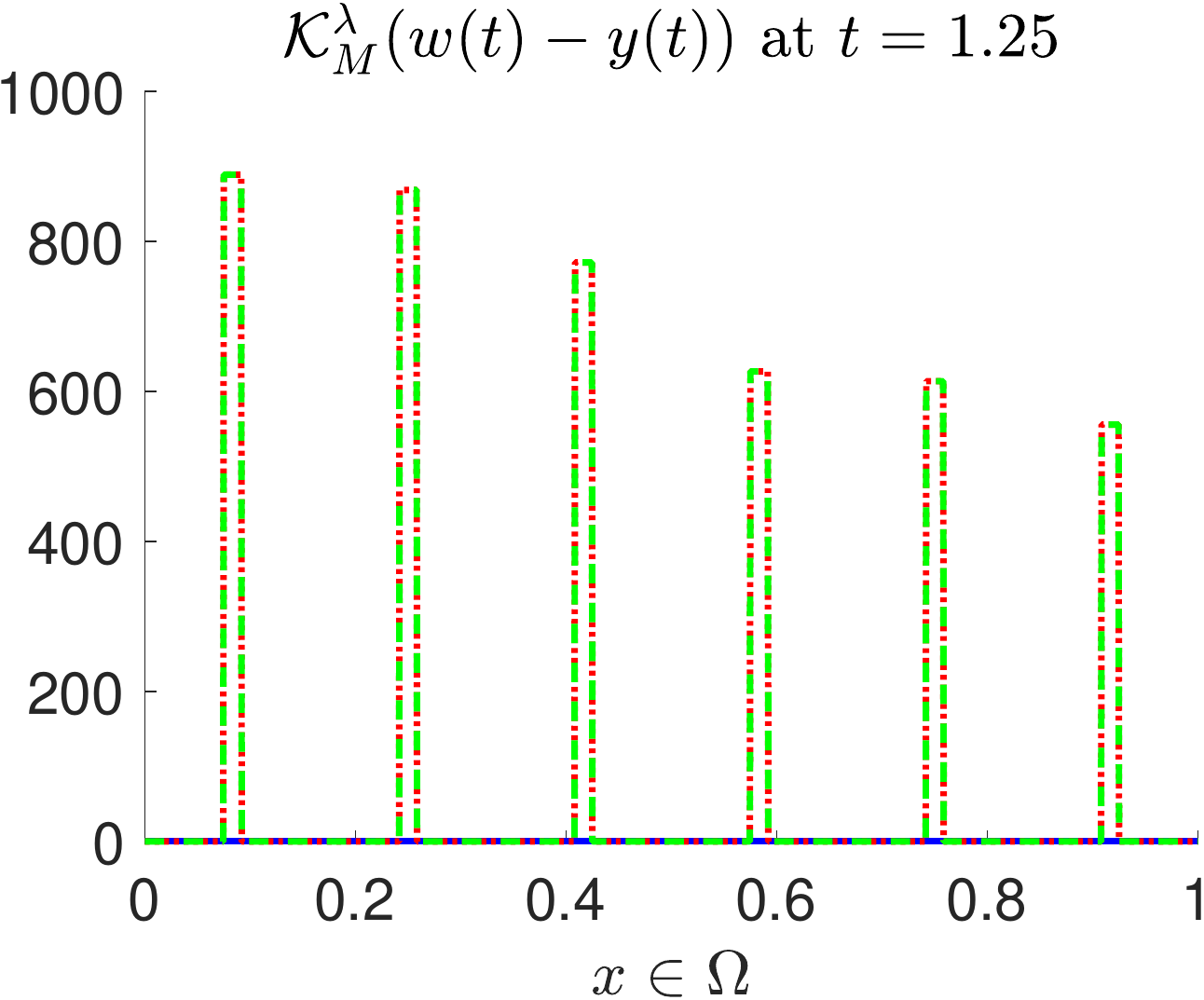}}
\caption{Time snapshots of trajectories and controls}
\label{Fig:FeedOnTwF_lam5M6Feed14T4}
\end{figure}

\subsection{Evolution of the contact set and the Moreau--Yosida parameter}\label{S:KoSet}

Here, we investigate the evolution of the contact (or, active) set.
In Figure~\ref{Fig:SKoSetDF_lam5M10Feed02T2}
\begin{figure}[ht]
\centering
\subfigure
{\includegraphics[width=0.45\textwidth,height=0.31\textwidth]{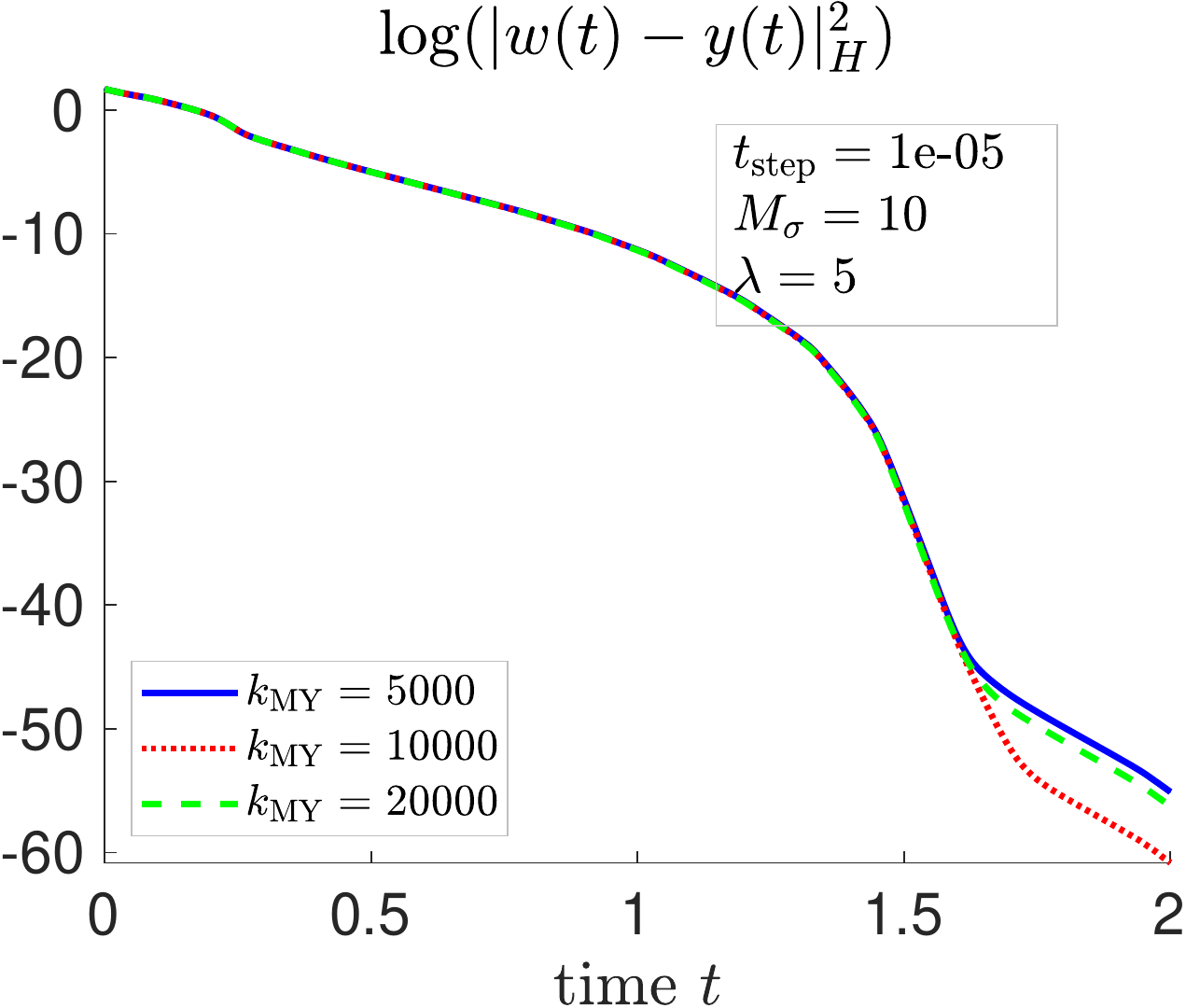}}
\qquad
\subfigure
{\includegraphics[width=0.45\textwidth,height=0.31\textwidth]{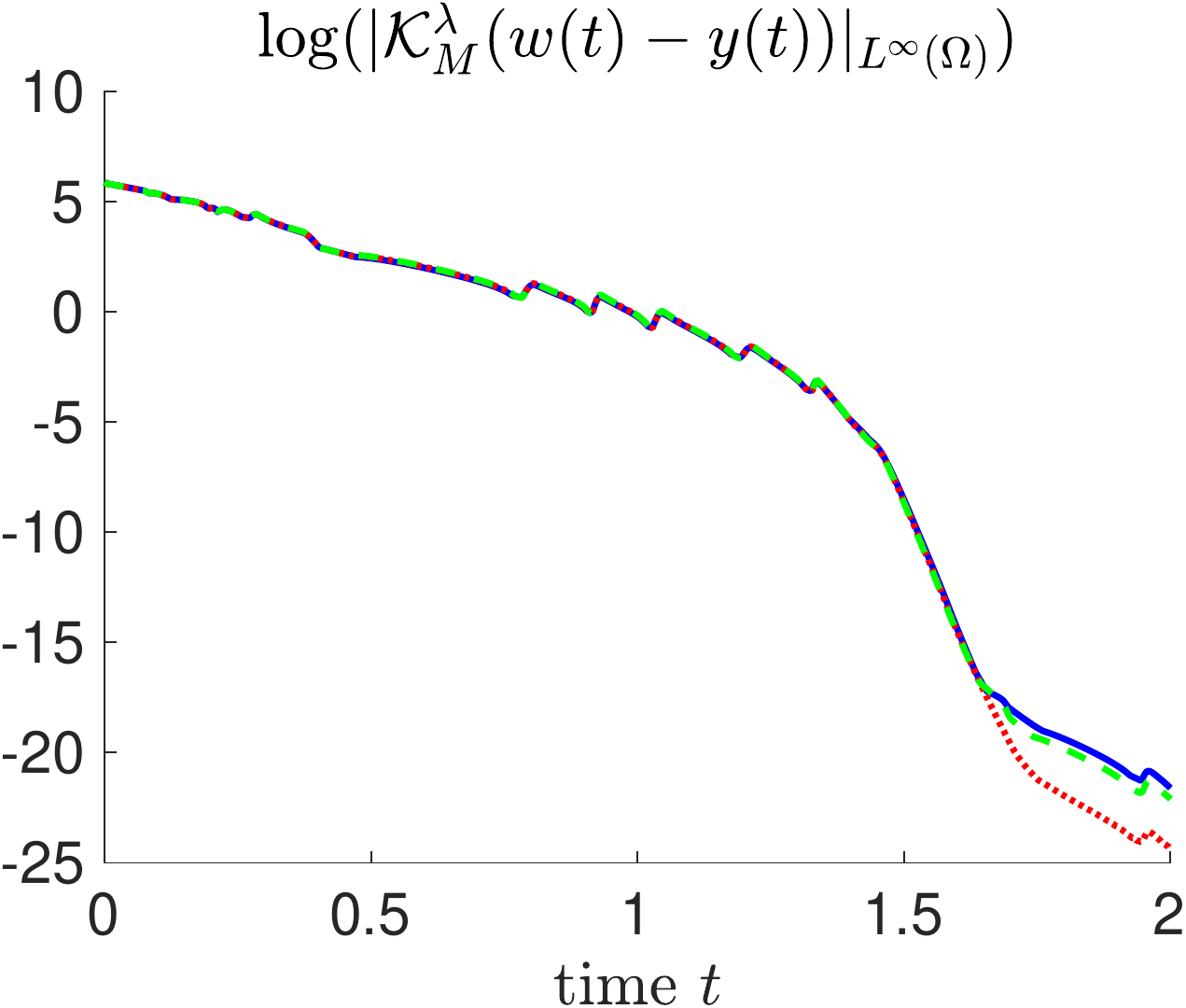}}
\caption{Norms of difference to target and control}
\label{Fig:SKoSetDF_lam5M10Feed02T2}
\end{figure}
 we see that  the behavior of the norm of the difference to target and of the control is similar for the several Moreau--Yosida parameters, with some differences for time~$t\ge 1.5$. So, the considered 
 parameters give us already a good picture of the qualitative behavior of the limit difference and control as $k_{\rm MY}$ diverges to~$+\infty$. 

The time snapshots in Figure~\ref{Fig:SKoSetTOF_lam5M10Feed02T2} 
\begin{figure}[ht]
\centering
\subfigure
{\includegraphics[width=0.45\textwidth,height=0.31\textwidth]{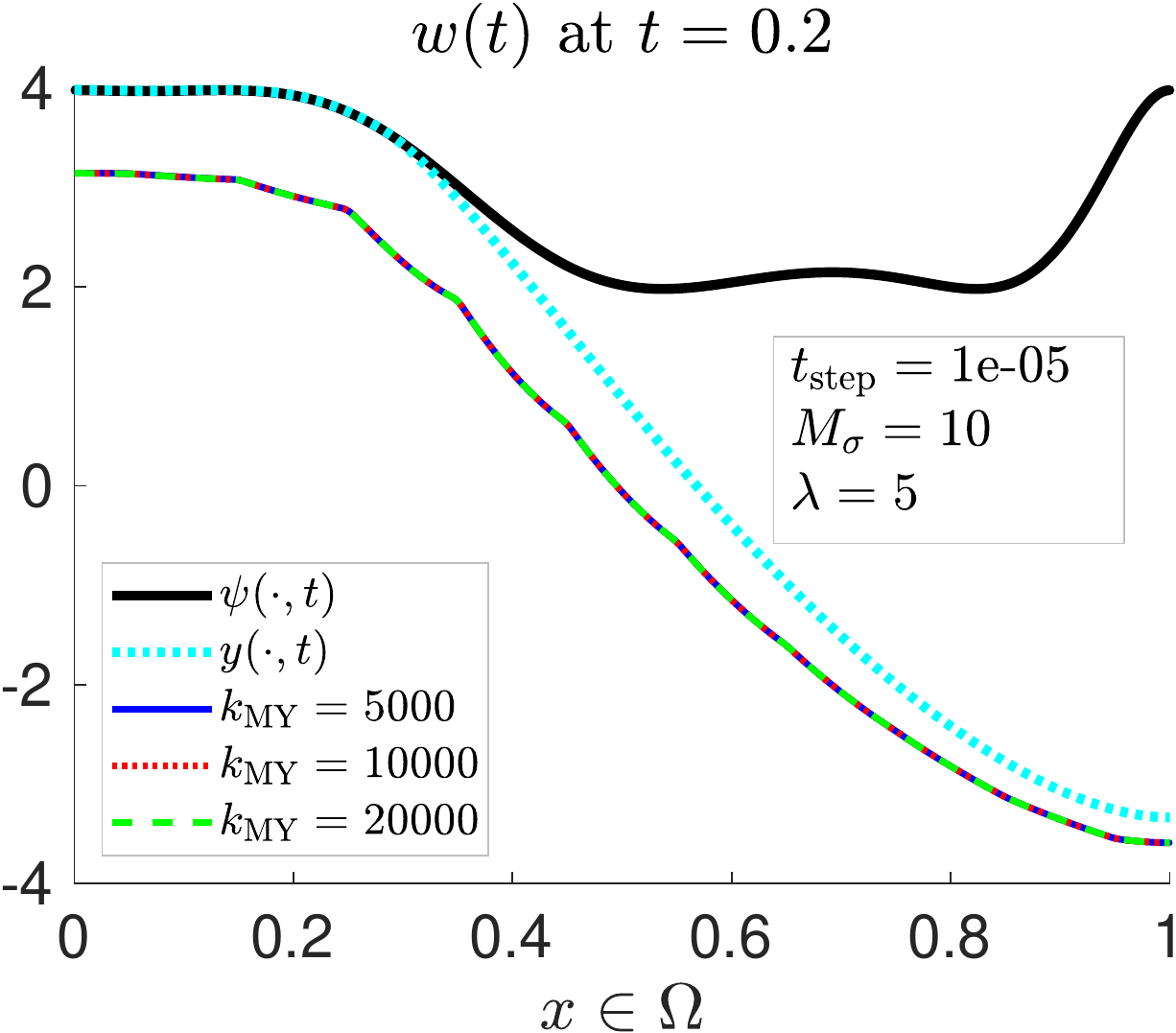}}
\qquad
\subfigure
{\includegraphics[width=0.45\textwidth,height=0.31\textwidth]{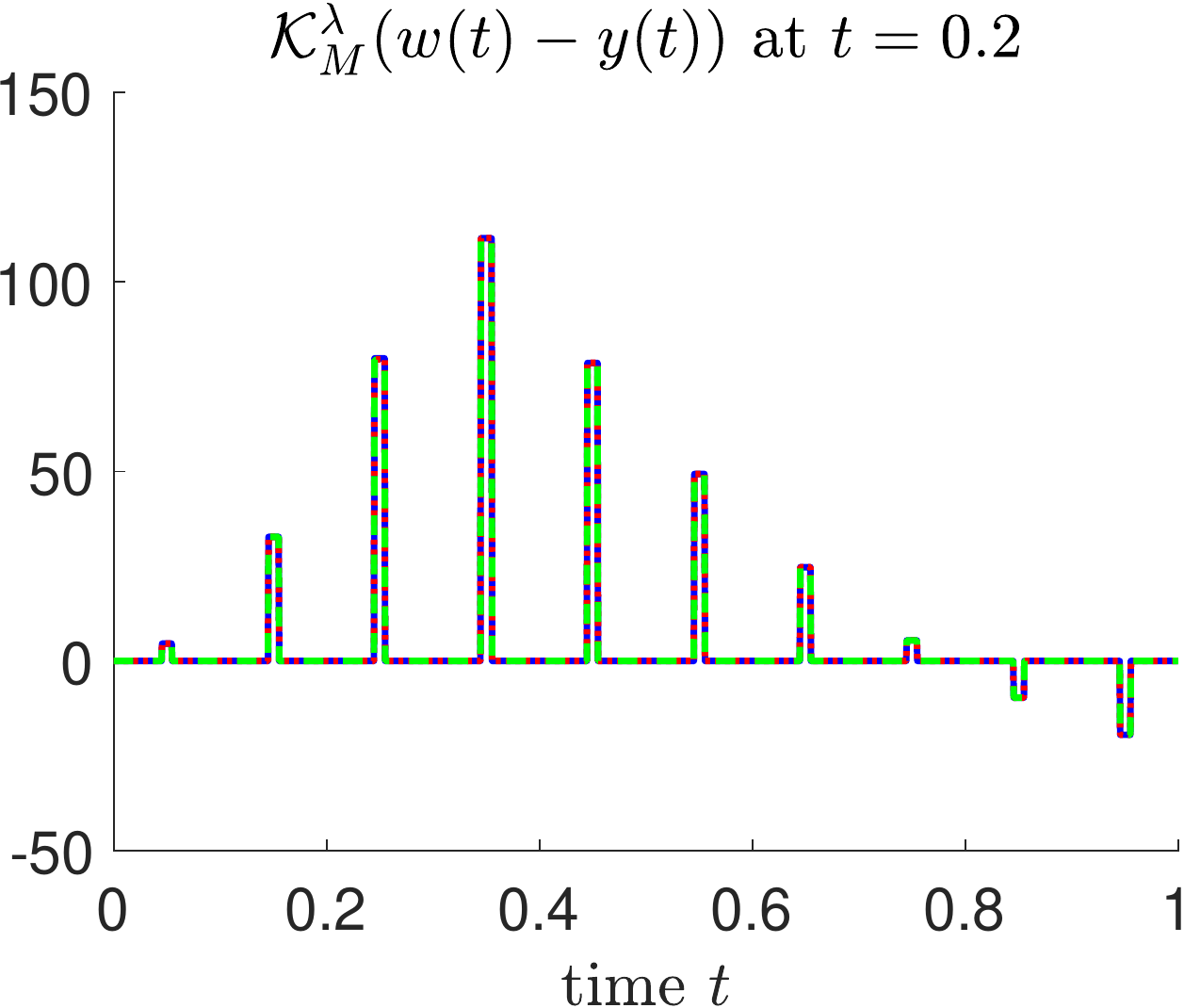}}
\subfigure
{\includegraphics[width=0.45\textwidth,height=0.31\textwidth]{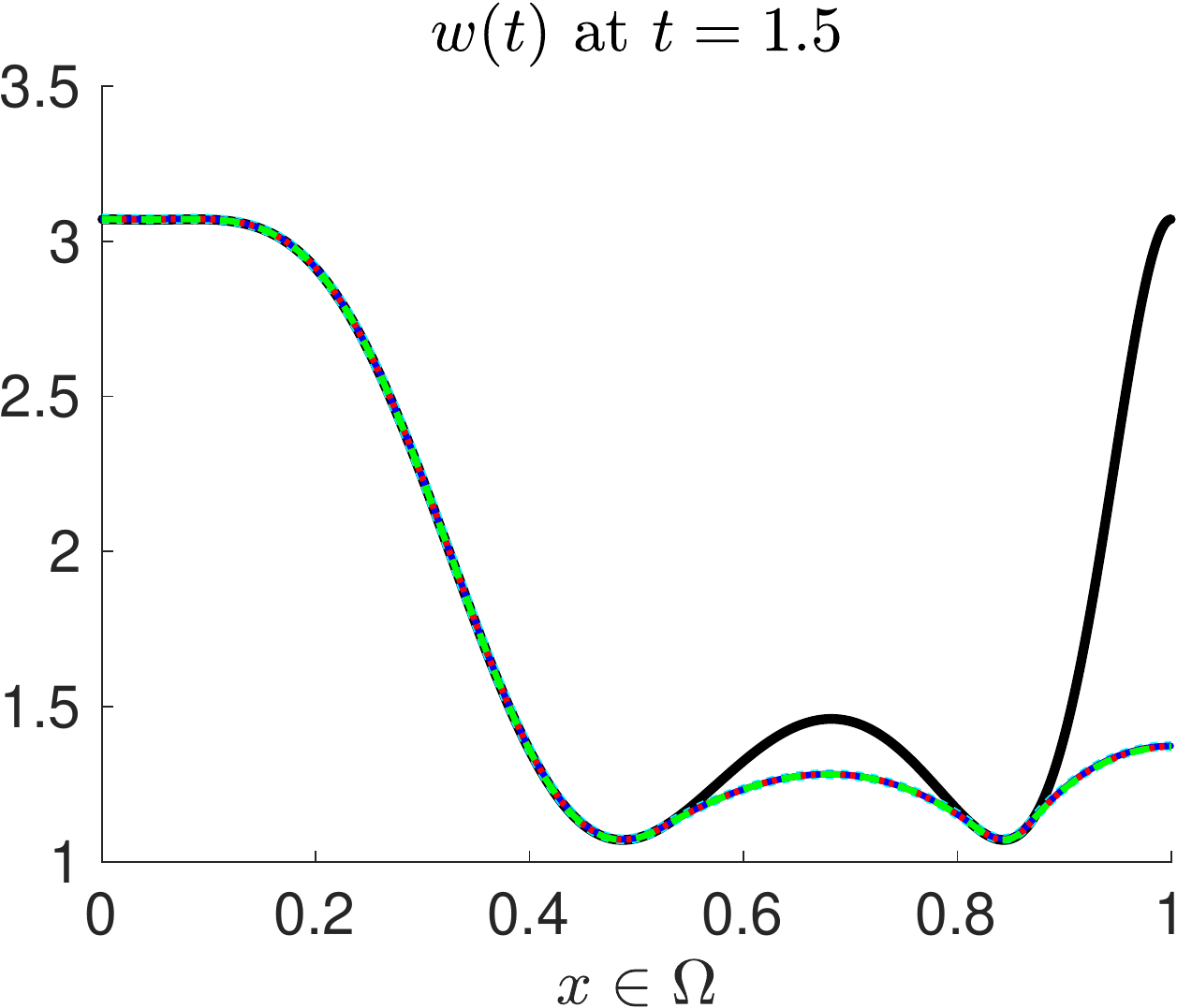}}
\qquad
\subfigure
{\includegraphics[width=0.45\textwidth,height=0.31\textwidth]{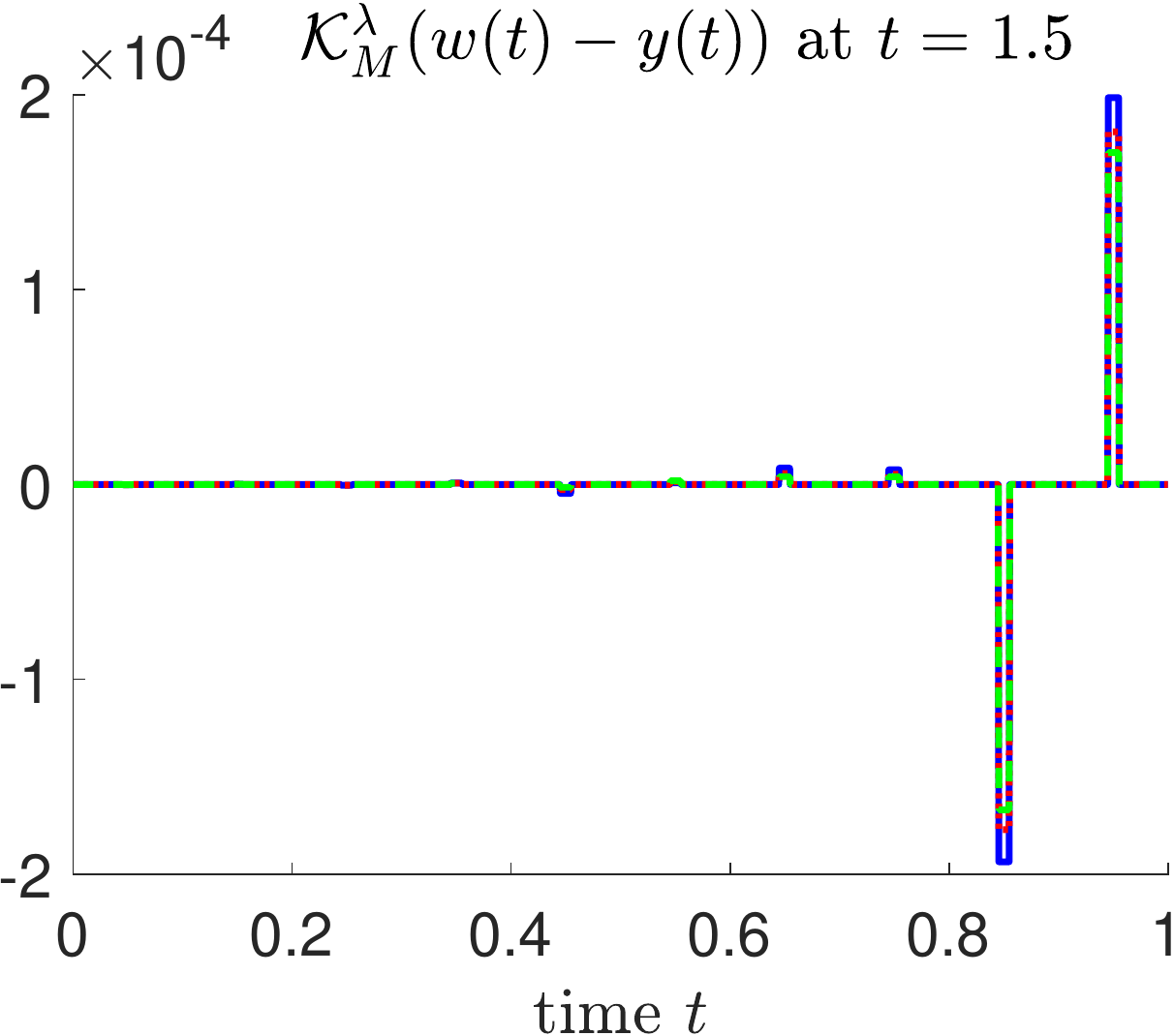}}
\caption{Time snapshots of trajectories and control}
\label{Fig:SKoSetTOF_lam5M10Feed02T2}
\end{figure}
show that the smallest value of~$k_{\rm MY}$ already captures a good
picture of the likely limit behavior for the parabolic variational inequality.

From Figure~\ref{Fig:SKoSetOyw_lam5M10Feed02T2}
\begin{figure}[ht]
\centering
\subfigure
{\includegraphics[width=0.45\textwidth,height=0.31\textwidth]{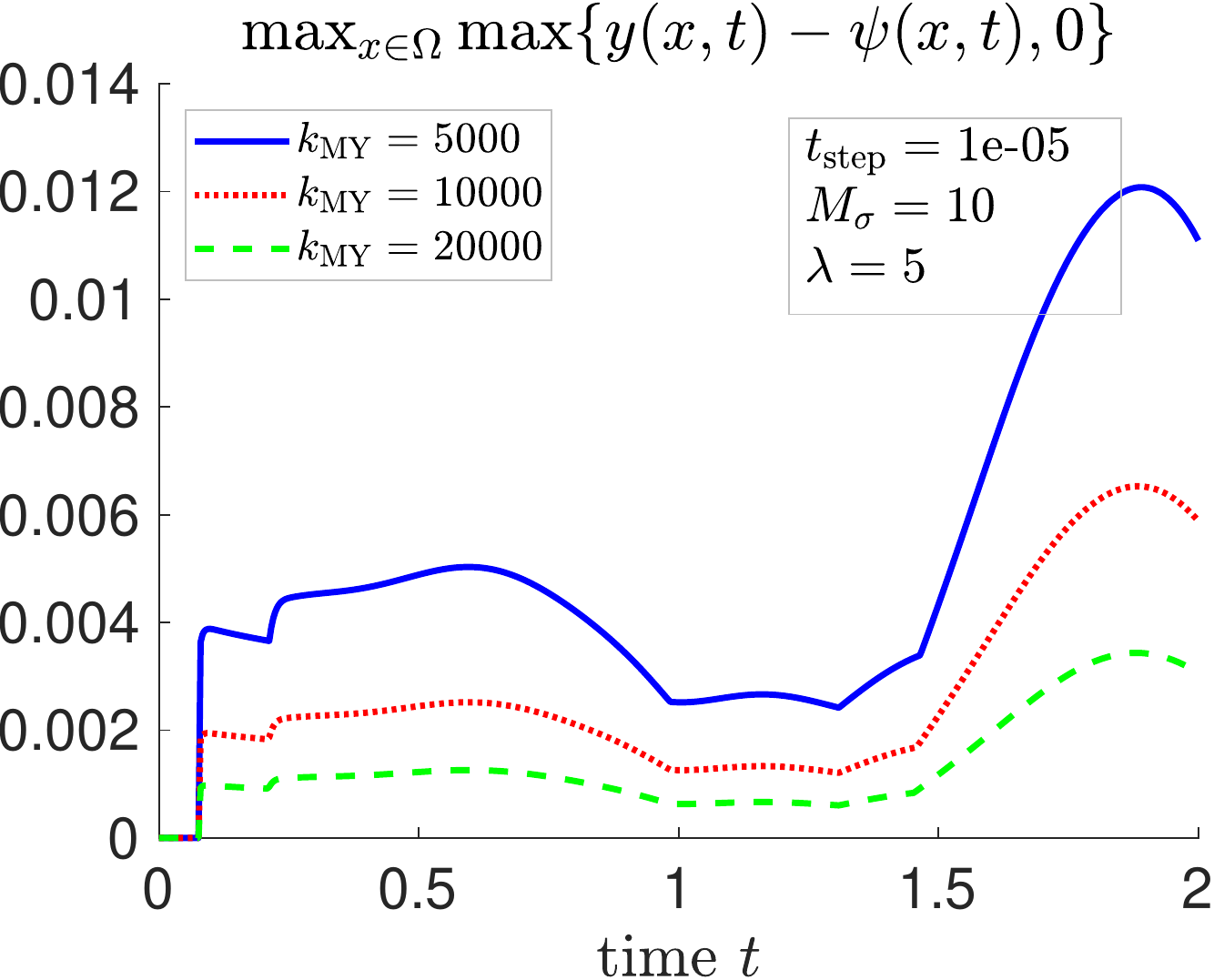}}
\qquad
\subfigure
{\includegraphics[width=0.45\textwidth,height=0.31\textwidth]{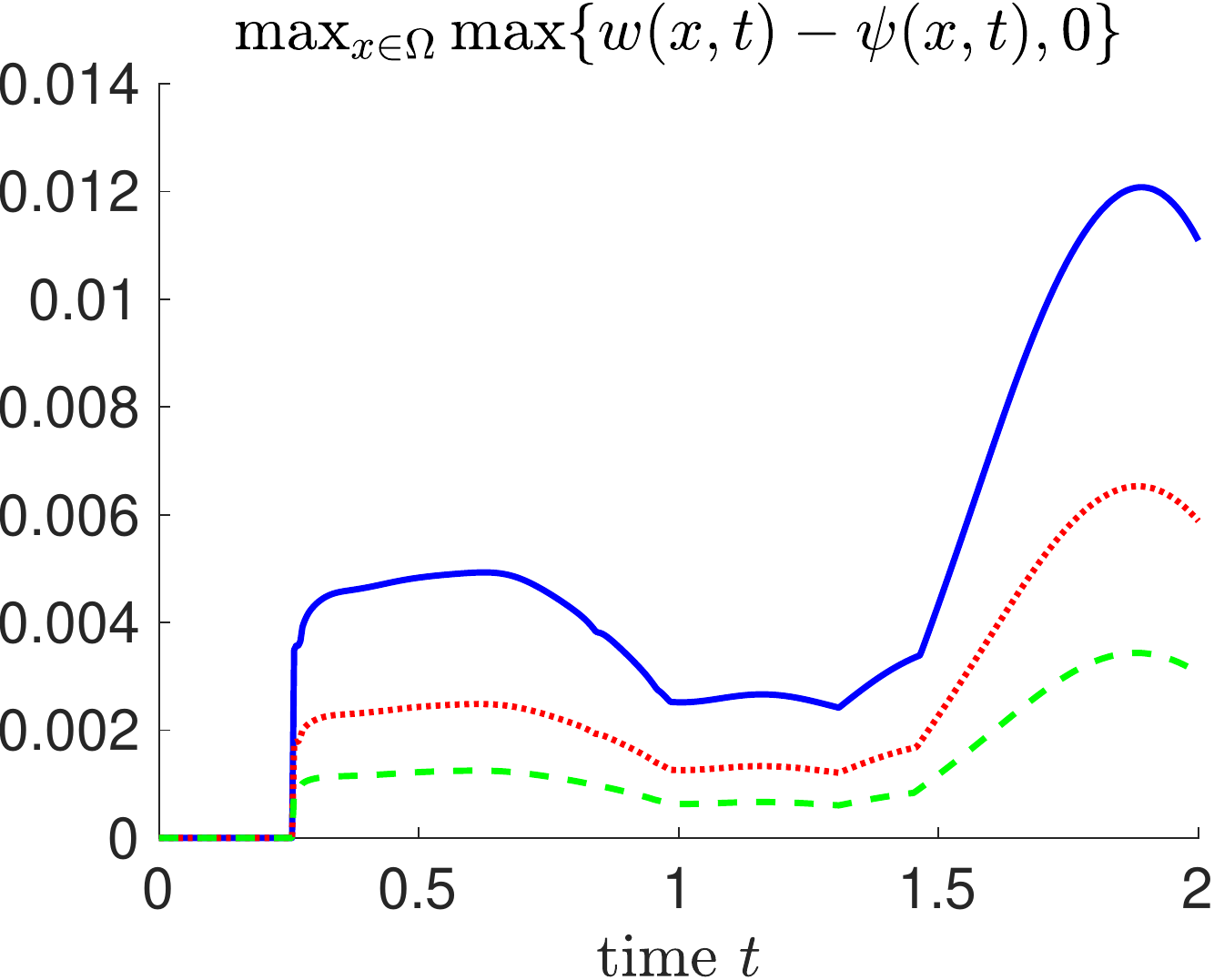}}
\caption{Largest magnitude of obstacle constraint violation}
\label{Fig:SKoSetOyw_lam5M10Feed02T2}
\end{figure}
we can conjecture also that the magnitude of the violation of the obstacle constraint
converges to zero  as $k_{\rm MY}\to\infty$.
That is, at the limit  such magnitude will vanish,
as we expect due to the theoretical results.

Finally,  in  Figures~\ref{Fig:KoSetKSy_lam5M10Feed02T2KTK} and~\ref{Fig:KoSetKSw_lam5M10Feed02T2KTK}
\begin{figure}[ht]
\centering
\subfigure
{\includegraphics[width=0.45\textwidth,height=0.31\textwidth]{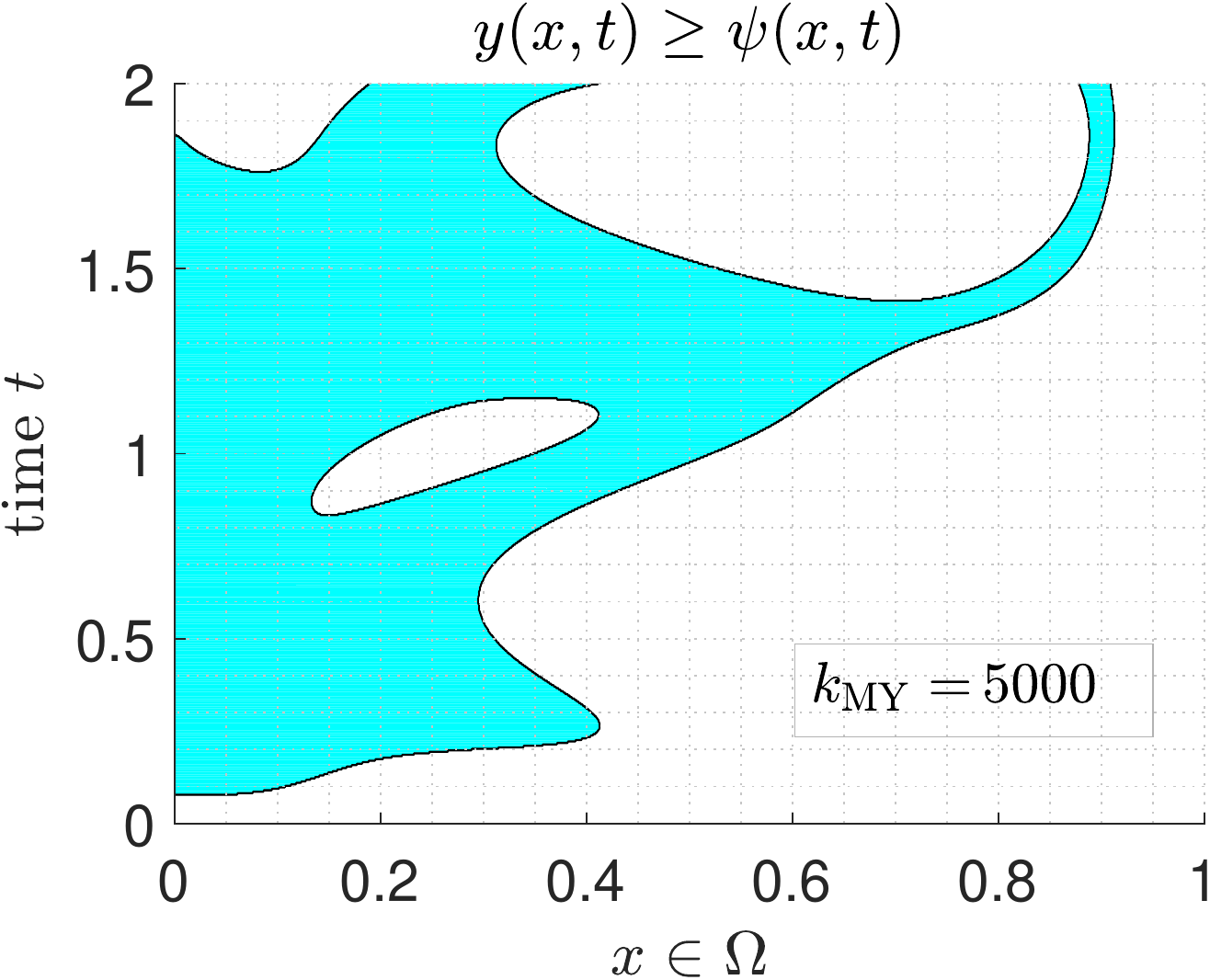}}
\qquad
\subfigure
{\includegraphics[width=0.45\textwidth,height=0.31\textwidth]{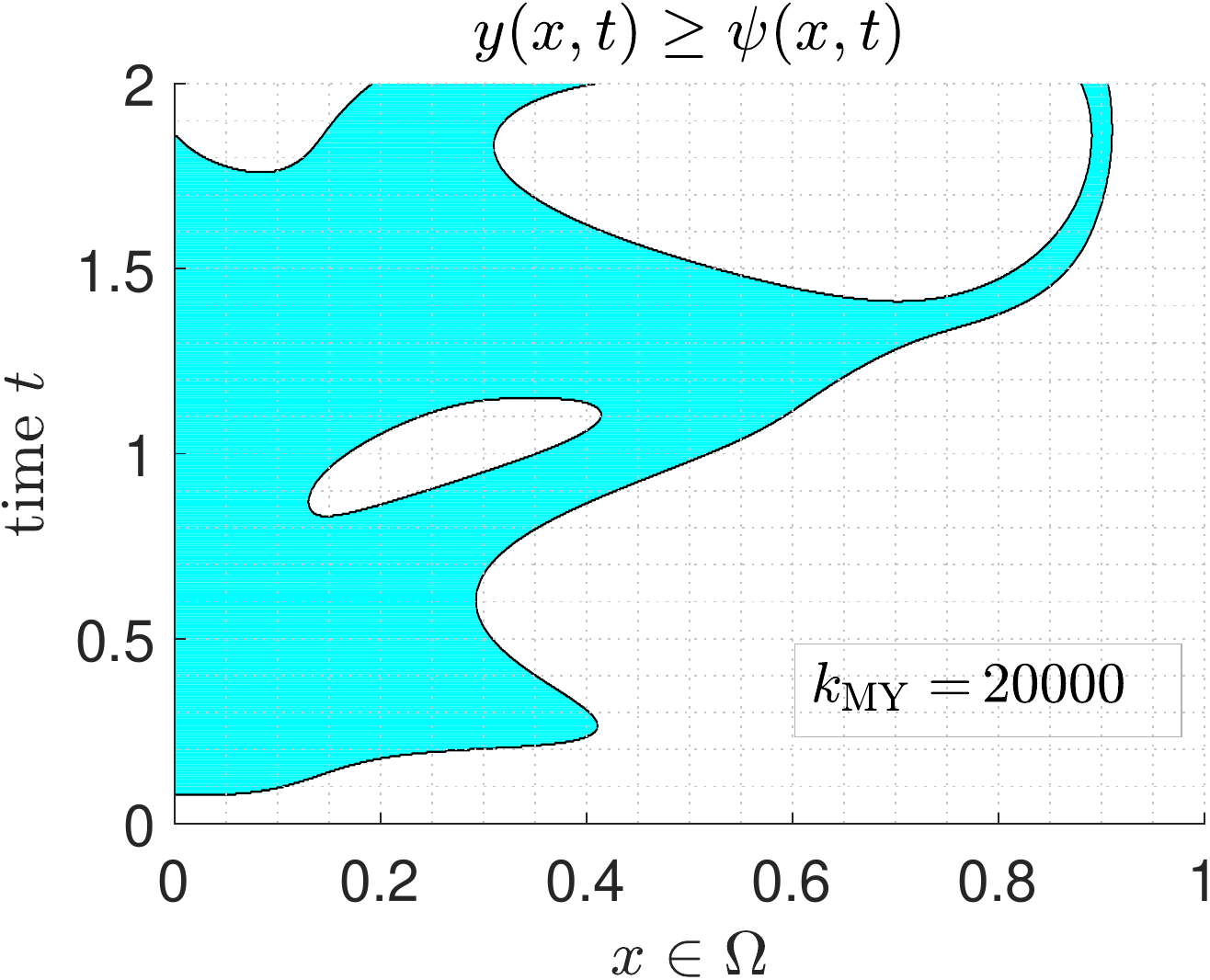}}
\caption{Evolution of obstacle constraint violation set for targeted trajectory}
\label{Fig:SKoSetKSy_lam5M10Feed02T2KTK}
\end{figure}
\begin{figure}[ht]
\centering
\subfigure
{\includegraphics[width=0.45\textwidth,height=0.31\textwidth]{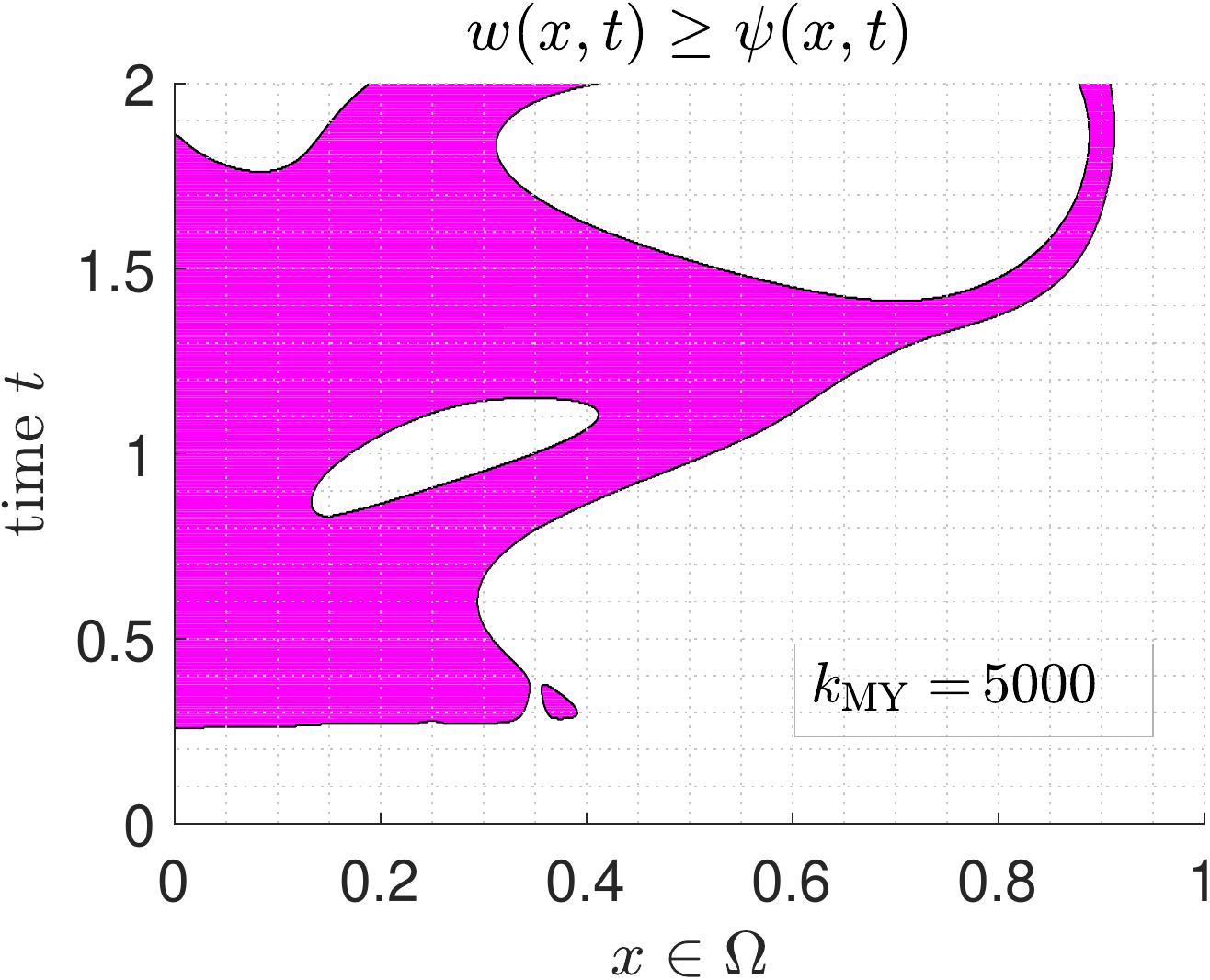}}
\qquad
\subfigure
{\includegraphics[width=0.45\textwidth,height=0.31\textwidth]{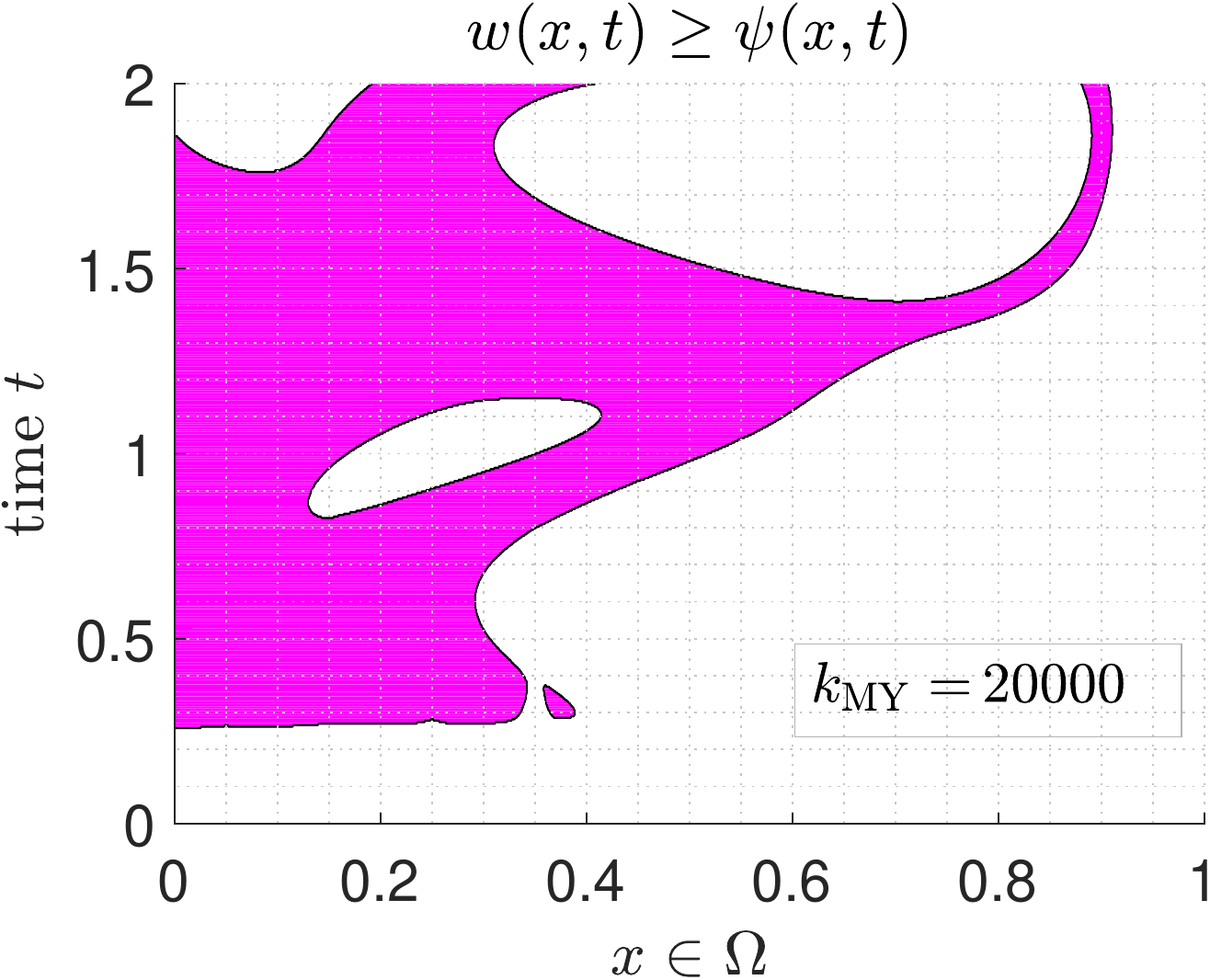}}
\caption{Evolution of obstacle constraint violation set for controlled trajectory}
\label{Fig:SKoSetKSw_lam5M10Feed02T2KTK}
\end{figure}
we can see the evolution of the obstacle constraint violation set.
It is interesting to observe that with the smallest value of~$k_{\rm MY}=5000$ considered,
we can already capture a good picture of the likely limit contact set evolution for the parabolic
variational inequality. The evolution is not simple, for example the number of contact connected
components change with time, this can simply be explained from the fact that the moving obstacle
and its shape  (cf.~Figure~\ref{Fig:TDF_lam4M5Feed04T4-tl} and other time snapshots)
are not simple themselves. 

\section{Numerical simulations for a nonsmooth obstacle}\label{sec:num_nonsmooth}
Note that the stability result for the sequence of $k_{\rm MY}$-Moreau--Yosida approximations hold true for obstacles which live
in~$L^2_{\rm loc}(\Omega\times\bbR_+)$, and in particular we have a weak limit for the pair~$z_k=y_k-w_k$,
Thus, we may ask ourselves if~$y_k$ and~$w_k$ also converge separately and if each of these limits satisfy
(a weaker formulation of) the variational inequality. 
Next, we present results of simulations which suggest that this may be indeed the case
for obstacles in~$\clC^1([0,+\infty),L^2(\Omega))$.
This means that our result can probably be extended to less regular obstacles.
Such extension is an interesting problem for future investigation. Note that, if possible,
such extension is nontrivial and thus will  likely  require a considerably different proof.

The following simulations correspond to the setting as in~\eqref{simul.setting} with the exception that we take a nonsmooth obstacle. Namely, we modify the smooth obstacle in~\eqref{Opsi-smooth}, by changing it to constant functions on the spatial set~$[0,\frac{1}{10}]\bigcup[\frac{8}{10},1]$. More precisely, we take the obstacle
\begin{equation}\notag
\psi(x,t)=\begin{cases}
\frac{31}{10},\;&\mbox{if }x\in[0,\frac{1}{10}];\\
2+\cos(t)+\cos\left(10\pi x(x-1)\bigl(x-\tfrac{1}{4}\cos(5 t)\bigr)\right),\;&\mbox{if }x\in(\frac{1}{10},\frac{8}{10});\\
-\frac{5}{10},\;&\mbox{if }x\in[\frac{8}{10},1].
\end{cases}
\end{equation}

In Figure~\ref{Fig:KoSetDF_lam5M10Feed02T2}
\begin{figure}[ht]
\centering
\subfigure
{\includegraphics[width=0.45\textwidth,height=0.31\textwidth]{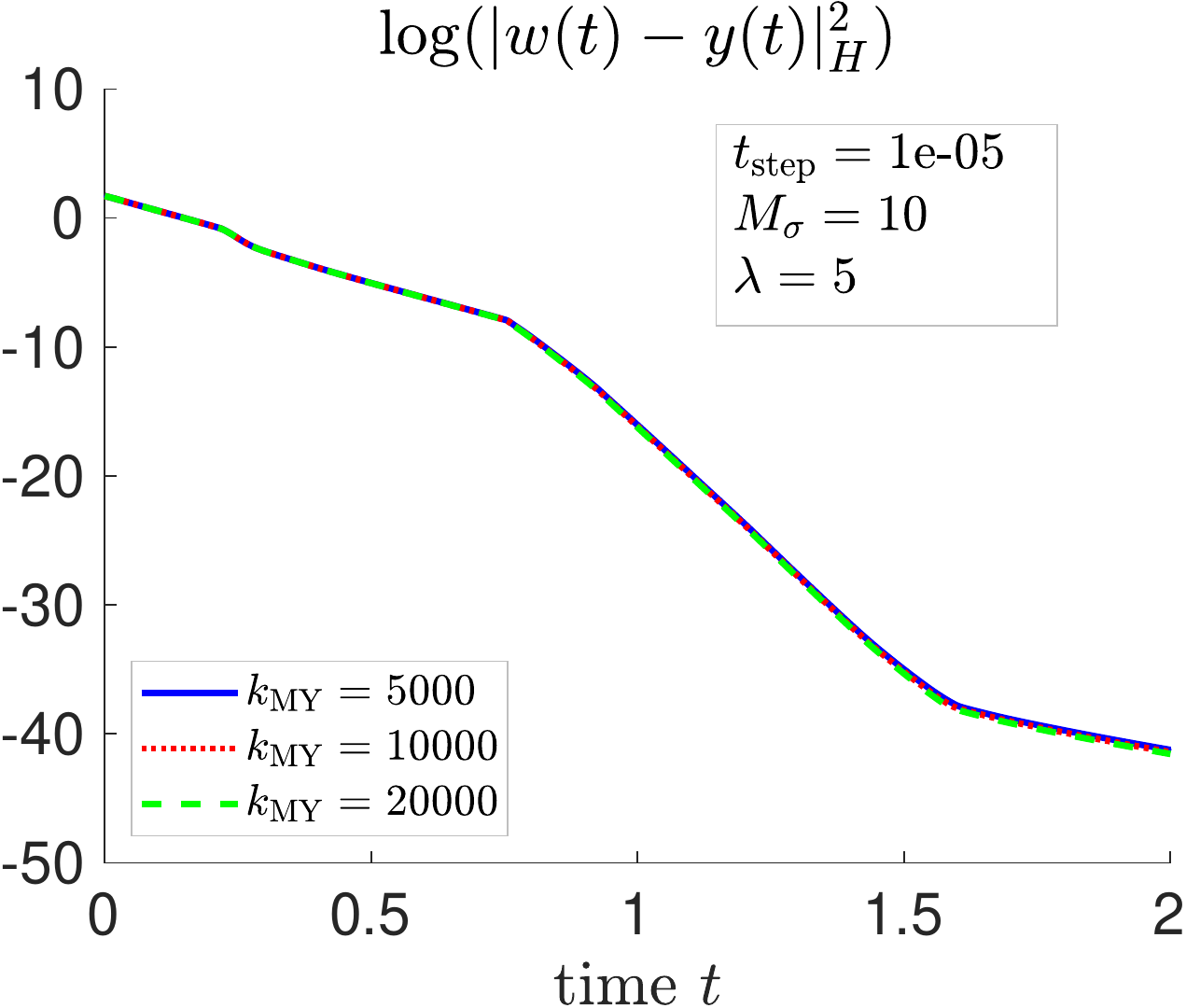}}
\qquad
\subfigure
{\includegraphics[width=0.45\textwidth,height=0.31\textwidth]{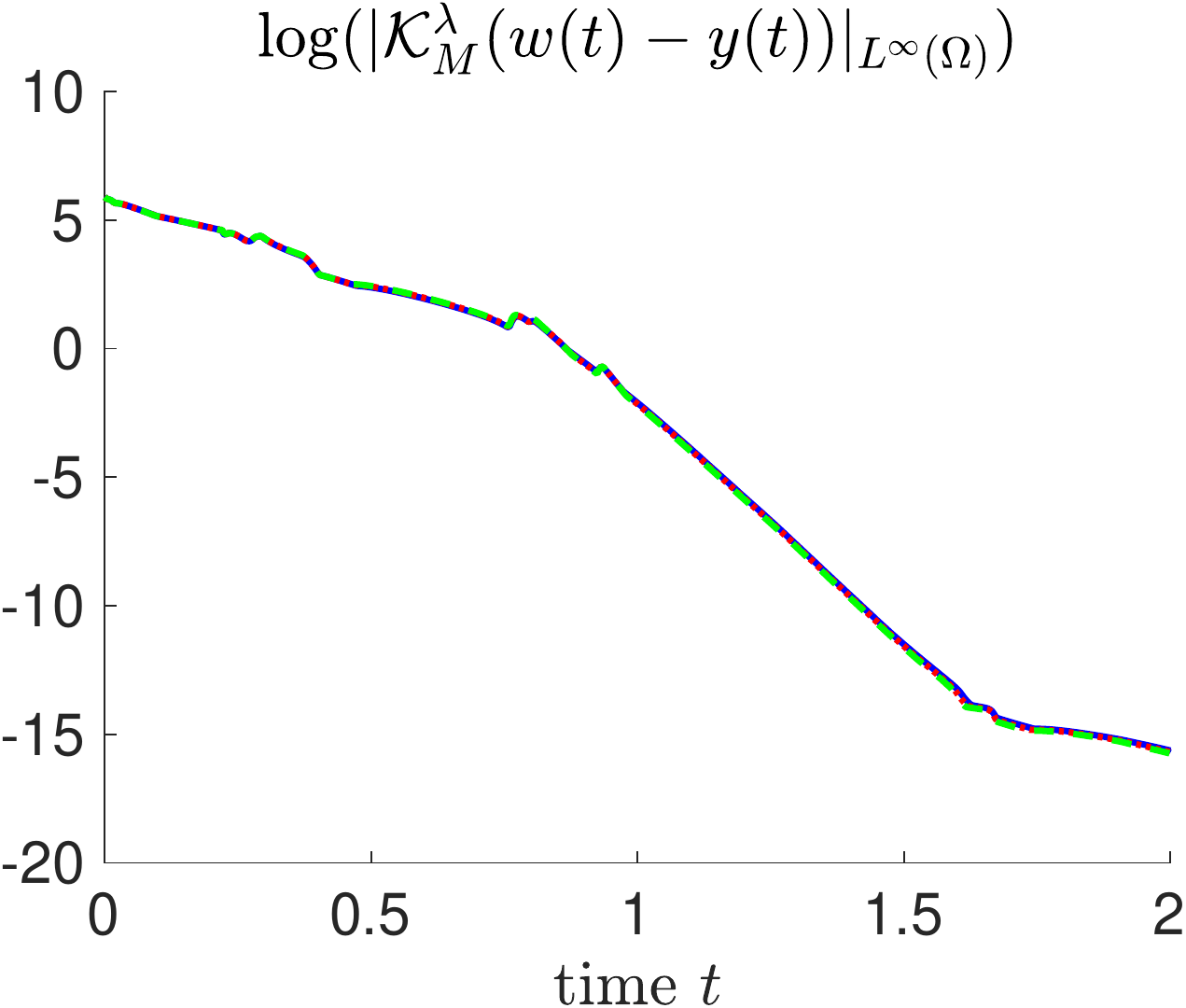}}
\caption{Norms of difference to target and control}
\label{Fig:KoSetDF_lam5M10Feed02T2}
\end{figure}
 we cannot see a considerable difference in the behavior of the
 norm of the difference to target and of the control for the several
 Moreau--Yosida parameters. The same holds for the time snapshots in Figure~\ref{Fig:KoSetTF_lam5M10Feed02T2}.
\begin{figure}[ht]
\centering
\subfigure
{\includegraphics[width=0.45\textwidth,height=0.31\textwidth]{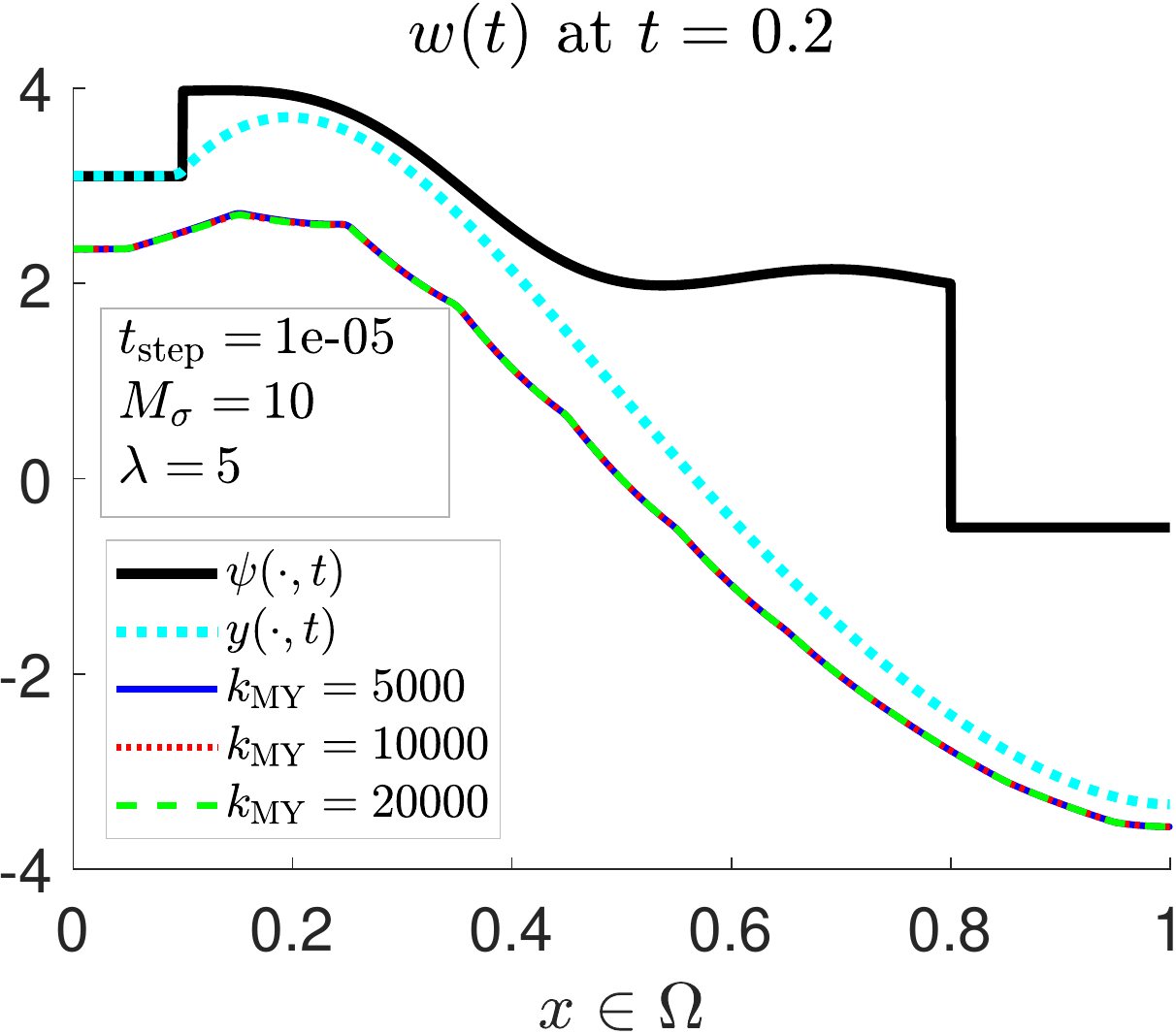}}
\qquad
\subfigure
{\includegraphics[width=0.45\textwidth,height=0.31\textwidth]{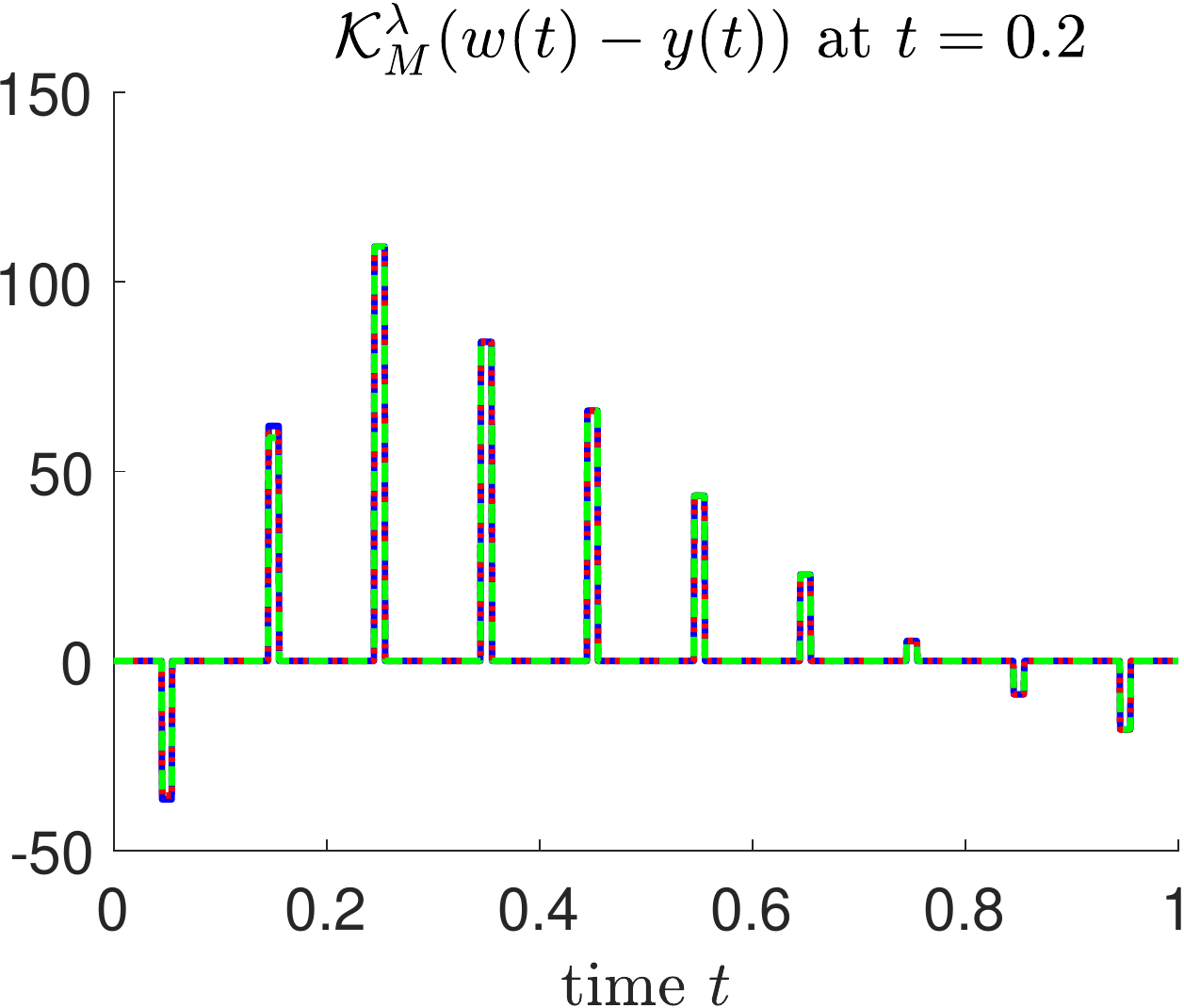}}
\subfigure
{\includegraphics[width=0.45\textwidth,height=0.31\textwidth]{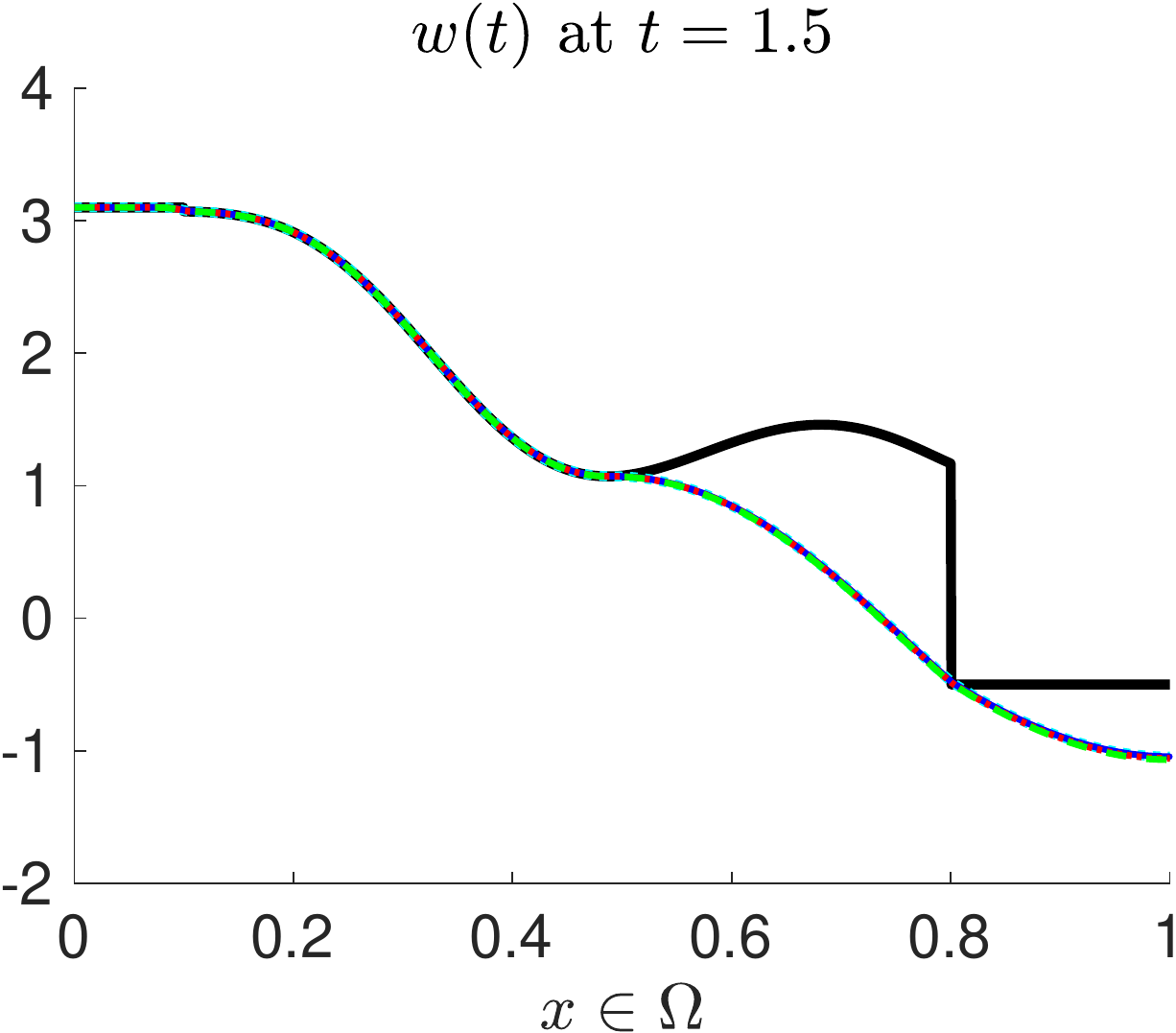}}
\qquad
\subfigure
{\includegraphics[width=0.45\textwidth,height=0.31\textwidth]{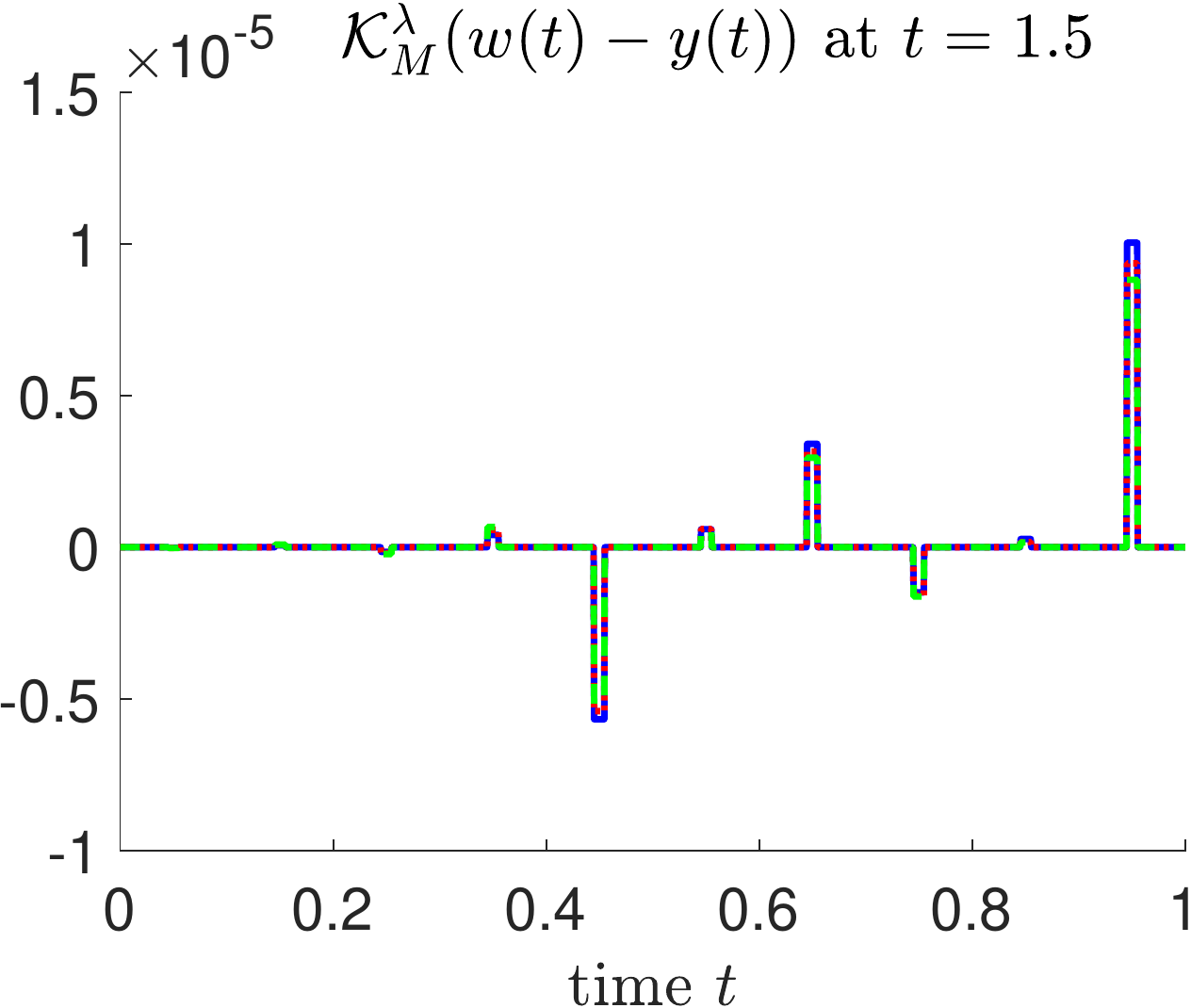}}
\caption{Time snapshots of trajectories and control}
\label{Fig:KoSetTF_lam5M10Feed02T2}
\end{figure}
 So we can conjecture that the considered parameters give us already a good picture of the behavior of the limit difference and control as $k_{\rm MY}$ diverges to~$+\infty$. 

From Figure~\ref{Fig:KoSetOyw_lam5M10Feed02T2} we can conjecture also that the magnitude of the violation of the obstacle
constraint converges to zero   as $k_{\rm MY}\to\infty$.
\begin{figure}[ht]
\centering
\subfigure
{\includegraphics[width=0.45\textwidth,height=0.31\textwidth]{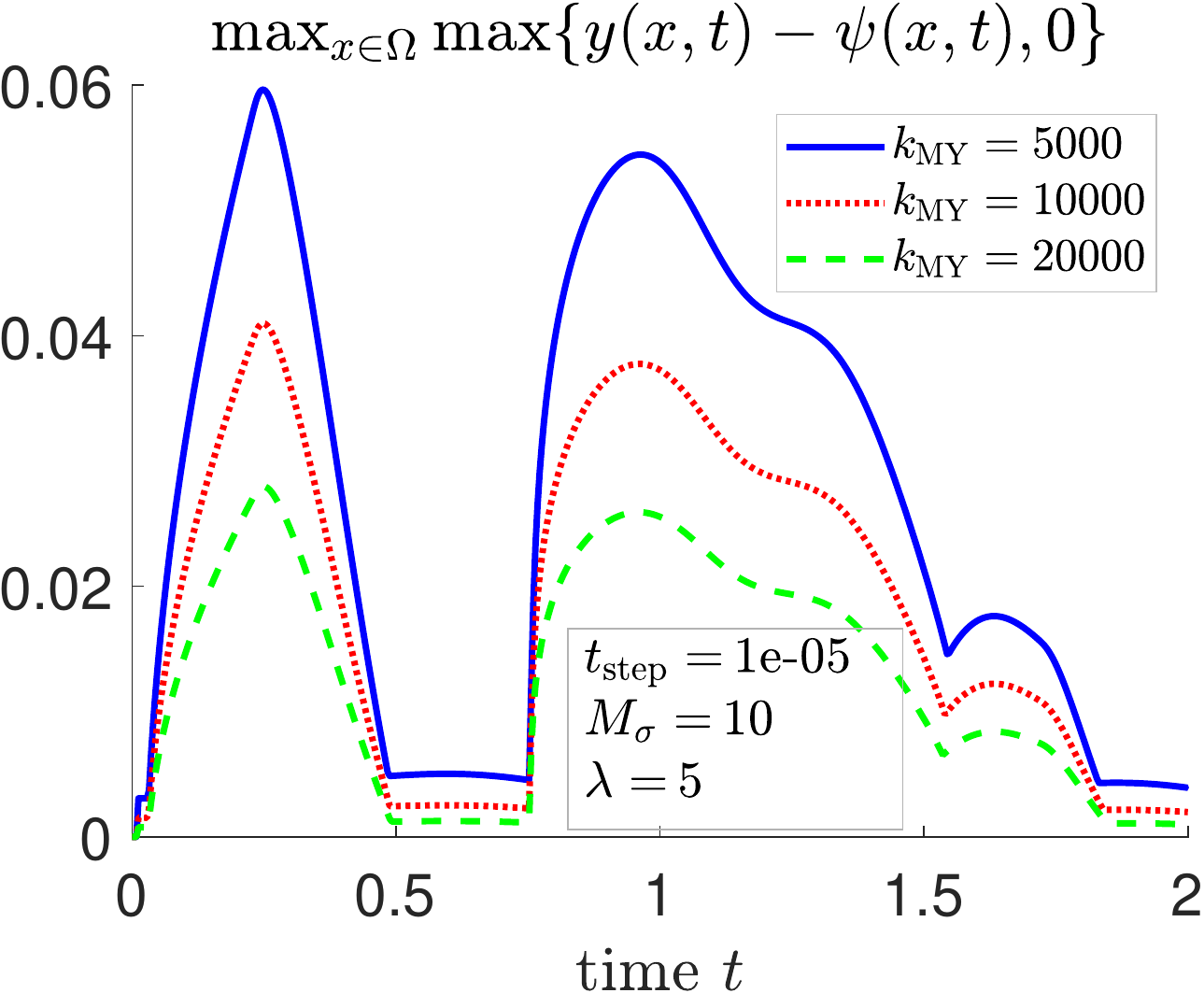}}
\qquad
\subfigure
{\includegraphics[width=0.45\textwidth,height=0.31\textwidth]{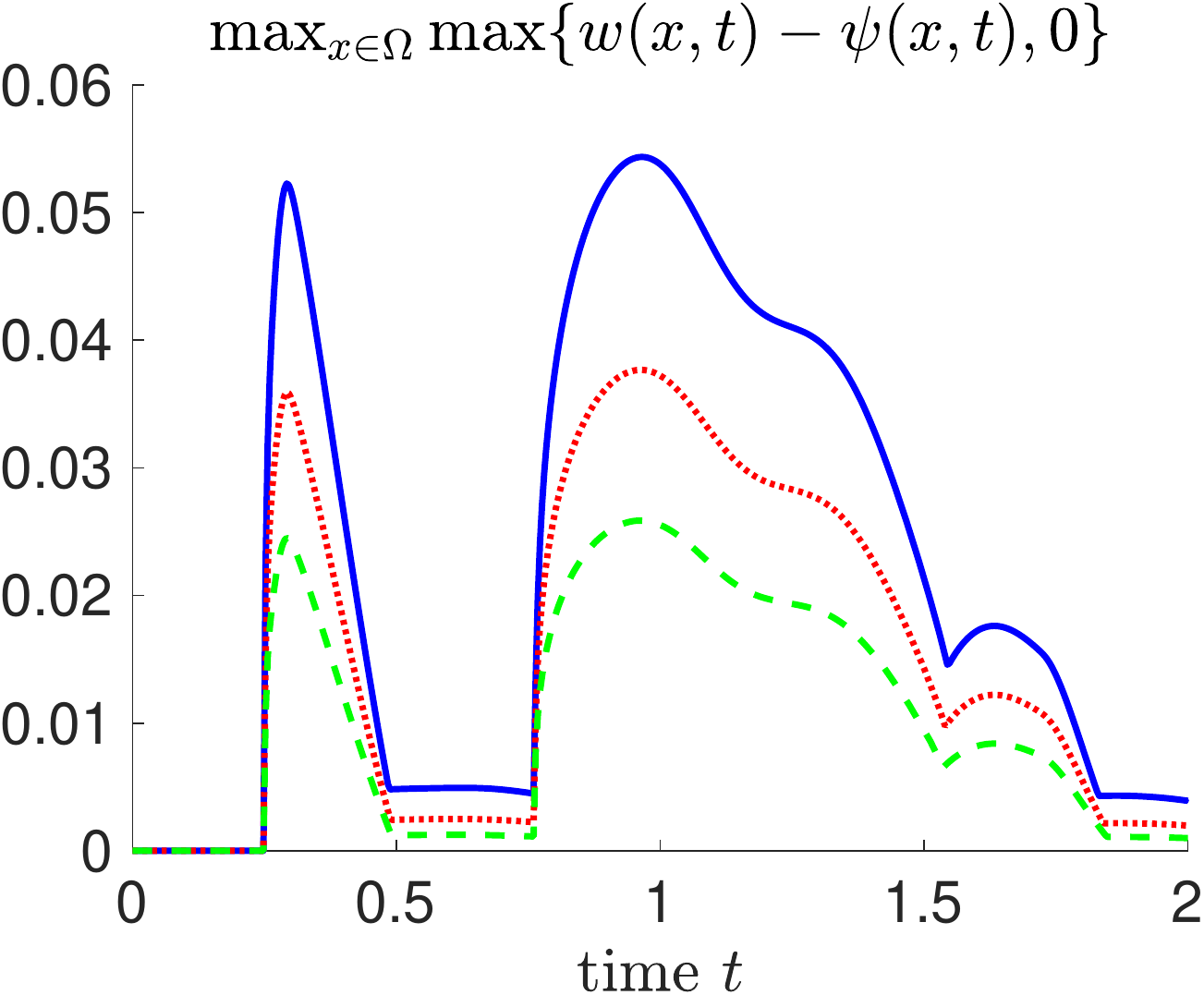}}
\caption{Largest magnitude of obstacle constraint violation}
\label{Fig:KoSetOyw_lam5M10Feed02T2}
\end{figure}

All the above suggest that a variational inequality will be satisfied at a limit.
But, this remains to be proven for nonsmooth obstacles.

Finally, in Figures~\ref{Fig:KoSetKSy_lam5M10Feed02T2KTK} and~\ref{Fig:KoSetKSw_lam5M10Feed02T2KTK}
\begin{figure}[ht]
\centering
\subfigure
{\includegraphics[width=0.45\textwidth,height=0.31\textwidth]{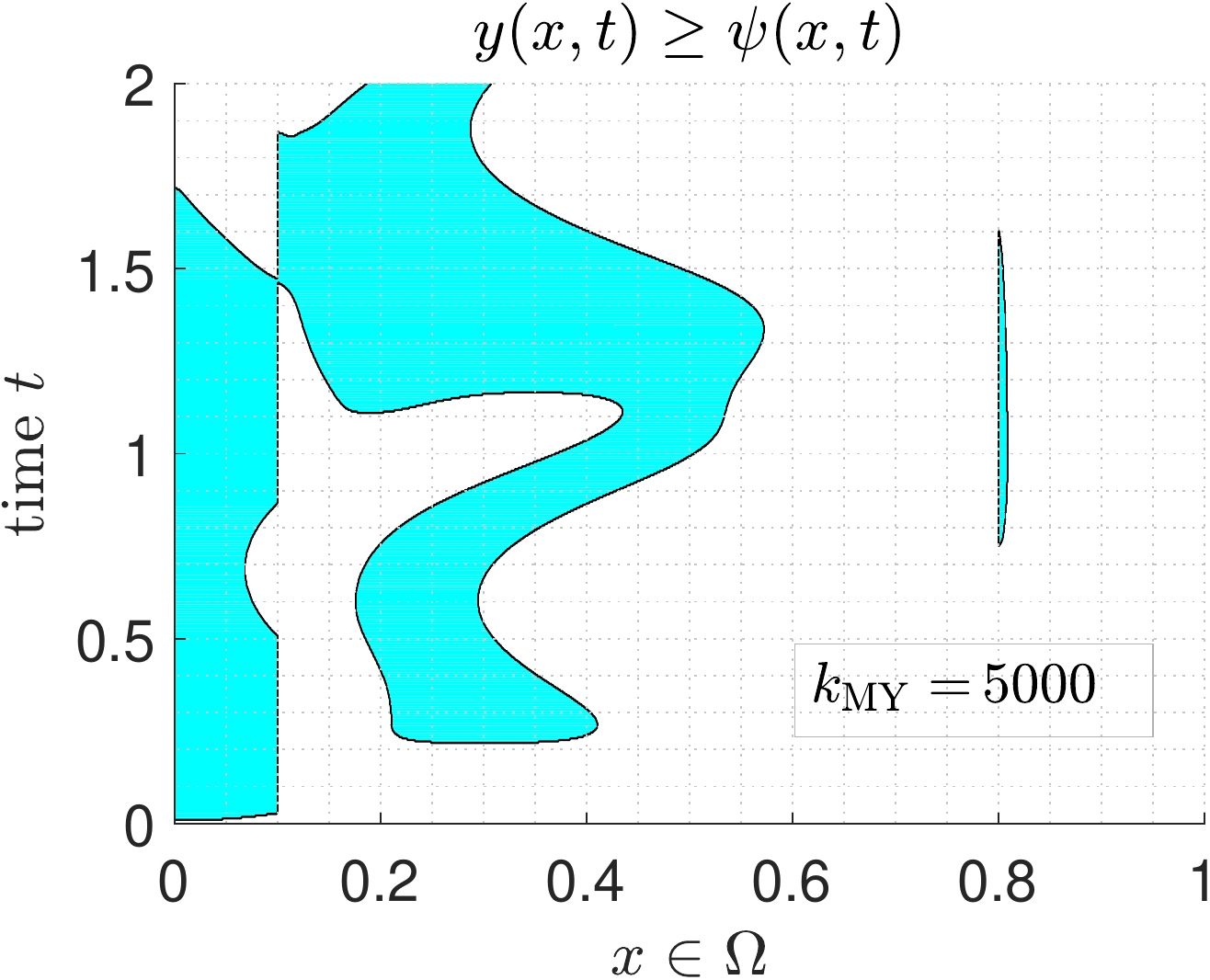}}
\qquad
\subfigure
{\includegraphics[width=0.45\textwidth,height=0.31\textwidth]{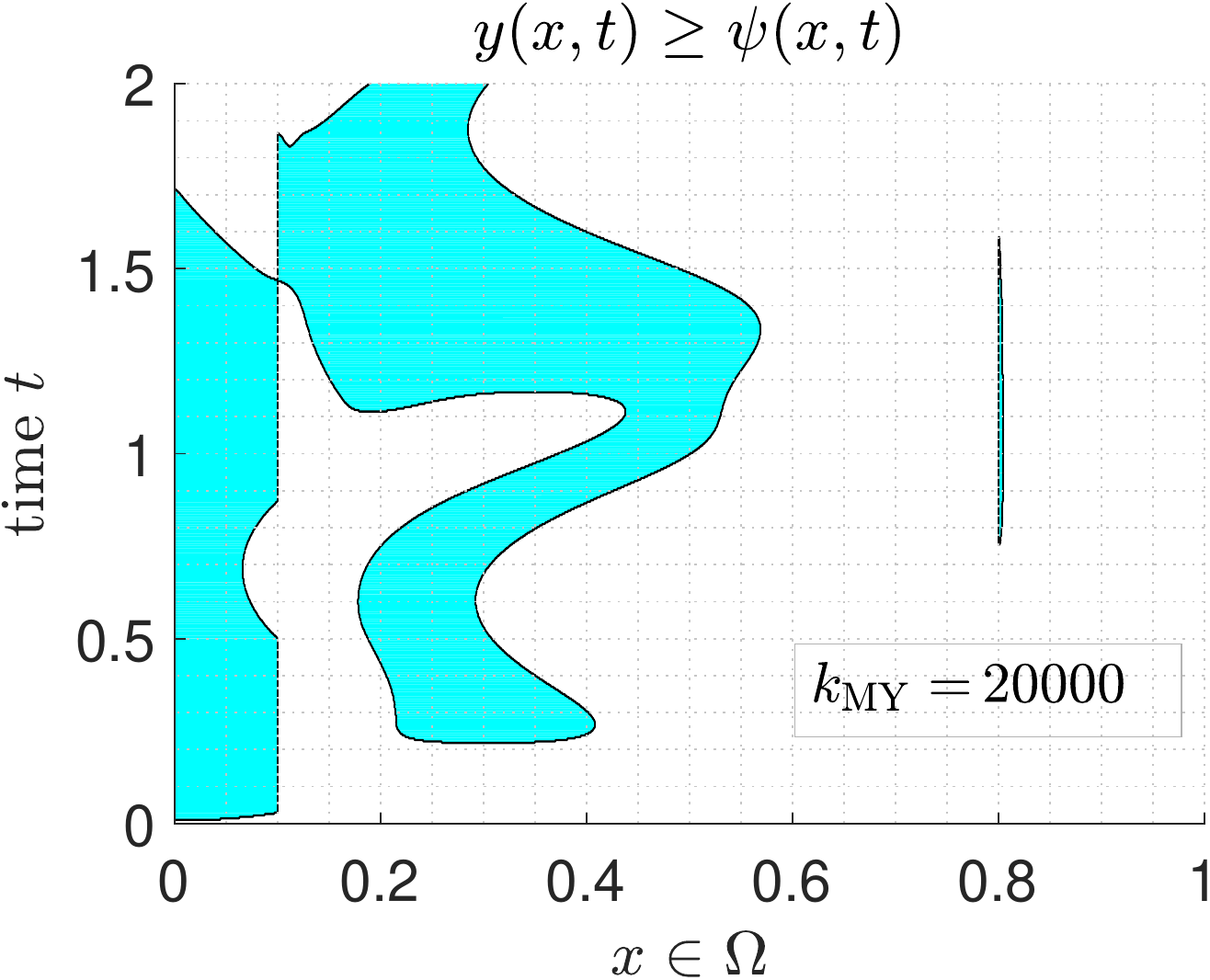}}
\caption{Evolution of obstacle constraint violation set for targeted trajectory}
\label{Fig:KoSetKSy_lam5M10Feed02T2KTK}
\end{figure}
\begin{figure}[ht]
\centering
\subfigure
{\includegraphics[width=0.45\textwidth,height=0.31\textwidth]{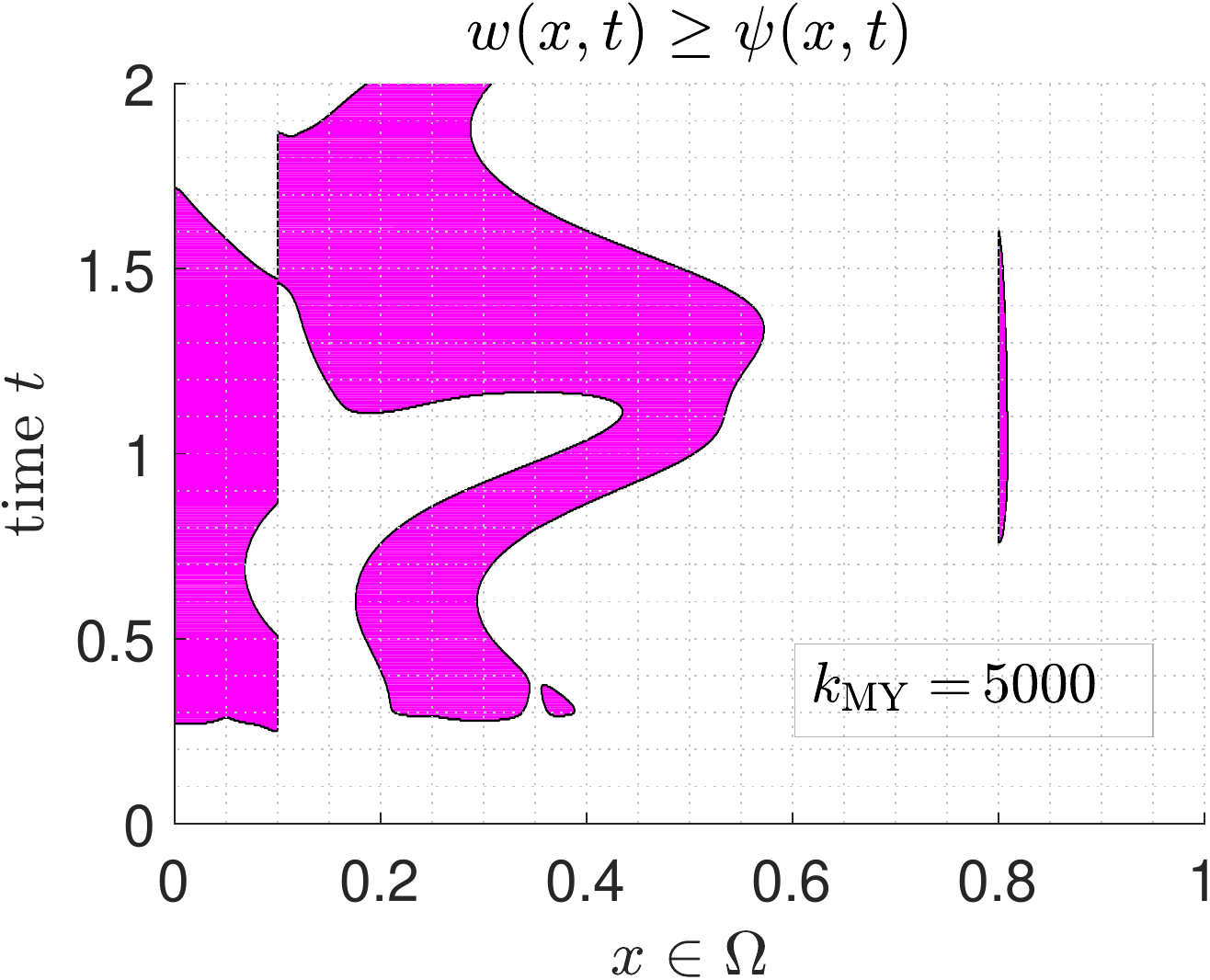}}
\qquad
\subfigure
{\includegraphics[width=0.45\textwidth,height=0.31\textwidth]{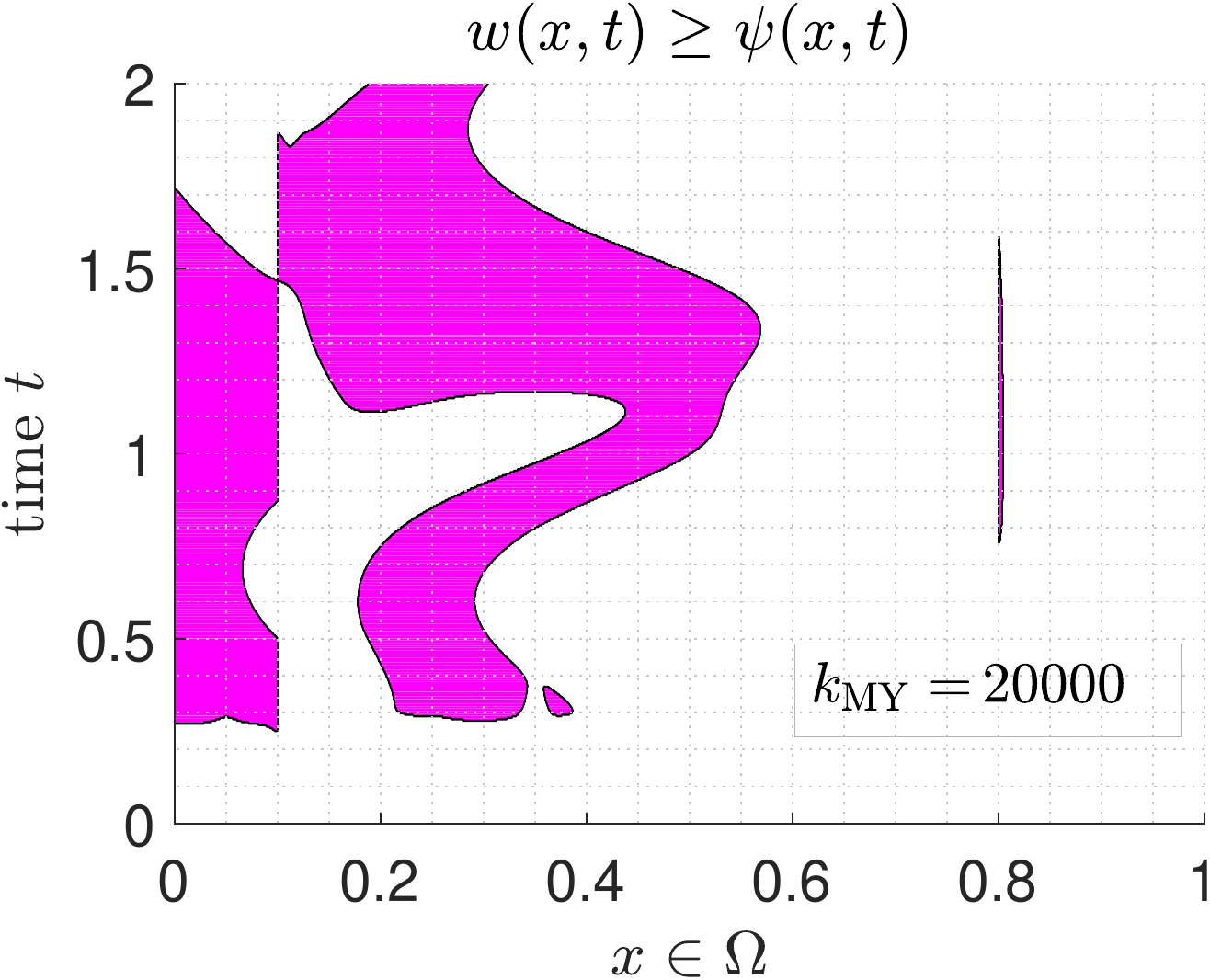}}
\caption{Evolution of obstacle constraint violation set for controlled trajectory}
\label{Fig:KoSetKSw_lam5M10Feed02T2KTK}
\end{figure}
we can see the evolution of the obstacle constraint violation sets.
Again, the smallest value of~$k_{\rm MY}$ provides
us already with good picture of such evolutions.
However, note that by taking the largest value we are able to
``sharpen'' the picture, in particular it confirms that locally the contact
is made at the single (discontinuity) point~$x=0.8$ during a suitable interval of time,
where~$t=1.5$ is included,  as we see in the snapshot in Figure~\ref{Fig:KoSetTF_lam5M10Feed02T2}. We also observe that the discontinuity of the obstacle at the spatial points $x\in\{0.1,0.8\}$ is somehow reflected in Figures~\ref{Fig:KoSetKSy_lam5M10Feed02T2KTK} and~\ref{Fig:KoSetKSw_lam5M10Feed02T2KTK}.


\bibliographystyle{plainurl}
\bibliography{ObstacleParab_Stabil_arx}

\begin{thebibliography}{10}

\bibitem{AzmiRod20}
B.~Azmi and S.~S. Rodrigues.
\newblock Oblique projection local feedback stabilization of nonautonomous
  semilinear damped wave-like equations.
\newblock {\em J. Differential Equations}, 269(7):6163--9192, 2020.
\newblock \href {https://doi.org/10.1016/j.jde.2020.04.033}
  {\path{doi:10.1016/j.jde.2020.04.033}}.

\bibitem{MR726973}
V.~Barbu.
\newblock The time-optimal control problem for parabolic variational
  inequalities.
\newblock {\em Appl. Math. Optim.}, 11(1):1--22, 1984.
\newblock \href {https://doi.org/10.1007/BF01442167}
  {\path{doi:10.1007/BF01442167}}.

\bibitem{Bensoussan1984}
A.~Bensoussan and J.-L. Lions.
\newblock {\em Impulse Control and Quasi-Variational Inequalities}.
\newblock Gauthier-Villars, 1984.

\bibitem{bensoussan2011applications}
A.~Bensoussan and J.-L. Lions.
\newblock {\em Applications of variational inequalities in stochastic control}.
\newblock Elsevier, 2011.

\bibitem{MR2801013}
M.~Boukrouche and D.~A. Tarzia.
\newblock Existence, uniqueness, and convergence of optimal control problems
  associated with parabolic variational inequalities of the second kind.
\newblock {\em Nonlinear Anal. Real World Appl.}, 12(4):2211--2224, 2011.
\newblock \href {https://doi.org/10.1016/j.nonrwa.2011.01.003}
  {\path{doi:10.1016/j.nonrwa.2011.01.003}}.

\bibitem{Brezis71}
H.~Brezis.
\newblock In\'equations variationnelles paraboliques.
\newblock {\em S\'eminaire Jean Leray}, pages 1--10, 1971.
\newblock talk:7.
\newblock URL: \url{http://www.numdam.org/item/SJL_1971____A7_0}.

\bibitem{MR2346390}
Q.~Chen, D.~Chu, and R.~C.~E. Tan.
\newblock Bilateral obstacle control problem of parabolic variational
  inequalities.
\newblock {\em SIAM J. Control Optim.}, 46(4):1518--1537, 2007.
\newblock \href {https://doi.org/10.1137/050638047}
  {\path{doi:10.1137/050638047}}.

\bibitem{MR3899156}
C.~Christof.
\newblock Sensitivity analysis and optimal control of obstacle-type evolution
  variational inequalities.
\newblock {\em SIAM J. Control Optim.}, 57(1):192--218, 2019.
\newblock \href {https://doi.org/10.1137/18M1183662}
  {\path{doi:10.1137/18M1183662}}.

\bibitem{GilbargTrudinger98}
D.~Gilbarg and N.S. Trudinger.
\newblock {\em Elliptic Partial Differential Equations of Second Order}.
\newblock Number 224 in Grundlehren Math. Wiss. Springer-Verlag, 1998.
\newblock \href {https://doi.org/10.1007/978-3-642-61798-0}
  {\path{doi:10.1007/978-3-642-61798-0}}.

\bibitem{MR635927}
Roland Glowinski, Jacques-Louis Lions, and Raymond Tr{{\'e}}moli{{\`e}}res.
\newblock {\em Numerical analysis of variational inequalities}, volume~8 of
  {\em Studies in Mathematics and its Applications}.
\newblock North-Holland Publishing Co., Amsterdam-New York, 1981.
\newblock Translated from the French.

\bibitem{Grisv85}
P.~Grisvard.
\newblock {\em Elliptic Problems in Nonsmooth Domains}.
\newblock Pitman Advanced Publishing Program, 1985.
\newblock \href {https://doi.org/10.1137/1.9781611972030}
  {\path{doi:10.1137/1.9781611972030}}.

\bibitem{MR2228961}
K.-H. Hoffmann, M.~Kubo, and N.~Yamazaki.
\newblock Optimal control problems for elliptic-parabolic variational
  inequalities with time-dependent constraints.
\newblock {\em Numer. Funct. Anal. Optim.}, 27(3-4):329--356, 2006.
\newblock \href {https://doi.org/10.1080/01630560600686116}
  {\path{doi:10.1080/01630560600686116}}.

\bibitem{MR2609034}
K.~Ito and K.~Kunisch.
\newblock Optimal control of parabolic variational inequalities.
\newblock {\em J. Math. Pures Appl. (9)}, 93(4):329--360, 2010.
\newblock \href {https://doi.org/10.1016/j.matpur.2009.10.005}
  {\path{doi:10.1016/j.matpur.2009.10.005}}.

\bibitem{Katopodes19}
N.~D. Katopodes.
\newblock {\em Free-Surface Flow: Computational Methods}.
\newblock Elsevier Butterworth-Heinemann Publications, 2019.

\bibitem{KunRod19-cocv}
K.~Kunisch and S.~S. Rodrigues.
\newblock Explicit exponential stabilization of nonautonomous linear
  parabolic-like systems by a finite number of internal actuators.
\newblock {\em ESAIM Control Optim. Calc. Var.}, 25, 2019.
\newblock 67.
\newblock \href {https://doi.org/10.1051/cocv/2018054}
  {\path{doi:10.1051/cocv/2018054}}.

\bibitem{KunRod19-dcds}
K.~Kunisch and S.~S. Rodrigues.
\newblock Oblique projection based stabilizing feedback for nonautonomous
  coupled parabolic-{ode} systems.
\newblock {\em Discrete Contin. Dyn. Syst.}, 39(11):6355--6389, 2019.
\newblock \href {https://doi.org/10.3934/dcds.2019276}
  {\path{doi:10.3934/dcds.2019276}}.

\bibitem{Lions69}
J.-L. Lions.
\newblock {\em Quelques M\'ethodes de R\'esolution des Probl\`emes aux Limites
  Non Lin\'eaires}.
\newblock Dunod et Gauthier--Villars, Paris, 1969.

\bibitem{LioMag72-I}
J.-L. Lions and E.~Magenes.
\newblock {\em Non-Homogeneous Boundary Value Problems and Applications,
  vol.~I}.
\newblock Number 181 in Die Grundlehren Math. Wiss. Einzeldarstellungen.
  Springer-Verlag, 1972.
\newblock \href {https://doi.org/10.1007/978-3-642-65161-8}
  {\path{doi:10.1007/978-3-642-65161-8}}.

\bibitem{MR2155305}
V.~Maksimov.
\newblock Feedback robust control for a parabolic variational inequality.
\newblock In {\em System modeling and optimization}, volume 166 of {\em IFIP
  Int. Fed. Inf. Process.}, pages 123--134. Kluwer Acad. Publ., Boston, MA,
  2005.
\newblock \href {https://doi.org/10.1007/0-387-23467-5_7}
  {\path{doi:10.1007/0-387-23467-5_7}}.

\bibitem{PhanRod18-mcss}
D.~Phan and S.~S. Rodrigues.
\newblock Stabilization to trajectories for parabolic equations.
\newblock {\em Math. Control Signals Syst.}, 30(2), 2018.
\newblock 11.
\newblock \href {https://doi.org/10.1007/s00498-018-0218-0}
  {\path{doi:10.1007/s00498-018-0218-0}}.

\bibitem{MR1816854}
C.~Popa.
\newblock Feedback laws for the optimal control of parabolic variational
  inequalities.
\newblock In {\em Shape optimization and optimal design ({C}ambridge, 1999)},
  volume 216 of {\em Lecture Notes in Pure and Appl. Math.}, pages 371--380.
  Dekker, New York, 2001.

\bibitem{Rod14-na}
S.~S. Rodrigues.
\newblock Local exact boundary controllability of {3D} {N}avier--{S}tokes
  equations.
\newblock {\em Nonlinear Anal.}, 95:175--190, 2014.
\newblock \href {https://doi.org/10.1016/j.na.2013.09.003}
  {\path{doi:10.1016/j.na.2013.09.003}}.

\bibitem{Rod20-eect}
S.~S. Rodrigues.
\newblock Semiglobal exponential stabilization of nonautonomous semilinear
  parabolic-like systems.
\newblock {\em Evol. Equ. Control Theory}, 9(3):635--672, 2020.
\newblock \href {https://doi.org/10.3934/eect.2020027}
  {\path{doi:10.3934/eect.2020027}}.

\bibitem{Rod21-sicon}
S.~S. Rodrigues.
\newblock Oblique projection exponential dynamical observer for nonautonomous
  linear parabolic-like equations.
\newblock {\em SIAM J. Control Optim.}, 59(1):464--488, 2021.
\newblock \href {https://doi.org/10.1137/19M1278934}
  {\path{doi:10.1137/19M1278934}}.

\bibitem{RodSturm20}
S.~S. Rodrigues and K.~Sturm.
\newblock On the explicit feedback stabilisation of one-dimensional linear
  nonautonomous parabolic equations via oblique projections.
\newblock {\em IMA J. Math. Control Inform.}, 37(1):175--207, 2020.
\newblock \href {https://doi.org/10.1093/imamci/dny045}
  {\path{doi:10.1093/imamci/dny045}}.

\bibitem{Rudin87}
W.~Rudin.
\newblock {\em Real and Complex Analysis}.
\newblock McGraw-Hill, 3rd edition, 1987.

\bibitem{Simon87}
J.~Simon.
\newblock Compact sets in the space~${L^p(0,T;B)}$.
\newblock {\em Ann. Mat. Pura Appl. (4)}, 146:65--96, 1987.
\newblock \href {https://doi.org/10.1007/BF01762360}
  {\path{doi:10.1007/BF01762360}}.

\bibitem{Stampacchia63}
G.~Stampacchia.
\newblock {\'E}quations elliptiques du second ordre \`a coefficients
  discontinus.
\newblock {\em S\'eminaire Jean Leray}, (3):1--77, 1963-1964.
\newblock URL: \url{http://www.numdam.org/item/SJL_1963-1964___3_1_0}.

\bibitem{Temam01}
R.~Temam.
\newblock {\em {N}avier--{S}tokes Equations: Theory and Numerical Analysis}.
\newblock AMS Chelsea Publishing, Providence, RI, {reprint of the 1984}
  edition, 2001.

\bibitem{MR3438403}
G.~Wachsmuth.
\newblock Optimal control of quasistatic plasticity with linear kinematic
  hardening {III}: {O}ptimality conditions.
\newblock {\em Z. Anal. Anwend.}, 35(1):81--118, 2016.
\newblock \href {https://doi.org/10.4171/ZAA/1556}
  {\path{doi:10.4171/ZAA/1556}}.

\bibitem{MR1866305}
G.~Wang.
\newblock Optimal control problem for parabolic variational inequalities.
\newblock {\em Acta Math. Sci. Ser. B (Engl. Ed.)}, 21(4):509--525, 2001.
\newblock \href {https://doi.org/10.1016/S0252-9602(17)30440-X}
  {\path{doi:10.1016/S0252-9602(17)30440-X}}.

\end{thebibliography}
\end{document}